\documentclass[11pt]{article}

\usepackage[margin=1in]{geometry} 
\usepackage{graphicx,subcaption}

\usepackage{graphicx}
\usepackage{caption}
\usepackage{subcaption}
\usepackage{indentfirst}
\usepackage{amsfonts}
\usepackage{amsmath}
\usepackage{amssymb}
\usepackage{amsthm}

\usepackage{amsbsy}
\usepackage{color}
\usepackage{xcolor}
\usepackage{multirow}
\usepackage{boldline}

\theoremstyle{plain}
\newtheorem{theorem}{Theorem}[section]
\newtheorem{dfn}{Definition}
\newtheorem{lemma}[theorem]{Lemma}
\newtheorem{claim}{Claim}[theorem]
\newtheorem*{claim-nonumber}{Claim}

\newtheorem{proposition}[theorem]{Proposition}
\newtheorem*{remark}{Remark}
\newtheorem{question}{Question}[section]

\newtheorem{corollary}[theorem]{Corollary}

\def\deg{{\rm deg}}
\def\C{{\mathcal C}}
\def\P{{\mathcal P}}
\def\G{{\mathcal G}}
\def\M{{\mathcal M}}
\def\D{{\mathcal D}}
\def\R{{\mathcal R}}
\def\bG{{\mathbb G}}

\def\F{{\mathcal F}}

\def\thredeg{{p_{\G}^{\D_r}}}

\title{Phase Transition of Degeneracy in Minor-Closed Families}

\author{Chun-Hung Liu\thanks{Department of Mathematics, Texas A\&M University. Email: chliu@math.tamu.edu. Partially supported by NSF under awards DMS-1664593, DMS-1929851, DMS-1954054 and CAREER award DMS-2144042.} \ \ \ \ \ \ \ \ \ \ 
	 Fan Wei\thanks{Department of Mathematics, Princeton University, Princeton, NJ 08544. Email: fanw@princeton.edu. Research supported by NSF Award DMS-1953958.}}

\begin{document}

\maketitle

\begin{abstract}
Given an infinite family $\G$ of graphs and a monotone property $\P$, an (upper) threshold for $\G$ and $\P$ is a ``fastest growing" function $p: \mathbb{N} \to [0,1]$ such that $\lim_{n \to \infty} \Pr(G_n(p(n)) \in \P)= 1$ for any sequence $(G_n)_{n \in \mathbb{N}}$ over $\G$ with $\lim_{n \to \infty}\lvert V(G_n) \rvert = \infty$, where $G_n(p(n))$ is the random subgraph of $G_n$ such that each edge remains independently with probability $p(n)$.

In this paper we study the upper threshold for the family of $H$-minor free graphs and the property of being $(r-1)$-degenerate and apply it to study the thresholds for general minor-closed families and the properties for being $r$-choosable and $r$-colorable.
Even a constant factor approximation for the upper threshold for all pairs $(r,H)$ is expected to be challenging by its close connection to a major open question in extremal graph theory.
We determine asymptotically the thresholds (up to a constant factor) for being $(r-1)$-degenerate  (and $r$-choosable, respectively) for a large class of pairs $(r,H)$, including all graphs $H$ of minimum degree at least $r$ and all graphs $H$ with no vertex-cover of size at most $r$, and provide lower bounds for the rest of the pairs of $(r,H)$.
\end{abstract}


\section{Introduction}
Given a graph $G$ and a real number $0 \leq p \leq 1$, let $G(p)$ be the random subgraph of $G$ where each edge remains independently with probability $p$. 
Note that if $G$ is an $n$-vertex complete graph, this is the well-studied Erd\H{o}s-R{\'e}nyi model $\mathbb{G}(n,p)$.

Studying the random perturbation $G(p)$ of a graph $G$ is of both theoretical and practical interests because most instances in the real world are subject to random noises.
It is therefore valuable to study robustness of a property or a specific algorithm. 
For example, Spielman and Teng \cite{ST, ST2} introduced smoothed analysis and studied a continuous version of random perturbations (e.g., noises are Gaussian distributions).
They \cite{ST} showed that the simplex method runs in polynomial time in expectation even though the worst-case scenario runs in exponential time, explaining rigorously why the simplex method is fast in practice. 
Since then, smoothed analysis for other hard graph problems has been studied (e.g. \cite{ABPW,ER}). 
However, the techniques in \cite{ABPW,ER} used for continuous random models stop working in the random model $G(p)$, where each edge follows a discrete distribution. 

Contrary to smoothed analysis, Bennett, Reichman and Shinkar \cite{BRS} showed that for some NP-hard problems related to coloring and independence number, the worst-case instances essentially remain hard under random perturbations. 

Coloring problems on planar graphs have been extensively studied.
It is trivial to determine whether a planar graph is 4-colorable\footnote{In this paper, a graph is {\it $k$-colorable} for some integer $k$ if its vertices can be colored with $k$ colors such that vertices with the same color are not adjacent.} by the Four Color Theorem, though it is NP-hard to determine whether a planar graph is 3-colorable or not \cite{gjs}. 
It is natural to consider the chromatic number of random perturbations of planar graphs.
In particular, the following question raised by Daniel Reichman (via private communication) remains open.

\begin{question}\label{eq1}
What are the values of $p$ such that for each planar graph $G$, the random subgraph $G(p)$ is $3$-colorable  with high probability?
\end{question}

In this paper we systematically study a generalization of Question \ref{eq1} by investigating the (crude) thresholds for stronger properties (degeneracy and choosability) on more general graph classes (minor-closed families), which is related to an open problem in extremal graph theory. 
One simple corollary of our main results is the following partial answer of Question \ref{eq1}.

\begin{corollary} \label{cor_planar_first}
For any sequence $(G_n)_{n \in {\mathbb N}}$ of planar graphs with $|V(G_i)|=i$ for each $i \in {\mathbb N}$, if $p=o(n^{-1/5})$, then $G_n(p)$ is 3-colorable a.a.s.\footnote{Given a sequence of events $(E_n)_{n \in  \mathbb{N}}$ in a probability space, we say $E_n$ happens {\it asymptotically almost surely} (or {\it a.a.s.}\ in short) if $\lim_{n \to \infty} \Pr(E_n) = 1$.}
Moreover, there exists a sequence $(G_n)_{n \in {\mathbb N}}$ of planar graphs with $|V(G_i)|=i$ for each $i \in {\mathbb N}$ such that if $p=\omega(n^{-1/6})$, then $G_n(p)$ is not 3-colorable a.a.s.
\end{corollary}

Corollary \ref{cor_planar_first} can be restated as ``the threshold probability for being 3-colorable and for the class of planar graphs is $\Omega(n^{-1/5})$ and $O(n^{-1/6})$''.
The formal definition of the thresholds is included in Definition \ref{def upper threshold}. 

Corollary \ref{cor_planar_first} follows from a very special case of our general results (Theorems \ref{thm: main critical exponent 1} and \ref{thm:critical exponent lower bound}).
More corollaries for other extensively studied minor-closed families\footnote{One example included in Corollary \ref{cor:many_intro} is the class of graphs with bounded Colin de Verdi\`{e}re parameter. This parameter, denoted by $\mu(G)$, is defined to be the largest corank of certain matrices associated with a graph $G$. Its formal definition is long and is omitted in this paper because we do not need the formal definition to derive results from our main theorem. It is known that $\mu(H) \leq \mu(G)$ if $H$ is a minor of $G$ \cite{c}, so the class of graphs whose $\mu$ is at most a fixed constant $k$ is a minor-closed family. This parameter can capture certain topological properties of graphs. It is known that $\mu(G) \leq 1$ if and only if $G$ is a disjoint union of paths \cite{c}; $\mu(G) \leq 2$ if and only if $G$ is outerplanar \cite{c}; $\mu(G) \leq 3$ if and only if $G$ is planar \cite{c}; $\mu(G) \leq 4$ if and only if $G$ is linklessly embeddable \cite{ls,rst_survey,rst}.} can be similarly derived from those main results.
We include some examples in Corollary \ref{cor:many_intro}.

\begin{corollary} \label{cor:many_intro}
All results in Table \ref{table:cor} hold. 

	\begin{table}
	\begin{center}
\begin{tabular}{|c||c||c||c||c|}
	\hline
	\multirow{2}{9em}{Graph class} & \multirow{2}{6em}{Range for $r$, where $r \in {\mathbb N}$} & \multicolumn{3}{|c|}{Properties} \\
	\cline{3-5}
	& & $(r-1)$-degenerate & $r$-choosable & $r$-colorable \\
	\clineB{1-5}{4} 
	& & & & \\
	\multirow{3}{9em}{Graphs embeddable in a surface $\Sigma$ with Euler genus $g$, where $\Sigma \neq {\mathbb S}^2$} & $r \geq \lfloor \frac{7+\sqrt{1+24g}}{2} \rfloor$ & $\Theta(1)$ & $\Theta(1)$ & $\Theta(1)$ \\
	\cline{2-5}	
	& & & & \\
	& $4 \leq r<\lfloor \frac{7+\sqrt{1+24g}}{2} \rfloor$, & $\Omega(n^{-\frac{2}{(r+1)r-2}})$  & $\Omega(n^{-\frac{2}{(r+1)r-2}})$ & $\Omega(n^{-\frac{2}{(r+1)r-2}})$ \\
	& and if $\Sigma$ is the Klein &  &  &  \\
	& bottle, then $r \neq 6$ & & & \\
	\cline{2-5}
	& \multirow{2}{9em}{$r=6$, and $\Sigma$ is the Klein bottle} & $\Omega(n^{-1/20})$ & $\Theta(1)$ & $\Theta(1)$ \\
	& &  & & \\
	\cline{2-5}
	& $r=3$ & $\Theta(n^{-1/5})$ & $\Theta(n^{-1/5})$ & $ \Omega(n^{-1/5})$ \\
	& & & & $O(n^{-1/6})$ \\
	\cline{2-5}
	& $r=2$ & $\Theta(n^{-1/2})$ & $\Theta(n^{-1/2})$ & $\Omega(n^{-1/2})$ \\
	& & & & $O(n^{-1/3})$ \\
	\hline
	\multirow{5}{9em}{Graphs with Colin de Verdi\`{e}re parameter $\mu \leq k$, for some fixed integer $k$} & $r=k \geq 2$ & $\Theta(n^{-\frac{1}{2k-1}})$ & $\Theta(n^{-\frac{1}{2k-1}})$ & $\Omega(n^{-\frac{1}{2k-1}})$ \\
	& & & & $O(n^{-\frac{2}{(k+1)k}})$ \\
	\cline{2-5}
	& $2 \leq r \leq k-1$ & $\Theta(n^{-\frac{1}{r}})$ & $\Theta(n^{-\frac{1}{r}}) $ & $\Omega(n^{-\frac{1}{r}})$ \\
	& & & & $O(n^{-\frac{2}{(r+1)r}})$ \\
	& & & & \\
	\hline
	\multirow{3}{9em}{Linkless embeddable graphs (i.e.\ $\mu \leq 4$)} & $r=5$ & $\Omega(n^{-1/12})$ & $\Omega(n^{-1/12})$ & $\Theta(1)$ \\
	& & $O(n^{-1/42})$ & $O(n^{-1/1470})$ & \\
	\cline{2-5}
	& $r=4$ & $\Theta(n^{-1/7})$ & $\Theta(n^{-1/7}) $ & $\Omega(n^{-1/7})$ \\
	& & & & $O(n^{-1/10})$ \\
	\cline{2-5}
	& $r=3$ & $\Theta(n^{-1/3})$ & $\Theta(n^{-1/3}) $ & $\Omega(n^{-1/3})$ \\
	& & & & $O(n^{-1/6})$ \\
	\cline{2-5}
	& $r=2$ & $\Theta(n^{-1/2})$ & $\Theta(n^{-1/2}) $ & $\Omega(n^{-1/2})$ \\
	& & & & $O(n^{-1/3})$ \\
	\hline
	\multirow{3}{9em}{Planar graphs (i.e.\ $\mu \leq 3$)} & $r \geq 6$ & $\Theta(1)$ & $\Theta(1)$ & $\Theta(1)$ \\
	\cline{2-5}
	& $r=5$ & $\Omega(n^{-1/14})$ & $\Theta(1)$ & $\Theta(1)$ \\
	& & $O(n^{-1/30})$ & & \\
	\cline{2-5}
	& $r=4$ & $\Omega(n^{-1/9})$ & $\Omega(n^{-1/9})$ & $\Theta(1)$ \\
	& & $O(n^{-1/12})$ & $O(n^{-1/219})$ & \\
	\cline{2-5}
	& $r=3$ & $\Theta(n^{-1/5})$ & $\Theta(n^{-1/5}) $ & $\Omega(n^{-1/5})$ \\
	& & & & $O(n^{-1/6})$ \\
	\cline{2-5}
	& $r=2$ & $\Theta(n^{-1/2})$ & $\Theta(n^{-1/2}) $ & $\Omega(n^{-1/2})$ \\
	& & & & $O(n^{-1/3})$ \\
	\hline
	\multirow{2}{9em}{Outerplanar graphs (i.e.\ $\mu \leq 2$)} & $r \geq 3$ & $\Theta(1)$ & $\Theta(1)$ & $\Theta(1)$ \\
	\cline{2-5}
	& $r = 2$ & $\Theta(n^{-1/3})$ & $\Theta(n^{-1/3})$ & $\Theta(n^{-1/3})$ \\
	\hline
\end{tabular}
\caption{Examples of the threshold probability for certain minor-closed families and three special properties: being $(r-1)$-degenerate, $r$-choosable and $r$-colorable, for integers $r$.
	\\
	Among the results in this table, results that state the thresholds are equal to $\Theta(1)$ follow from known results in the literature; all other lower bounds of the thresholds in this table follow from our main theorems (Theorems \ref{thm: main critical exponent 1} and \ref{thm:critical exponent lower bound}); all other upper bounds follow from a combination of our tools developed in Section \ref{sec:upper bound} and either known results in the literature or trivial observations.}
	\label{table:cor}
\end{center}
	\end{table}
\end{corollary}

\subsection{Definitions and main questions}

A \emph{graph property} $\P$ is a class of graphs such that $\P$ is invariant under graph automorphisms. 
A graph class $\G$ is \emph{monotone} if every subgraph of a member of $\G$ is in $\G$.  
We remark that a graph property is also a graph class. 
So a graph property $\P$ is monotone if every subgraph of a member of $\P$ is in $\P$.  

Given a monotone graph property $\P$ and an infinite sequence $(G_n)_{n \in \mathbb{N}}$, a function $p^*: \mathbb{N} \to [0,1]$ is a {\it threshold (probability)} for $\P$ and $(G_n)_{n \in \mathbb{N}}$ if  for any slowly growing function $x(n)$: (1) $G_n(p^*(n) x(n)) \notin \P$ \text{a.a.s.};   and (2) $G_n(p^*(n) /x(n)) \in \P$ \text{a.a.s.} 
Thresholds for various graph properties in $\bG(n,p)$ were first observed by Erd\H{o}s and R\'enyi  \cite{ErdosRenyi}, and then generalized to  all monotone set properties and general random set models by Bollob{\'a}s and Thomason \cite{BT-existenceThreshold} (see also  \cite{FKalai-existenceThreshold}). 
In addition, the results in \cite{BT-existenceThreshold} imply the existence of the following more general setting of threshold probabilities for any monotone graph class $\G$. 

\begin{dfn}[Upper threshold $p_\G^\P$] \label{def upper threshold}
Let $\P$ be a monotone graph property and let $\G$ be a monotone graph class.
When $\G$ is an infinite family, we say that a function $p_\G^\P: \mathbb{N} \to [0,1]$ is an {\rm upper threshold} for $\G$ and $\P$ if the following two conditions hold. 
 \begin{enumerate}
	 \item For every sequence $(G_n)_{n \in {\mathbb N}}$ of graphs with $G_n \in \G$ and $\lvert V(G_n) \rvert = n$, and for any function $q: \mathbb{N} \to [0,1]$ with $p_\G^\P(n) / q(n) \to \infty$, the random subgraphs $G_n(q(n))$ are in $\P$ a.a.s. 
	 \item There exists a sequence $(G_n)_{n \in {\mathbb N}}$ of graphs with $G_n \in \G$ and $\lvert V(G_n) \rvert  = n$ such that for any function $q: \mathbb{N} \to [0,1]$ with $q(n)/p_\G^\P \to \infty$, the random subgraphs $G_n(q(n))$ are not in $\P$ a.a.s. 
\end{enumerate}
When $\G$ is finite, a function $p_\G^\P: \mathbb{N} \to [0,1]$ is an {\rm upper threshold} for $\G$ and $\P$ if $p_\G^\P=\Theta(1)$. 
\end{dfn}
In the case when $\G$ consists of a sequence  $(G_n)_{n \in \mathbb{N}}$ where $|V(G_n) | = n$ for each $n \in {\mathbb N}$, the definition for the upper threshold for $\G$ coincides with the aforementioned definition for a threshold for the sequence $(G_n)_{n \in \mathbb{N}}$. 
For simplicity, we call the upper threshold the threshold.
Note that $p_\G^\P$ is unique up to multiplying by positive constant factors.
So it is sufficient to determine the order of $p_\G^\P$.

In comparison to numerous results on thresholds for various graph properties on $\bG(n,p)$ (see e.g. \cite{B-RGbook,FK-RGbook, JLR-RGbook}), only few results are known when the host graphs are other graphs.
In addition, these few known results tend to depend on the high density, such as in \cite{GNS-threshold}, or special geometric or spectral features, such as being expanders \cite{Alon-percolation} or being an $n$-dimensional cube \cite{CCHSS-phasetransition-cube}. 

In this paper, we complement the knowledge in this direction by considering the case when the host graphs belong to minor-closed families.
Minor-closed families are classes of sparse graphs with a pure combinatorial property that generalizes a number of topological properties.

A graph $H$ is a {\it minor} of another graph $G$ if $H$ is isomorphic to a graph that can be obtained from a subgraph of $G$ by contracting edges.
A family $\G$ of graphs is {\it minor-closed} if every minor of any member of $\G$ belongs to $\G$. 
A minor-closed family is {\it proper} if it does not contain all graphs.  
Typical examples of minor-closed families include the class of planar graphs (and more generally, the class of graphs embeddable in a fixed surface), the class of linkless embeddable graphs, and the class of knotless embeddable graphs.

Inspired by coloring problems such as Question \ref{eq1}, the key property studied in this paper is the property of being $r$-degenerate, for any fixed integer $r$, defined below.

\begin{dfn}
Let $r$ be a nonnegative integer, and let $G$ be a graph.
Then
	\begin{itemize}
		\item $G$ is \emph{$r$-degenerate} if every subgraph of $G$ contains a vertex of degree at most $r$;
		\item $G$ is \emph{$r$-choosable} if for every list-assignment $(L_v: v \in V(G))$ with $\lvert L_v \rvert \geq r$, there exists a function $c$ that maps each vertex $v \in V(G)$ to an element of $L_v$ such that $c(x) \neq c(y)$ for any edge $xy$ of $G$. 
	\end{itemize}
\end{dfn}

It is well-known that any $r$-degenerate graph is $(r+1)$-choosable and $(r+1)$-colorable by a simple greedy algorithm. 
Note that a graph is not $r$-degenerate if and only if it contains a subgraph of minimum degree at least $r+1$.
So being non-$r$-degenerate is equivalent to having an ``$(r+1)$-core'', which is an object whose size has been extensively studied in random graphs (see \cite{Luczak-3-col,PSW-core} for example).

In this paper, we consider the following properties.

\begin{dfn}\label{dfn:prop}
Let $r$ be a positive integer.
We define
	\begin{itemize}
	 	\item $\D_r$ to be the property of being $(r-1)$-degenerate,
	 	\item $\chi_r$ to be the property of being $r$-colorable, 
	 	\item $\chi_r^\ell$ to be the property of being $r$-choosable, and
	 	\item $\R_r$ to be the property of having no $r$-regular subgraphs.
	\end{itemize}
\end{dfn}

\begin{question}\label{ques:deg_gen} 
For any integer $r \geq 2$ and proper minor-closed family $\G$, what is the threshold probability $p_\G^{\P}$, where $\P \in \{\D_r,\chi^\ell_r,\chi_r,\R_r\}$?
\end{question}

It is easy to see that $\D_r, \chi_r,\chi_r^\ell$ and $\R_r$ are monotone properties, $\D_r \subseteq \chi_r^\ell \subseteq \chi_r$ and $\D_r \subseteq \R_r$.
So results for $\D_r$ is a crucial part for our results on the other three properties.

For every graph $H$, let $\M(H)$ be the set of $H$-minor free graphs. 
Note that $\M(H)$ is a proper minor-closed family for any fixed graph $H$.
And for every proper minor-closed family $\G$, there exists a graph $H'$ such that $H' \not \in \G$, so $\G \subseteq \M(H')$.
This simple observation together with the following trivial proposition show that the heart of Question \ref{ques:deg_gen} is the special case stated in Question \ref{ques:deg}.

\begin{proposition}\label{prop:minor-closed-and-MH}
For any proper minor-closed family $\G$ and graph $H \not \in \G$, $p_\G^\P=\Omega(p_{\M(H)}^\P)$ for every monotone property $\P$.
\end{proposition}

\begin{question}\label{ques:deg} 
For any graph $H$ and integer $r \geq 2$, what is the threshold probability $p_{\M(H)}^{\D_r}$? 
\end{question}

\paragraph{Remark.} 
Question \ref{ques:deg} is expected to be very challenging: determining whether the threshold is $\Theta(1)$ or not for all $r$ and $H$ implies solutions of long-standing open questions in extremal graph theory about determining the degeneracy function or giving a constant approximation of the extremal function for all $H$. 
The \emph{degeneracy function} $d_H(n)$ is the minimum $d$ such that any $H$-minor free graph on $n$ vertices is $d$-degenerate. 
A simple observation shows that for any fixed connected graph $H$, determining whether the answer to Question \ref{ques:deg} is $\Theta(1)$ for every $r \geq 2$ is equivalent to determining $\lim_{n \rightarrow \infty}d_H(n)$, denoted by $d^*_H$. 
(See Proposition \ref{equiv_thr_den} for the precise description and a proof.\footnote{Note that $d_H$ is a non-decreasing function, as being $(r-1)$-degenerate remains when adding isolated vertices. 
In addition, a result of Mader \cite{Mader} implies that $d_H(n)$ has a constant (only depending on $H$) upper bound for every $n \in {\mathbb N}$.
Therefore $d^*_H$ is well-defined, which equals $\sup_{n \in {\mathbb N}}d_H(n)$.})
The limit $d^*_H$ is closely related to  another challenging function, the  \emph{extremal function} $f_H(n)$, which is the maximum possible number of edges in an $H$-minor free graph on $n$ vertices. 
Proposition \ref{2approx}  shows that $f_H^* \leq d^*_H \leq 2f_H^*$, where $f^*_H = \sup_{n \in {\mathbb N}} \frac{f_H(n)}{n}$.\footnote{Mader \cite{Mader} proved that $f^*_H$ exists.} 
Despite having been extensively studied, even approximating $f_H^*$ within a factor of $2$ is not known for general sparse graphs (see \cite{Thomason-survey}).
We remark that a combination of recent results \cite{nrtw,rw,rw2,Thomason-Wales} gives an approximation with a factor $\frac{0.319+\epsilon}{0.319-\epsilon}$ for almost every graph $H$ of average degree at least a function of $\epsilon$ (so a density condition for $H$ is still required), where $0<\epsilon<1$; and very recently, some results about the extremal function were obtained for the case when $H$ satisfies certain sparsity assumptions about its expansion, such as \cite{hkl,hnw}.

Similarly, the analogous problems for $r$-colorability or $r$-choosability correspond to Hadwiger's conjecture and its variants about the chromatic number or choice number of graphs in minor-closed families, which are other major open problems in graph theory.

In addition, even though the thresholds for $\D_r,\chi_r$ and $\chi_r^\ell$ are well-studied in $\bG(n,p)$ (all of which are $\Theta(n^{-1})$), those techniques do not work for $H$-minor free graphs, since $H$-minor free graphs are sparse and lacks symmetry.

\subsection{Our results}\label{subsec:ourresult}

Our main Theorem \ref{thm: main critical exponent 1} answers Question \ref{ques:deg} for a large family of pairs $(r,H)$. 
The other main Theorem \ref{thm:critical exponent lower bound} provides lower bounds for the rest of the pairs $(r, H)$. 
Those results are for degeneracy and choosability (i.e.\ the properties $\D_r$ and $\chi_r^\ell$).
They yield results for colorability and existence of regular subgraphs (i.e.\ the properties $\chi_r$ and $\R_r$, see Corollary \ref{cor:many_intro} and Theorems \ref{thm:critical exponent regular intro} and \ref{thm:critical exponent colorable intro}), but we do not put any effort to optimize results for these two properties in this paper.
The results for these four properties also generalize to arbitrary proper minor-closed family by Proposition \ref{prop:minor-closed-and-MH}. 

It turns out that the threshold for $H$-minor free graphs and degeneracy is closely related to the {\it vertex cover} of $H$, defined below.

\begin{dfn} 
A {\rm vertex-cover} of a graph $H$ is a subset $S$ of $V(H)$ such that $H-S$ is edgeless. 
We denote the minimum size of a vertex-cover of a graph $H$ by $\tau(H)$. 
\end{dfn}

We also need the following standard definition of a {\it join} of two graphs. 
\begin{dfn}
For any graphs $G,G'$ and positive integer $t$, we define $tG$ to be the disjoint union of $t$ copies of $G$, and define $G \vee G'$ to be the graph that is obtained from a disjoint union of $G$ and $G'$ by adding all edges with one end in $V(G)$ and one end in $V(G')$.
\end{dfn}

Clearly, $\tau(H)=0$ if and only if $H$ has no edge. 
Hence if $\tau(H)=0$, then no $H$-minor free graph has more than $\lvert V(H) \rvert$ vertices, so the threshold $p_{\M(H)}^\P$ is $\Theta(1)$ for any property $\P$.
Therefore, we are only interested in graphs $H$ with $\tau(H) \geq 1$.

The following theorem determines the threshold for $\M(H)$ and for degeneracy $\D_r$ and choosability $\chi_r^\ell$ for a large class of pairs $(r,H)$, including the case $\tau(H) > r$ or the case that $H$ has minimum degree at least $r$. 

\begin{theorem} \label{thm: main critical exponent 1}
Let $r \geq 2$ be an integer and $H$ a graph (not necessarily connected).  
Let $\P \in \{\D_r,\chi_r^\ell\}$. 
Then $$p_{\M(H)}^{\P}=\Theta(n^{-\frac{1}{q_H}})$$ in each of the following cases, where $q_H$ is an integer defined as follows.
	\begin{enumerate}
		\item If $\tau(H) \geq r+1$, then 
		$q_H=r$.
		\item If $1 \leq \tau(H) \leq r$ and $H$ is not a subgraph of $K_{\tau(H)-1} \vee tK_{r+2-\tau(H)}$ for any positive integer $t$, then $q_H=(r+2-\tau(H))r-{r+2-\tau(H)\choose 2}$.
		\item If $1 \leq \tau(H) \leq r$, $H$ has minimum degree at least $r$, and $H$ is not a subgraph of $K_{r-1} \vee tK_{2}$ for any positive integer $t$, then $q_H=2r-1$.
		\item If $1 \leq \tau(H) \leq r$, $H$ has minimum degree at least $r$, $H$ is a subgraph of $K_{r-1} \vee tK_{2}$ for some positive integer $t$, and $H \not \in \{K_2,K_3,K_4\}$, then $q_H = 3r-3$.
	\end{enumerate}
Furthermore, $p_{\M(H)}^{\P}=\Theta(1)$ if either
	\begin{itemize}
		\item $H = K_{r+1}$ and $r \leq 3$, or
		\item $H$ has at most one component on at least two vertices and every component of $H$ is an isolated vertex or a star of maximum degree at most $r$.
	\end{itemize}
\end{theorem}
Note that Statements 2 and 3 of Theorem \ref{thm: main critical exponent 1} are consistent since if $\tau(H) \leq r$ and $\delta(H) \geq r$, then $\tau(H)=r$. 

It is clear that for any fixed nonnegative integer $r$, the graphs $H$ in which the threshold is not determined in Theorem \ref{thm: main critical exponent 1} belong to the set
$$\mathcal{H}_r:=\{ H: 1 \leq \tau(H) \leq r \text{ and }  H \subseteq K_{\tau(H)-1} \vee t^*K_{r+2-\tau(H)} \text{ for some positive integer } t^*\}.$$
Theorem \ref{thm: main critical exponent 1} also shows\footnote{When $\tau(H)=1$, $H$ is a graph that is a disjoint union of $K_{1,s}$ for some positive integer $s$ and isolated vertices. Since $H \in {\mathcal H}_r$ and $\tau(H)=1$, $H$ is a subgraph of $t^*K_{r+1}$ for some positive integer $t^*$, so $s \leq r$, and hence every component of $H$ is either an isolated vertex or a star of maximum degree at most $r$.} that $p_{\M(H)}^{\D_r}=\Theta(1)$ if $H$ is a graph in ${\mathcal H}_r$ with $\tau(H)=1$.
Therefore, we have the following corollary. 

\begin{corollary} 
The thresholds for $\D_r$ and $\chi_r^\ell$ are determined by Theorem \ref{thm: main critical exponent 1} unless $H \in {\mathcal H}_r$ and $\tau(H) \geq 2$. 
\end{corollary}

We remark that the number of cases not covered by Theorem \ref{thm: main critical exponent 1} is not large.
Every graph in ${\mathcal H}_r$ has the property that deleting at most $\tau(H)-1$ vertices results in a graph whose every component has at most $r+2-\tau(H) \leq r+1$ vertices.
Even though every graph $W$ is a subgraph of $K_{\tau(W)} \vee \lvert V(W) \rvert K_1$, which looks close to the definition of the graphs in ${\mathcal H}_r$, there is no control for the maximum degree of the remaining graph if we only delete $\tau(W)-1$ vertices from $W$.

We provide a lower bound for the thresholds in Theorem \ref{thm:critical exponent lower bound} for the cases not covered by Theorem \ref{thm: main critical exponent 1}. 

The lower bound in Theorem \ref{thm:critical exponent lower bound} might look artificial at first glance, but it naturally arises from constructions that establish upper bounds. 
Those upper bounds come from gluing several copies of the same graph.
We explain the intuition before we provide the formal description.
We use the following notation. 

\begin{dfn} \label{dfn:wedge} 
Let $G$ be a graph, and let $Z=\{z_1,z_2,...,z_{\lvert Z \rvert}\}$ be a subset of $V(G)$.
For any positive integer $k$, let $G \wedge_k Z$ be the graph obtained from a union of $k$ disjoint copies of $G$ by identifying, for each $i$ with $1 \leq i \leq \lvert Z \rvert$, the $k$ copies of $z_i$. 
\end{dfn}

For example, if $G$ is a star with three leaves and $Z$ is the set of the leaves, then $G \wedge_k Z$ is $K_{k,3}$.

Note that if every vertex in $V(G)-Z$ has degree at least $r$ in $G$, then $G \wedge_k Z$ is not $r$-degenerate when $k$ is sufficiently large, and it can be used to prove upper bounds for the thresholds.
This motivates the following notions.

\begin{dfn}\label{dfn:FGHR}
For graphs $G$ and $F_0$ and a nonnegative integer $r$, define ${\mathcal F}(G,F_0,r)$ to be the set consisting of the graphs that can be obtained from a disjoint union of $G$ and $F_0$ by adding edges between $V(G)$ and $V(F_0)$ such that every vertex in $V(F_0)$ has degree at least $r$.

For a graph $F$ in ${\mathcal F}(G,F_0,r)$, the {\rm type} of $F$ is the number of edges of $F$ incident with $V(F_0)$, and we call $V(G)$ the {\rm heart} of $F$.
\end{dfn}

Note that every graph in ${\mathcal F}(G,F_0,r)$ has type at least $r$. Figure \ref{fig:pedal}(a) is an example of some $F \in \mathcal{F}(I_2, K_3, 3)$ of type $6$.\footnote{For every nonnegative integer $t$, we denote the edgeless graph on $t$ vertices by $I_t$, where $I_0$ is the empty graph with no vertices. Note $I_t = tK_1$. We will use the notation $I_t$ instead of $tK_1$ for simplicity because the description for $t$ can be complicated.}
Figure \ref{fig:pedal}(b) is an example of $F \wedge_t V(I_2)$ for some $F \in \mathcal{F}(I_2, K_3, 3)$ of type $6$ and $t = 4$.

  \begin{figure}[h]
    \begin{subfigure}[t]{0.4\textwidth}
      \includegraphics[trim= 20pt 60pt 0pt 0pt, width=\textwidth, scale=0.3]{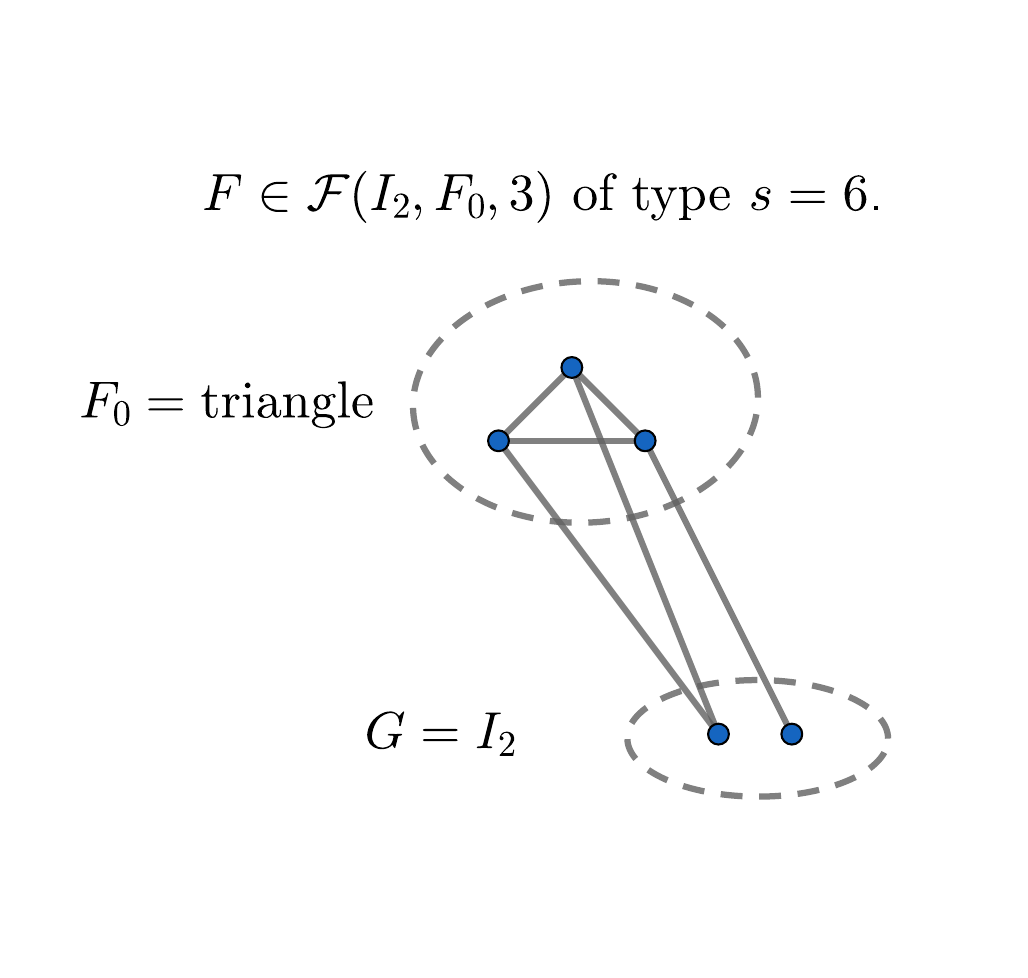}
      \caption{\footnotesize{$F_0$ is a triangle $K_3$. Each vertex of $F_0$ has degree at least $r=3$ in $F$. There are in total $6$ edges incident with vertices in $F_0$. Thus $F \in \mathcal{F}(I_2, F_0, 3)$ and is of type $6$. The vertex-set of $I_2$ is the heart of $F$.}}
    \end{subfigure}
    \hfill
    \begin{subfigure}[t]{0.35\textwidth}
      \includegraphics[trim= 0pt 0pt 0pt 0pt, width=\textwidth, scale=0.3]{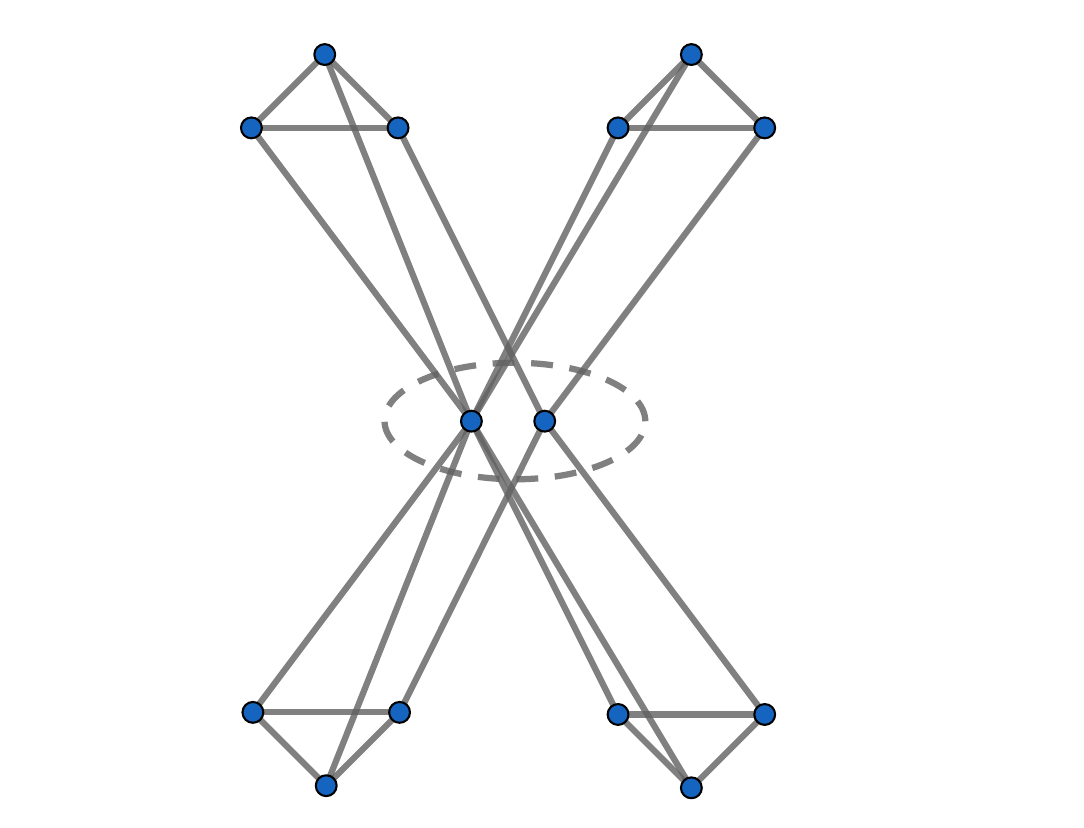}
      \caption{\footnotesize{$F \wedge_4 Z$, where $Z$ is the heart of $F$.}}
    \end{subfigure}
  \caption{An example of a graph $F \in \mathcal{F}(I_2, F_0, 3)$ of type $6$ and $F \wedge_4 Z$ for some set $Z$.}  \label{fig:pedal}
    \end{figure}

The type defined in Definition \ref{dfn:FGHR} provides an upper bound for the thresholds:
Assume that $F$ is a graph in ${\mathcal F}(I,F_0,r)$ with type $q$ for some edgeless graph $I$.
Consider $G_k:=F \wedge_k V(I)$ for sufficiently large $k$.
Note that every edge of $G_k$ is incident with a copy of $F_0$ since $I$ is edgeless.
If a subgraph of $G_k$ is $(r-1)$-degenerate, then for each of almost all copies $C$ of $F$ in $G_k$, at least one edge in $C$ cannot appear in this $(r-1)$-degenerate subgraph.
So a simple probabilistic argument shows that the random subgraph $G_k(p)$ of $G_k$ is not $(r-1)$-degenerate if the probability $p$ is $\omega(|V(G_k)|^{-1/q})$.
Hence if the graph class $\G$ contains such $G_k=F \wedge_k V(I)$ for infinitely many integers $k$, then the type of $F$ gives an upper bound for the threshold for $\G$ and $\D_r$. 
Therefore, we would like to know the smallest type of such graphs $F$.
And this smallest type is essentially $s_r(H)+1$, where $s_r(H)$ is the number in the following definition, except that we have to restrict $s_r(H) \leq {r+1 \choose 2}$ for some technical reasons.

\begin{dfn} \label{def:sr}
For any graph $H$ and integer $r \geq 2$, let $s_r(H)$ be the largest integer $s$ with $0 \leq s \leq {r+1 \choose 2}$ such that for every integer $s'$ with $0 \leq s' \leq s$, every connected graph $F_0$ and every graph $F \in {\mathcal F}(I_{\tau(H)-1},F_0,r)$ of type $s'$, $H$ is a minor of $F \wedge_t I$ for some positive integer $t$, where $I$ is the heart of $F$. 
\end{dfn}

Note that $s_r(H) \geq r-1$, since there exists no connected graph $F_0$ such that there exists a graph in ${\mathcal F}(I_{\tau(H)-1},F_0,r)$ of type at most $r-1$.

As discussed above, $s_r(H)+1$ gives an upper bound for the threshold, subject to an extra requirement that $s_r(H) \leq {r+1 \choose 2}$.
Our result for lower bounds (Theorem \ref{thm:critical exponent lower bound}) essentially matches this upper bound $s_r(H)+1$ and the upper bounds appeared in Theorem \ref{thm: main critical exponent 1}, except that we also have to consider ${r+1 \choose 2}$ due to technical reasons.

\begin{theorem}\label{thm:critical exponent lower bound}
Let $r \geq 2$ be an integer and $H \in {\mathcal H}_r$. 
Let $\P \in \{\chi_r^\ell,\D_r\}$. 
	\begin{enumerate}
		\item If $\tau(H) \geq 2$, then $s_r(H) \geq (r-\tau(H)+2)r-{r-\tau(H)+2 \choose 2}-1 = {r+1 \choose 2} - {\tau(H)-1 \choose 2} -1$.
		\item If $2 \leq \tau(H) \leq r$, then $p_{\M(H)}^{\P}=\Omega(n^{-1/q_H})$, where $q_H=\min\{s_r(H)+1,{r+1 \choose 2}\}$. 
	\end{enumerate}
\end{theorem}

Simple corollaries of Theorems \ref{thm: main critical exponent 1} and \ref{thm:critical exponent lower bound} for many extensively studied minor-closed families are included in Corollary \ref{cor:many_intro}.
We remark that it is far from a complete list of corollaries that can be derived from Theorems \ref{thm: main critical exponent 1} and \ref{thm:critical exponent lower bound}.

We will sketch the proof ideas in the next section.

\section{Proof ideas and organization}

\subsection{Notations}
In this paper, graphs are simple.
Let $G$ be a graph and $X$ a subset of $V(G)$.
We denote the subgraph of $G$ induced by $X$ by $G[X]$.
We define $N_G(X)=\{v \in V(G)-X: v$ is adjacent in $G$ to some vertex in $X\}$, and define $N_G[X]=N_G(X) \cup X$.
For any vertex $v$, $G-v, N_G(v)$ and $N_G[v]$ are defined to be $G[V(G)-\{v\}]$, $N_G(\{v\})$ and $N_G[\{v\}]$, respectively.
The {\it degree} of a vertex is the number of edges incident with it.
The minimum degree of $G$ is denoted by $\delta(G)$. 
The {\it length} of a path is the number of its edges.
The {\it distance} of two vertices in $G$ is the minimum length of a path in $G$ connecting these two vertices; the distance is infinity if no such path exists.

For every real number $k$, we define $[k]$ to be the set $\{x \in {\mathbb Z}: 1 \leq x \leq k\}$. 
We use $\mathbb{N}$ to denote the set of all positive integers, which does not include $0$. 

\subsection{Proof ideas}
To establish an upper bound $f$ of the thresholds in Theorem \ref{thm: main critical exponent 1}, it suffices to construct a sequence $(G_n: n \in \mathbb{N})$ of graphs in $\M(H)$ such that $\lim_{n \rightarrow \infty}\Pr(G_n(p(n)) \in \P)=0$ for every function $p$ with $f(n)/ p(n) \to 0$, as shown in Section \ref{sec:upper bound}.
These constructions use notions in Definitions \ref{dfn:wedge} and \ref{dfn:FGHR}.
Roughly speaking, the graphs $G_n$ in the construction are altered from the complete bipartite graphs such that they have minimum degree at least $r$ in various ways.

Proving the lower bounds in Theorems \ref{thm: main critical exponent 1} and \ref{thm:critical exponent lower bound} are much more difficult, and most of the paper is dedicated to it.
It suffices to prove lower bounds for $p_{\M(H)}^{\D_r}$ and then use the following trivial observation (Proposition \ref{relation three properties}) for other properties. 

\begin{proposition} \label{relation three properties}
For any positive integer $r$ and graph class $\G$, the threshold for $\D_r$ is upper bounded by each of the thresholds for the properties $\chi_r^\ell, \chi_r$ and $\R_r$.
\end{proposition}

If $G(p)$ is $(r-1)$-degenerate, every subgraph $R$ of $G$ with $\delta(R) \geq r$ has to have some edge disappear in $G(p)$. 
Hence the lower bound of the threshold for $\D_r$ comes from a value $p$ that ensures that every subgraph $R$ of $G$ with $\delta(R) \geq r$ has a missing edge in $G(p)$. 
If one ignores the interplay between different subgraphs of $G$ of minimum degree at least $r$, then since there are possibly $O(2^{|E(G)|}) = O(2^{O(n)})$ such subgraphs, there is no hope to obtain a lower bound of the form $O(n^{-1/q})$ (for some $q>0$) which is the corresponding upper bound.

The strategy to obtain a lower bound of the form $O(n^{-1/q})$ is to find a small set of ``signatures'' such that every subgraph of minimum degree at least $r$ contains such a signature, so as long as $G(p)$ misses an edge for each signature, $G(p)$ has no subgraph of minimum degree at least $r$.
The following definitions and lemma formalize this idea. 

\begin{dfn}\label{def:goodsig}
For any real number $c$ and nonnegative integers $q$ and $r$, a {\rm $(c,q,r)$-good signature collection} for a graph $G$ is a collection $\C$ of subsets of $E(G)$ with the following properties.
\begin{enumerate}
\item Each member of $\C$ has exactly $q$ edges. 
\item $|\C| \leq c |V(G)|$. 
\item For every subgraph of $G$ of minimum degree at least $r$, its edge-set contains some member in $\C$. 
\end{enumerate}
\end{dfn}

\begin{dfn}\label{def:goodsig2}
For a given graph class $\G$ and nonnegative integers $q$ and $r$, we say $\G$ has {\rm $(q,r)$-good signature collections} if there is a constant $c = c(\G)$ such that for every graph $G$ in $\G$, there is a $(c,q,r)$-good signature collection for $G$.
\end{dfn}

The following lemma shows that the existence of $(q,r)$-good signature collections for $\G$ provides a lower bound on the threshold probability in terms of $q$.

\begin{lemma}\label{lem:coverCandp}
Let $\G$ be a class of graphs and $q,r$ be positive integers. 
If $\G$ has $(q,r)$-good signature collections,
then $p_\G^{\D_r}=\Omega(n^{-1/q})$.
\end{lemma}

\begin{proof}
Let $p^*: \mathbb{N} \rightarrow [0,1]$ be the function such that $p^*(n) = n^{-1/q}$ for every $n \in {\mathbb N}$. 
Let $p: \mathbb{N} \rightarrow [0,1]$ be a function with $\lim_{n \rightarrow \infty}p(n)/p^*(n)=0$. 
Let $(G_n)_{n \in {\mathbb N}}$ be a sequence of graphs in $\G$ such that $\lvert V(G_n) \rvert=n$ for every $n \in {\mathbb N}$.
To show $p_\G^{\D_r}=\Omega(n^{-1/q})$, it suffices to show that $\lim_{n\rightarrow \infty} \Pr(G_n(p(n)) \in \D_r) = 1$. 

Since $\G$ has $(q,r)$-good signature collections, there exists a constant $c$ such that for any integer $n$, there exists a $(c,q,r)$-good collection $\C_n$ for $G_n$. 
For each $T \in \C_n$, since $\lvert T \rvert=q$, $\Pr(T \subseteq E(G_n(p(n))))=(p(n))^q$. 
Since for each subgraph $R$ of $G_n$ with $\delta(R) \geq r$, there exists $T \in \C_n$ with $T \subseteq E(R)$, so the probability that $G_n(p(n))$ contains a subgraph of minimum degree at least $r$ is at most the probability that some member of $\C_n$ is a subset of $E(G_n(p(n)))$ which is at most $\lvert \C_n \rvert (p(n))^q$ by a union bound. 
But as $n \rightarrow \infty$, $ \lvert \C_n \rvert (p(n))^q \leq cn \cdot (p(n))^q = cn \cdot (n^{-1/q} \cdot \frac{p(n)}{p^*(n)})^q = c \left(\frac{p(n)}{p^*(n)}\right)^q \to 0.$

Thus with probability approaching $1$ as $n$ approaches infinity, no subgraph of minimum degree at least $r$ is contained in $G_n(p(n))$. 
Therefore, $\lim_{n \rightarrow \infty} \Pr(G_n(p(n)) \in \D_r)=1$. 
\end{proof}

\begin{remark}
We want to emphasize that the value $q$ in Lemma \ref{lem:coverCandp} determines the lower bound for $p_\G^{\D_r}$.
The majority of work of this paper is to find the largest or sufficiently large value of $q$, which turns out to be the value $q_H$ defined in Theorems \ref{thm: main critical exponent 1} and \ref{thm:critical exponent lower bound}.
Theorems \ref{thm: main critical exponent 1} and \ref{thm:critical exponent lower bound} are proved by combining it with the upper bound established in Section \ref{sec:upper bound}.
\end{remark}

We will prove the existence of $(q,r)$-good signature collections with a large value of $q$ in Lemmas \ref{weak general collection} and \ref{minor collection}.
We will show how these two lemmas imply the main theorems \ref{thm: main critical exponent 1} and \ref{thm:critical exponent lower bound} in Section \ref{sec:Cimplies-main-thm}.

A key sufficient condition to establish the existence of $(q,r)$-good signature collections in Lemmas \ref{weak general collection} and \ref{minor collection} is the following. 

\begin{lemma} \label{sufficient good collection}
Let $r$ be a positive integer and $\G$ a hereditary graph class\footnote{A graph class is {\it hereditary} if every induced subgraph of a member of this class is a member of this class. Note that $\M(H)$ is hereditary for every graph $H$.}.
If there exist nonnegative real numbers $a,t,\zeta$ with $t \leq 2r+1$ such that for every graph $G \in \G$, there exist a subset $Z$ of $V(G)$ with $\lvert Z \rvert \leq \zeta$ and a vertex $z^* \in Z$ such that 
	\begin{enumerate}
		\item every vertex in $Z$ has degree at most $a$ in $G$, and 
		\item for every subgraph $R$ of $G$ with $\delta(R) \geq r$ and with $z^* \in V(R)$, $\lvert V(R) \cap Z \rvert \geq t$, 
	\end{enumerate}
then every graph in $\G$ has a $\left( {\zeta \choose t}{a \choose r}^t, rt-{t \choose 2}, r\right)$-good signature collection. In other words, $\G$ has $\left( rt-{t \choose 2}, r\right)$-good signature collections.
\end{lemma}

\begin{proof}
Assume that for every graph $G \in \G$, there exist $Z$ and $z^*$ satisfying the two conditions stated in this lemma.
We shall prove that every graph $G \in \G$ has a $\left( {\zeta \choose t}{a \choose r}^t, rt-{t \choose 2}, r\right)$-good signature collection by induction on the number of vertices in $G$. 
The claim trivially holds when $|V(G)|=1$, as there exists no subgraph of $G$ of minimum degree at least one.

For any set $T$ of $t$ distinct vertices $z_1,...,z_t$ in $Z$ and every sequence $s=(S_{T,1},S_{T,2},...,S_{T,t})$, where $S_{T,i}$ is a set consisting of $r$ edges of $G$ incident with $z_i$ for every $i \in [t]$, let $S_s = \bigcup_{j=1}^tS_{T,j}$.
Note that $|S_s| \geq rt - \binom{t}{2}$.
Let $\C_0$ be the collection of all possible such sets $S_s$.
Then $\lvert \C_0 \rvert \leq {\zeta \choose t}{a \choose r}^t$ as the number of $t$-element subsets of $Z$ is at most ${\zeta \choose t}$, and each vertex in $Z$ is incident with at most $a$ edges.

The second condition mentioned in the statement of this lemma implies that for every subgraph $R$ of $G$ with $\delta(R) \geq r$ and with $z^* \in V(R)$, the edge-set $E(R)$ contains some member of $\C_0$.
Since $\G$ is hereditary, $G-z^* \in \G$.
Applying the induction hypothesis to $G-z^*$, $G-z^*$ has a $\left( {\zeta \choose t}{a \choose r}^t, rt-{t \choose 2}, r\right)$-good signature collection $\C_1$.
For every subgraph $R$ of $G$ with $\delta(R) \geq r$ and $z^* \not \in V(R)$, $R$ is a subgraph of $G-z^*$ with $\delta(R) \geq r$, so $E(R)$ contains some member of $\C_1$ by the induction hypothesis.

Let $\C_2=\C_0 \cup \C_1$.
Then $\C_2$ has the property that for every subgraph $R$ of $G$ with $\delta(R) \geq r$, $E(R)$ contains some member of $\C_2$.
In addition, by the induction hypothesis, $\lvert \C_2 \rvert \leq \lvert \C_0 \rvert + \lvert \C_1 \rvert \leq {\zeta \choose t}{a \choose r}^t + {\zeta \choose t}{a \choose r}^t(\lvert V(G) \rvert-1)={\zeta \choose t}{a \choose r}^t\lvert V(G) \rvert$.

Note that $\C_2$ satisfies the conditions of being a $\left( {\zeta \choose t}{a \choose r}^t, rt-{t \choose 2}, r\right)$-good signature collection except some member of $\C_2$ possibly has size strictly greater than $rt-{t \choose 2}$.
For each member $M$ of $\C_2$, let $f(M)$ be an arbitrary subset of $M$ of size $rt-{t \choose 2}$.
Note that for every subgraph $R$ of $G$ with $\delta(R) \geq r$, $E(R)$ contains some member $M$ of $\C_2$ and hence contains $f(M)$.
Then the collection $\{f(M): M \in \C_2\}$ is a $\left( {\zeta \choose t}{a \choose r}^t, rt-{t \choose 2}, r\right)$-good signature collection for $G$.
\end{proof}

Note that the exponent of $n$ in $p_{\M(H)}^{\D_r}$ is essentially determined by the size $q$ of the members of $\C$ mentioned in Lemma \ref{lem:coverCandp}, and $q$ is determined by the value $t$ mentioned in Lemma \ref{sufficient good collection}.
The majority of work of this paper is to prove the sufficient condition in Lemma \ref{sufficient good collection} with the correct value $t$. 
The starting point is the following lemma.

\begin{lemma}[Special case of Lemma \ref{stronger better path island shallow}] \label{special_path_island_sketch}
For any $\alpha,\beta \in {\mathbb N}$ and nonnegative integer $\ell$, there exists a real number $d$ such that for every graph $G$ with no $K_{\alpha,\beta}$-minor, there exist $X \subseteq V(G)$ and $z^* \in X$ such that 
			\begin{enumerate}
				\item every vertex in $X$ has degree at most $d$ in $G$, and
				\item there exists a set $S \subseteq V(G)-X$ with $\lvert S \rvert \leq \alpha-1$ such that for every vertex $x \in X$, if the distance between $x$ and $z^*$ is at most $\ell-1$ in $G[X]$, then every neighbor of $x$ in $G$ not contained in $X$ is contained in $S$. 
			\end{enumerate}
\end{lemma}

It is easy to see that every graph with no $H$-minor has no $K_{\tau(H),\beta}$-minor for some large $\beta$, so Lemma \ref{special_path_island_sketch} can be applied to $H$-minor free graphs by choosing $\alpha=\tau(H)$ and some $\beta$ and $\ell$. 
Assume $\alpha=\tau(H) <r$.
Then for every subgraph $R$ of $G$ with $\delta(R) \geq r$ containing $z^*$, it must contain at least $r-\lvert S \rvert \geq r-\alpha+1$ neighbors of $z^*$ in $X$; moreover, unless there are many edges of $R$ between those neighbors, $R$ also contains neighbors of those neighbors in $X$, and so on.
So we can obtain the desired set $Z$ in Lemma \ref{sufficient good collection} by defining it to be the set of vertices in $X$ with distance at most $\ell-1$ from $z^*$ in $G[X]$ (where $\ell$ only depends on the value $t$ that we aim), and we are done.
However, this argument is based on two assumptions: one is $\tau(H)<r$ and the other is that $R$ does not contain many edges between neighbors of $z^*$ in $X$.
The second one is not serious, because in that case $R$ still contains many edges incident with $Z$ and we can modify the proof of Lemma \ref{sufficient good collection} so that the parameter in a good signature collection determining the value $q$ in the exponent of $n$ in the threshold is the lower bound for the number of edges of $R$ incident with $Z$ that we can ensure, instead of a function of the lower bound $t$ for $\lvert V(R) \cap Z \rvert$.
The assumption $\tau(H)<r$ is more serious. 
So we should not assume it.
Before we proceed further, we remark that if $C$ is the component of $R[X \cap V(R)]$ containing $z^*$, then $R[V(C) \cup (V(R) \cap S)]$ is a graph in $\F(R[V(R) \cap S],C,r)$ whose type equals the number of edges of $R$ incident with a vertex in $V(C) \subseteq Z$.

In fact, Lemma \ref{special_path_island_sketch} is a very special case of what we really prove (Lemma \ref{stronger better path island shallow}).
Lemma \ref{stronger better path island shallow} actually shows that we can obtain many such $z^*$ instead of just one $z^*$, with the same set $S$.
It will allow us to construct a subgraph of $G$ isomorphic to $F \wedge_t I$ for some graph $F$ in $\F(I_{\tau(H)-1},C,r)$ with type equal to the number of edges of $R$ incident with $V(C) \subseteq Z$, and for some large integer $t$, where $I$ is the heart.
So if we cannot make sure that $R$ has many edges incident with $V(C)$, then we get a graph $F \wedge_t I$ for some $F \in \F(I_{\tau(H)-1},C,r)$ with small type; but it implies that $s_r(H)$ is small by its definition, since $G$ is $H$-minor free.
Hence it would imply that the value $q$ in the threshold is bounded by $s_r(H)$, as stated in Theorem \ref{thm:critical exponent lower bound}.
With more work we can improve this bound for some graphs $H$ to obtain Theorems \ref{thm: main critical exponent 1}.
This completes the proof sketch. 
More details can be found in Sections \ref{sec:proof-of-C} and \ref{sec:Cimplies-main-thm}. 

We remark that Lemma \ref{stronger better path island shallow} applies to graph classes that are more general than minor-closed families and is of independent interests.
In particular, it generalizes a result of Ossona de Mendez, Oum, and Wood \cite{oow}, and the first author further extends it in a later paper \cite{l_homo} to solve a number of open Tur\'{a}n-type questions for robustly sparse graphs.

\subsection{Organization of the paper}
We prove upper bounds for the threshold probabilities in Section \ref{sec:upper bound} and prove the lower bounds in Sections \ref{sec:islands}, \ref{sec:proof-of-C}, and \ref{sec:Cimplies-main-thm}. 
As we have discussed earlier,  the main lemmas are Lemmas \ref{weak general collection} and \ref{minor collection} regarding the existence of good collections, which are proved in Section \ref{sec:proof-of-C}. 
Lemma \ref{weak general collection} is simple, but Lemma \ref{minor collection} is much more complicated and requires a technical lemma (Lemma \ref{stronger better path island shallow}, which is proved in Section \ref{sec:islands}). 
We then use Lemmas \ref{weak general collection} and \ref{minor collection} to prove the main theorems in Section \ref{sec:Cimplies-main-thm}.
More intuitions for the proof of those lemmas will be provided in Sections \ref{sec:proof-of-C} and \ref{sec:Cimplies-main-thm}.
Finally we conclude the paper with some remarks in Section \ref{sec:concluding remarks}.

\section{Upper bounds for the threshold probabilities} \label{sec:upper bound}

Our goal in this section is proving Corollary \ref{upper bound threshold} which proves some upper bounds of the thresholds. 
We will construct sequences of graphs $(G_n: n \in {\mathbb N})$  that are hard to be made $(r-1)$-degenerate by randomly deleting edges.
Namely, if $p$ goes to 0 too slow, then $\lim_{n \rightarrow \infty}\Pr(G_n(p) \in \D_r)=0$.
These sequences $(G_n: n \geq 1)$ will be used to establish upper bounds for $p_{\M(H)}^{\D_r}$ for different graphs $H$.
The same construction will also be used for proving upper bounds for $p_{\M(H)}^{\chi^\ell_r}$.

A {\it stable set} in a graph is a subset of pairwise non-adjacent vertices. 
A standard second moment method proves the following lemma.

\begin{lemma} \label{prob general}
Let $Q$ be a graph and $Z$ a (possibly empty) stable set in $Q$.
Let $q=\lvert E(Q) \rvert \geq 2$. 
For every $n \in {\mathbb N}$, let $\ell_n=\lfloor \frac{n-\lvert Z \rvert}{\lvert V(Q) \rvert-\lvert Z \rvert} \rfloor$ and let $G_n$ be the graph obtained from $Q \wedge_{\ell_n}Z$ by adding isolated vertices to make $G_n$ have $n$ vertices.
Let $p: {\mathbb N} \rightarrow [0,1]$ with $\lim_{n \rightarrow \infty} n^{-1/q}/p(n)=0$. Then for every $k \in {\mathbb N}$, $\lim_{n \rightarrow \infty}\Pr(Q \wedge_k Z \subseteq G_n(p))=1$.
\end{lemma}

\begin{proof}
Note that for every $n \in {\mathbb N}$, $G_n$ contains $\ell_n$ edge-disjoint copies of $Q$, denoted by $A_1,A_2,...,A_{\ell_n}$.
For each $1 \leq i \leq \ell_n$, define a random variable $X_i$ to be 1 if all the edges of $A_i$ remain in the random subgraph $G_n(p)$; let $X_i = 0$ otherwise. 

Let $X = \sum_{i=1}^{\ell_n} X_i$.
It suffices to show $\lim_{n \to \infty} \Pr(X \geq k) = 1$. 
This is because when $X \geq k$, $\bigcup_{i: X_i=1}A_i$ contains $Q \wedge_k Z$ as a subgraph in $G_n(p)$, as desired. 

Note $\mathbb{E}[X_i] = \Pr(X_i =1) = (p(n))^q$. 
So $n\mathbb{E}[X_i] \to \infty$ as $n \to \infty$, by the definition of $p$.
Hence $\ell_n\mathbb{E}[X_i] \to \infty$ as $n \to \infty$ by the definition of $\ell_n$.
By the linearity of expectation, 
$\mathbb{E}[X] =\sum_{i=1}^{\ell_n}\mathbb{E}[X_i] =  \ell_n \mathbb{E}[X_i] \gg k$.
Since $X_i$'s are i.i.d.\ random variables, the desired inequality $\lim_{n \to \infty} \Pr(X \geq k) = 1$ is a consequence of standard concentration inequalities (such as Chebyshev's inequality). 
\end{proof}

\begin{lemma} \label{claim:UBd}
Let $r,r'$ be  integers with $r \geq 2$ and $0 \leq r' \leq r$ and let $s$ be a nonnegative integer.
Let $F_0$ be a connected graph and let $F \in {\mathcal F}(I_{r'},F_0,r)$ be of type $s$. 
Let $Z$ be the heart of $F$ (thus $Z$ is a stable set of size $r'$ in $F$).

For every positive integer $n$, let $\ell_n = \lfloor \frac{n-\lvert Z \rvert}{\lvert V(F_0) \rvert} \rfloor$ and let $G_n$ be an $n$-vertex graph obtained from $F \wedge_{\ell_n} Z$ by adding isolated vertices to make $G_n$ have $n$ vertices.
Let $p: {\mathbb N} \rightarrow [0,1]$ with $\lim_{n \rightarrow \infty} n^{-1/s}/p(n)=0$.
Then $\lim_{n \rightarrow \infty} \Pr(G_n(p) \in \D_r) = 0$.
\end{lemma}

\begin{proof}
By Lemma \ref{prob general}, $\lim_{n \rightarrow \infty} \Pr(F \wedge_r Z \subseteq G_n)=1$.
We claim $F \wedge_r Z$ has a subgraph of minimum degree at least $r$. 
Every vertex in $V(F) \setminus Z$ has degree at least $r$ in $F$.
So every vertex in $V(F \wedge_r Z) \setminus Z$ has degree at least $r$ in $F \wedge_r Z$. 
For each vertex in $Z$, if it has zero degree in $F$, it has zero degree in $F \wedge_r Z$; if it has degree at least one in $F$, it has degree at least $r$ in $F \wedge_r Z$ as each of its neighbors has $r$ copies in $F \wedge_r Z$. 
So some component of $F \wedge_r Z$ has of minimum degree at least $r$. 
Therefore, $\lim_{n \rightarrow \infty} \Pr(G_n(p) \in \D_r) = 0$.
\end{proof}

Recall that for every nonnegative integer $t$, we denote the edgeless graph on $t$ vertices by $I_t$, where $I_0$ is the empty graph with no vertices. 
The proof of the following lemma is straightforward, so we move it to the appendix (Section \ref{subsec:appendix_pf_color}).

\begin{lemma} \label{nonchoosable}
Let $r$ be a positive integer and $w$ an integer with $0 \leq w \leq r$.
Then $I_{r-w} \vee r^{r-w}K_{w+1}$ is not $r$-choosable.
\end{lemma}

\begin{lemma}\label{claim:UB2}
Let $r$ be an integer with $r \geq 2$ and let $w$ be an integer with $r \geq w \geq 0$. 
Let $q = (w+1) r - \binom{w+1}{2}$.
For every $n \in {\mathbb N}$ with $n>r$, let $G_n$ be the $n$-vertex graph obtained from $I_{r-w} \vee \lfloor \frac{n-(r-w)}{w+1} \rfloor K_{w+1}$ by adding isolated vertices to make $G_n$ have $n$ vertices. 
Let $p: {\mathbb N} \rightarrow [0,1]$ be a function with $\lim_{n \rightarrow \infty} n^{-1/q}/p(n)=0$.
Then the following hold.
	\begin{enumerate}
		\item $\lim_{n \rightarrow \infty} \Pr(I_{r-w} \vee r^r K_{w+1} \subseteq G_n(p) ) = 1$.
		\item $\lim_{n \rightarrow \infty} \Pr(G_n(p) \in \D_r) = \lim_{n \rightarrow \infty} \Pr(G_n(p) \in \chi_r^\ell) = 0$.
		\item If $r \neq w$, then $\lim_{n \rightarrow \infty} \Pr(G_n(p) \in \chi_{w+1})=0$.
		\item If $r$ is divisible by $w+1$, then $\lim_{n \rightarrow \infty} \Pr(G_n(p) \in \R_r) = 0$.
	\end{enumerate}
\end{lemma}

\begin{proof}
Let $Q=I_{r-w} \vee K_{w+1}$, and let $Z$ be the subset of $V(Q)$ corresponding to $V(I_{r-w})$.
For every $n \in {\mathbb N}$ with $n>r$, let $G_n'=Q \wedge_{\lfloor \frac{n-(r-w)}{w+1} \rfloor} Z$.
So for every $n \in {\mathbb N}$, $G_n$ is the graph obtained from $G_n'$ by adding isolated vertices to make $G_n$ have $n$ vertices.
Note that the number of edges of $Q$ incident with at least one vertex in $V(Q)-Z$ is $(w+1)(r-w) + \binom{w+1}{2} = q$.

Note that for any positive integer $k$, $I_{r-w} \vee k K_{w+1} = Q\wedge_{k} Z$.
Hence by Lemma \ref{prob general}, $\lim_{n \rightarrow \infty} \Pr(Q \wedge_{r^r} Z \subseteq G_n(p))  = \lim_{n \rightarrow \infty} \Pr(I_{r-w} \vee r^r K_{w+1} \subseteq G_n(p)) = 1$.
Since $I_{r-w} \vee r^r K_{w+1}$ has minimum degree at least $r$, $\lim_{n \rightarrow \infty} \Pr(G_n(p) \in \D_r) = 0$.
By Lemma \ref{nonchoosable}, $I_{r-w} \vee r^r K_{w+1}$ is not $r$-choosable, so $\lim_{n \rightarrow \infty} \Pr(G_n(p) \in \chi_r^\ell) = 0$.

If $r \neq w$, then $I_{r-w} \vee r^r K_{w+1}$ is not properly $(w+1)$-colorable, so $\lim_{n \rightarrow \infty} \Pr(G_n(p) \in \chi_{w+1})=0$.

If $r$ is divisible by $w+1$, then $I_{r-w} \vee \frac{r}{w+1} K_{w+1}$ is a $r$-regular subgraph of $I_{r-w} \vee r^r K_{w+1}$, so $\lim_{n \rightarrow \infty} \Pr(G_n(p) \in \R_r) = 0$.
\end{proof}

To define another sequence of graphs that are hard to be made $(r-1)$-degenerate, we need the following definition.
\begin{dfn}\label{dfn:Lt}
Let $r$ be a positive integer with $r \geq 4$. In the graph $I_{r-1} \vee K_3$, 
let $Y$ be the stable set of size $r-1$ corresponding to $V(I_{r-1})$, and let $X$ be the three vertices in $K_3$.
Let $L$ be a connected graph obtained from $I_{r-1} \vee K_3$ by deleting the edges of a matching of size three between $X$ and $Y$.

For every positive integer $t$, let $L_t=L \wedge_t Y$. 

\end{dfn}
Note that $L$ exists as $r \geq 4$.
Also, $L_t$ has $(r-1) + 3t$ vertices and $3(r-1)t$ edges.

\begin{lemma}\label{claim:UB3}
Let $r$ be an integer with $r\geq 4$.  
For every $n \in {\mathbb N}$, let $G_n$ be the $n$-vertex graph obtained from $L_{\lfloor (n-r+1)/3 \rfloor}$ by adding isolated vertices to make $G_n$ have $n$ vertices. 
Let $p: {\mathbb N} \rightarrow [0,1]$ be a function such that $\lim_{n \rightarrow \infty} n^{-1/(3r-3)}/p(n)=0$.
Then $\lim_{n \rightarrow \infty} \Pr(G_n(p) \in \D_r) = \lim_{n \rightarrow \infty} \Pr(G_n(p) \in \chi_r^\ell) = 0$.
\end{lemma}

\begin{proof}
By Lemma \ref{prob general}, $\lim_{n \rightarrow \infty} \Pr(L_{r^{r-1}} \subseteq G_n(p)) = 1$.
Since $L_{r^{r-1}}$ has minimum degree at least $r$, $\lim_{n \rightarrow \infty} \Pr(G_n(p) \in \D_r) = 0$.

Now we show that $L_{r^{r-1}}$ is not $r$-choosable.
Note that it implies that $\lim_{n \rightarrow \infty} \Pr(G_n(p) \in \chi_r^\ell) = 0$. We will construct a list of $r$ colors for each vertex $v$, denoted as $L_v$.
Denote $Y=\{y_1,y_2,...,y_{r-1}\}$.
For each $i$ with $1 \leq i \leq r-1$, define $L_{y_i}=\{ri+j: 0 \leq j \leq r-1\}$. Thus the color lists of $y_i$ and $y_j$ are disjoint for any $i \neq j$.
Let $C_1,C_2,...,C_{r^{r-1}}$ be the $r^{r-1}$ copies of $V(L)-Y$ in $L_{r^{r-1}}$.
For each $i$ with $1 \leq i \leq r^{r-1}$, let $S_i$ be a set of size $r-1$ such that $|S_i \cap L_{v_j}|=1$ for every $1 \leq j \leq r-1$, and $S_k \neq S_{k'}$ for distinct $k,k' \in [r^{r-1}]$. 
This is possible since there are $r^{r-1}$ ways to pick exactly one element from each of $L_{y_i}$ for $1 \leq i \leq r-1$.
For each $i$ with $1 \leq i \leq r^{r-1}$ and each vertex $v$ in $C_i$, define $L_v=\{-1,-2\} \cup (S_i-L_{y_v})$ where $y_v$ is the vertex in $\{y_1,y_2,...,y_{r-1}\}$ such that $v$ is not adjacent in $L_{r^{r-1}}$ to $y_v$.
Note that each $L_v$ has size $r$.
Then it is easy to see that $L_{r^{r-1}}$ is not colorable with respect to $(L_v: v \in V(L_{r^{r-1}}))$.
\end{proof}

The following corollary provides an upper bound for threshold probabilities. 

\begin{corollary} \label{upper bound threshold}
Let $r \geq 2$ be an integer and let $w$ be an integer with $r \geq w \geq 0$.
Let $\G$ be a monotone class of graphs.
Then the following hold.
	\begin{enumerate}
		\item If there exists $n_0 \in {\mathbb N}$ such that $\{K_{r,n-r}: n \geq n_0\} \subseteq \G$, then $p_{\G}^{\D_r} = O(n^{-1/r})$, $p_{\G}^{\chi_r^\ell} = O(n^{-1/r})$ and $p_{\G}^{\R_r} = O(n^{-1/r})$.
		\item If there exists $n_0 \in {\mathbb N}$ such that $\{I_{r-w} \vee tK_{w+1}: t \geq n_0\} \subseteq \G$, then the following hold.
			\begin{enumerate}
				\item $p_{\G}^{\D_r} = O(n^{-1/q})$ and $p_{\G}^{\chi_r^\ell} = O(n^{-1/q})$, where $q=(w+1)r-{w+1 \choose 2}$.
				\item If $r \neq w$, then $p_{\G}^{\chi_{w+1}} = O(n^{-1/q})$, where $q=(w+1)r-{w+1 \choose 2}$.
				\item If $r$ is divisible by $w+1$, then $p_{\G}^{\R_r} = O(n^{-1/q})$, where $q=(w+1)r-{w+1 \choose 2}$.
			\end{enumerate}
		\item If $r \geq 4$ and there exists $n_0 \in {\mathbb N}$ such that $\{L_t: t \geq n_0\} \subseteq \G$, then $p_{\G}^{\D_r}=O(n^{-1/(3r-3)})$ and $p_{\G}^{\chi_r^\ell}=O(n^{-1/(3r-3)})$.
		\item Let $r,r'$ be an integers with $r \geq 2$ and $r' \leq r$ and let $s$ be a nonnegative integer.
		Let $F_0$ be a connected graph and let $F \in {\mathcal F}(I_{r'},F_0,r)$ be with type $s$. 
		Let $Z$ be the heart of $F$.
		If there exists $n_0 \in {\mathbb N}$ such that $\{F \wedge_t Z: t \geq n_0\} \subseteq \G$, then $p_{\G}^{\D_r}=O(n^{-1/s})$.
	\end{enumerate}
\end{corollary}

\begin{proof}
Statements 2, 3 and 4 of this corollary immediately follows from Lemmas \ref{claim:UB2}, \ref{claim:UB3} and \ref{claim:UBd}, respectively.
Statement 1 of this corollary following from Statement 2 by taking $w=0$.
\end{proof}

\section{Neighbors of low degree vertices} \label{sec:islands}

In this section we prove Lemma \ref{stronger better path island shallow}, which gives structural information for graphs with no dense shallow minors and is a generalization of the main lemma in the work of Ossona de Mendez, Oum, and Wood \cite{oow}, where they used the lemma to study defective coloring for a broader class of graphs.  
We refer interested readers to \cite{oow} for details.
Lemma \ref{stronger better path island shallow} might be of independent interests beyond this paper.
We will state the motivation of this lemma in Section \ref{sec:proof-of-C} when we are about to apply it.

The {\it average degree} of a graph $G$ is $\frac{2\lvert E(G) \rvert}{\lvert V(G) \rvert}$.
The {\it maximum average degree} of a graph $G$ is $\max_H \frac{2\lvert E(H) \rvert}{\lvert V(H) \rvert}$, where the maximum is over all subgraphs $H$ of $G$.
The following lemma can be found in \cite[Lemma 18]{Wood-count-clique} or in the proof of \cite[Theorem 1.1]{FW-count-clique}.

\begin{lemma}[{\cite[Lemma 18]{Wood-count-clique},\cite{FW-count-clique}}] \label{number cliques}
Let $r$ be a positive integer and let $k$ be a positive real number.
If $G$ is a graph of maximum average degree at most $k$, then $G$ contains at most ${k \choose r-1}\lvert V(G) \rvert$ cliques of size $r$.
\end{lemma}

For every nonnegative integer $\ell$, we say that a graph $G$ is an {\it $\ell$-subdivision} of a graph $H$ if it can be obtained from $H$ by subdividing each edge of $H$ exactly $\ell$ times.
That is, $G$ can be obtained from $H$ by replacing each edge of $H$ by a path of length $\ell+1$, where those paths are pairwise internally disjoint.
For a set $S$ of nonnegative integers, we say a graph $G$ is an {\it $S$-subdivision} of $H$ if for every $e \in E(H)$, there exists $s_e \in S$ such that $G$ can be obtained from $H$ by subdividing each edge $e$ of $H$ exactly $s_e$ times. 
Thus an $\ell$-subdivision is the same as an $\{ \ell \}$-subdivision. 
For every nonnegative integer $\ell$, a graph $G$ is a {\it $(\leq \ell)$-subdivision} of $H$ if it is an $([\ell] \cup \{0\})$-subdivision of $H$. 

The {\it radius} of a graph $G$ is the minimum $k$ such that there exists a vertex $v$ of $G$ such that every vertex of $G$ has distance from $v$ at most $k$.
Let $\ell \in \mathbb{Z} \cup \{\infty\}$.
We say that a graph $G$ contains a graph $H$ as an {\it $\ell$-shallow minor} if $H$ can be obtained from a subgraph $G'$ of $G$ by contracting pairwise disjoint connected subgraphs of $G'$ of radius at most $\ell$.
In other words, every branch set of an $\ell$-shallow minor is a connected subgraph of radius at most $\ell$.
Note that $G$ contains $H$ as an $\infty$-shallow minor if and only if $G$ contains $H$ as a minor; $G$ contains $H$ as a $0$-shallow minor if and only if $H$ is a subgraph of $G$.

The next concept is important in our proof.
\begin{dfn}\label{dfn:adherent}
For a graph $G$, a subset $Y$ of $V(G)$, and an integer $r$, we say a subgraph $H$ of $G$ is {\rm $r$-adherent to $Y$} if $V(H) \cap Y = \emptyset$ and $\lvert N_G(V(H)) \cap Y \rvert \geq r$.
\end{dfn}

The proof of the following lemma is inspired by the proof of \cite[Lemma 2.2]{oow}.

\begin{lemma} \label{r-adherent size}
For any $r,t \in {\mathbb N}$ and positive real number $k'$, there exists a real number $\alpha>0$ such that for every $\ell \in {\mathbb Z} \cup \{\infty\}$, for every graph $G$, for every $Y \subseteq V(G)$, and for every collection $\C$ of disjoint connected subgraphs of $G-Y$, if each member of $\C$ is $r$-adherent to $Y$ and of radius at most $\ell$, then either 
	\begin{enumerate}
		\item there exists a graph $H$ of average degree greater than $k'$ such that $G$ contains a subgraph isomorphic to a $[2\ell+1]$-subdivision of $H$, 
		\item $G$ contains $K_{r,t}$ as an $\ell$-shallow minor, or 
		\item $\lvert \C \rvert \leq \alpha \lvert Y \rvert$.
	\end{enumerate}
\end{lemma}

\begin{proof}
Let $r,t \in {\mathbb N}$ and let $k'$ be a positive real number.
Define $\alpha=(t-1){k' \choose r-1} + k'/2$.

Let $\ell \in {\mathbb Z} \cup \{\infty\}$, let $G$ be a graph and $Y \subseteq V(G)$.
Let $\C$ be a collection of disjoint connected subgraphs of $G-Y$ such that each member of $\C$ is $r$-adherent to $Y$ and of radius at most $\ell$.

Assume that for every graph $H$, if $G$ contains some subgraph isomorphic to a $[2\ell+1]$-subdivision of $H$, then the average degree of $H$ is at most $k'$.
Assume that $G$ does not contain $K_{r,t}$ as an $\ell$-shallow minor.
We shall show that Statement 3 of this lemma holds.

Let $G'$ be the graph obtained from $G$ by contracting each member of $\C$ into a vertex.
Since each member of $\C$ is a subgraph of $G$ of radius at most $\ell$, $G'$ is an $\ell$-shallow minor of $G$.
In addition, $Y \subseteq V(G')$ since each member of $\C$ is disjoint from $Y$. 
Let $Z=V(G')-V(G)$.
(That is, $Z$ is the set of the vertices of $G'$ obtained by contracting members of $\C$.)
Define $G''$ to be the graph obtained from $G' - E(G'[Y])$ by repeatedly picking a vertex $v$ in $Z$ that is adjacent in $G'$ to a pair of nonadjacent vertices $u,w$ in $(G' - E(G'[Y]))[Y]$, deleting $v$, and adding an edge $uw$, until for any remaining vertex in $Z$, its neighbors in $Y$ form a clique.

Let $H=G''[Y]$.
So some subgraph $H'$ of $G'$ is isomorphic to a $1$-subdivision of $H$.
Together with the fact that $G$ contains $H'$ as an $\ell$-shallow minor and $V(H) = Y \subseteq V(G)-Z$, we know $G$ contains a subgraph isomorphic to a $[2\ell+1]$-subdivision of $H$.
This implies that for every subgraph $L$ of $H$, $G$ contains a subgraph isomorphic to a $[2\ell+1]$-subdivision of $L$.
So the average degree of any subgraph of $H$ is at most $k'$ by our assumption.
Hence there are at most 
\begin{equation}
     {k' \choose r-1} \lvert V(H) \rvert = {k' \choose r-1} \lvert Y \rvert \label{eq:1}
\end{equation} cliques of size $r$ in $H$ by Lemma \ref{number cliques}.

Since $G$ contains $G'$ as an $\ell$-shallow minor, $G'$ does not contain $K_{r,t}$ as a subgraph, for otherwise $G$ contains $K_{r,t}$ as an $\ell$-shallow minor. 
This implies that for each clique $K$ in $G''[Y]$ of size $r$, $\lvert \{z \in Z \cap V(G''): K \subseteq N_{G''}(z)\} \rvert \leq t-1$.
In addition, for every $z \in Z \cap V(G'')$, $N_{G''}(z) \cap Y$ is a clique consisting of at least $r$ vertices in $H$ since every member of $\C$ is $r$-adherent to $Y$.
So a double counting argument applied to (\ref{eq:1}) implies
\[\lvert Z \cap V(G'') \rvert \leq (t-1) {k' \choose r-1} \lvert Y \rvert.\]

Furthermore, by the definition of $G''$, the vertices in $Z$ but not in $V(G'')$ are the ones being deleted while adding one edge in between two vertices in $Y$. Thus $\lvert Z-V(G'') \rvert \leq \lvert E(G''[Y]) \rvert = \lvert E(H) \rvert \leq k'\lvert Y \rvert/2$.
So $\lvert \C \rvert = \lvert Z \rvert \leq (t-1){k' \choose r-1} \lvert Y \rvert + k' \lvert Y \rvert/2 = \alpha \lvert Y \rvert$.
\end{proof}

For any vertex $x$ of a graph $G$ and any (possibly negative) real number $\ell$, we denote by $N_G^{\leq \ell}[x]$ the set of all the vertices in $G$ whose distance to $x$ is at most $\ell$; in particular, $N_G^{\leq 1}[x]=N_G[x]$.

\begin{dfn}\label{dfn:span}
For a subset $Y$ of $V(G)$, $v \in V(G)-Y$ and integers $k$ and $r$, we define a \emph{$(v,Y,k,r)$-span (in $G$)} to be a connected subgraph $H$ of $G-Y$ containing $v$ such that $\lvert Y \cap N_G(V(H)) \rvert \geq r$, and for every vertex $u$ of $H$, there exists a path in $H$ from $v$ to $u$ of length at most $k$.
\end{dfn}

Note that if $H$ is a $(v, Y,k,r)$-span in $G$, then $H$ is $r$-adherent to $Y$, and $V(H) \subseteq N_H^{\leq k}[v] \subseteq N_G^{\leq k}[v]$. 
A $(v,Y,k,r)$-span $H$ is {\it minimal} if no proper subgraph of $H$ is a $(v,Y,k,r)$-span. 

\begin{lemma}\label{claim:number of vtx in minimal span} 
Let $G$ be a graph, $Y \subseteq V(G)$ and $k,r$ be integers.
Then every minimal $(v,Y,k,r)$-span is a subgraph of a union of at most $r$ paths in $G-Y$ where each path starts from $v$ and is of length at most $k$. 
In particular, every minimal $(v,Y,k,r)$-span contains at most $kr+1$ vertices.
\end{lemma}

\begin{proof}
Let $H$ be a minimal $(v,Y,k,r)$-span.
Since $H$ is a $(v,Y,k,r)$-span, $\lvert N_G(V(H)) \cap Y \rvert \geq r$.
So there exists a subset $S=\{v_1,v_2,...,v_{\lvert S \rvert}\}$ of $V(H)$ with $\lvert S \rvert \leq r$ such that $|Y \cap N_G(S)| \geq r$ and $Y \cap N_G(v_{i+1})-\bigcup_{j=1}^i N_G(v_j) \neq \emptyset$ for each $i$ with $0 \leq i \leq \lvert S \rvert-1$.
Since $H$ is a $(v,Y,k,r)$-span, for every $u \in S$, there exists a path $P_u$ in $H \subseteq G-Y$ from $v$ to $u$ of length at most $k$.
Hence $\bigcup_{u \in S}P_u$ is a $(v,Y,k,r)$-span and is a subgraph of $H$.
By the minimality of $H$, $H=\bigcup_{u \in S}P_u$.
Therefore, $H$ is a union of $\lvert S \rvert \leq r$ paths in $G-Y$ where each path starts from $v$ and is of length at most $k$.
Note that $\lvert V(P_u)-\{v\} \rvert \leq k$, so $\lvert V(H) \rvert \leq 1+rk$.
\end{proof}

Now we are ready to state and prove Lemma \ref{stronger better path island shallow}.

\begin{lemma} \label{stronger better path island shallow}
For any $r,t \in {\mathbb N}$, nonnegative integer $\ell$, positive real numbers $k,k'$, and nonnegative real number $\beta$, there exists a real number $d$ such that for every graph $G$, either
	\begin{enumerate}
		\item the average degree of $G$ is greater than $k$,
		\item there exists a graph $H$ of average degree greater than $k'$ such that some subgraph of $G$ is isomorphic to a $[2\ell+1]$-subdivision of $H$,
		\item $G$ contains $K_{r,t}$ as an $\ell$-shallow minor, or 
		\item there exist $X,Z \subseteq V(G)$ with $Z \subseteq X$ and $\lvert Z \rvert > \beta\lvert V(G)-X \rvert$ such that
			\begin{enumerate}
				\item every vertex in $X$ has degree at most $d$ in $G$, 
				\item for any distinct $z,z' \in Z$, the distance in $G[X]$ between $z,z'$ is at least $\ell+1$, 
				\item for every $z \in Z$ and $u \in X$ whose distance from $z$ in $G[X]$ is at most $\ell$, $\lvert N_G(u) -X \rvert  \leq r-1$, and 
				\item $\lvert N_G(N_{G[X]}^{\leq \ell-1}[z])-X \rvert \leq r-1$ for every $z \in Z$.
			\end{enumerate}
	\end{enumerate}
\end{lemma}

\begin{proof}
Let $r,t \in {\mathbb N}$, $\ell$ be a nonnegative integer, $k,k'$ be positive real numbers, and $\beta$ be a nonnegative real number.
Let $\alpha$ be the one in Lemma \ref{r-adherent size} by taking $r=r$, $t=t$ and $k'=k'$.
Let $\gamma=\beta+(\ell r+1)\alpha$. 
Define $d=(1+(1+\gamma)^{(r+1)^\ell})k$.

Let $G$ be a graph.
Assume that the average degree of $G$ is at most $k$, and assume that there exists no graph $H$ of average degree greater than $k'$ such that some subgraph of $G$ is isomorphic to a $[2\ell+1]$-subdivision of $H$.
Assume that $G$ does not contain $K_{r,t}$ as an $\ell$-shallow minor.

For any $Y \subseteq V(G)$ and $v \in V(G)-Y$, we define the {\it $Y$-correlation of $v$} to be the sequence $(a_0,a_1,a_2,...,a_{\ell-1})$, where $a_i = \lvert Y \cap N_G(N_{G-Y}^{\leq i}[v]) \rvert$ for each $i$ with $0 \leq i \leq \ell-1$.
Note that if some entry of the $Y$-correlation of a vertex $v$ is at least $r$, then there exists a $(v,Y,\ell,r)$-span.
Observe that the $Y$-correlation of $v$ is an empty sequence when $\ell=0$.

Let $X_0$ be the set of the vertices of $G$ of degree at most $d$, and let $Y_0=V(G)-X_0$.
We use the following iterative procedure for each step $i \geq 0$ to define the vertex partition $V(G) = X_i \cup Y_i$. 
For each $i \geq 0$, we define the following.
	\begin{itemize}
		\item Define $\C_i$ to be a maximal collection of pairwise disjoint subgraphs of $G[X_i]$, where each member of $\C_i$ is a minimal $(v,Y_i,\ell,r)$-span for some vertex $v \in X_i$ satisfying that if $\ell \geq 1$, then $\lvert Y_i \cap N_G(N_{G-Y_i}^{\leq \ell-1}[v]) \rvert \geq r$. 
		\item $D_i = \bigcup_{H \in \C_i}V(H)$.
		\item $Z_i$ is a maximal subset of $X_i-D_i$ such that
			\begin{itemize}
				\item for any two distinct vertices in $Z_i$, their distance in $G[X_i]$ is at least $\ell+1$, and 
				\item for every $z \in Z_i$, $N_{G[X_i-D_i]}^{\leq \ell-1}[z] \cap N_G(D_i)=\emptyset$.
			\end{itemize}
		\item $X_{i+1} = X_i - (Z_i \cup D_i)$.
		\item $Y_{i+1} = Y_{i} \cup Z_i \cup D_i$.
	\end{itemize}
Note that $\{X_i, Y_i\}$ is a partition of $V(G)$ for every $i \geq 0$, where $X_i$ or $Y_i$ is possibly empty.

\begin{claim} \label{claim:1} For every nonnegative integer $i$ and $z \in Z_i$, 
	\begin{itemize}
		\item $\lvert N_G(N_{G[X_i]}^{\leq \ell-1}[z])-X_i \rvert \leq r-1$, and 
		\item if $u$ is a vertex in $X_i$ such that the distance in $G[X_i]$ from $z$ to $u$ is at most $\ell$, then $\lvert N_G(u)-X_i \rvert \leq r-1$.
	\end{itemize}
\end{claim}	

\noindent{\bf Proof of Claim \ref{claim:1}:}
We first suppose $N_{G[X_i]}^{\leq \ell-1}[z] \neq N_{G[X_i-D_i]}^{\leq \ell-1}[z]$.
So there exists $v' \in N_{G[X_i]}^{\leq \ell-1}[z]-N_{G[X_i-D_i]}^{\leq \ell-1}[z]$.
This means that there exists a path in $G[X_i]$ of length at most $\ell-1$ from $z$ to $v'$ intersecting $D_i$.
Hence there exists $v \in X_i$ such that there exists a path $P_v$ in $G[X_i]$ of length at most $\ell-1$ from $z$ to $v$ intersecting $D_i$.
We may assume that $v$ is chosen such that $\lvert V(P_v) \rvert$ is as small as possible.
Hence $P_v$ is internally disjoint from $D_i$.
Since $z \in Z_i$, $z \not \in D_i$.
So $v \in D_i$ and $P_v$ contains at least two vertices.
Let $v''$ be the neighbor of $v$ in $P_v$.
Since $P_v$ is internally disjoint from $D_i$ and $z \not \in D_i$, it follows that $v'' \in N_{G[X_i-D_i]}^{\leq \ell-1}[z] \cap N_G(v) \subseteq N_{G[X_i-D_i]}^{\leq \ell-1}[z] \cap N_G(D_i)$.
However, since $z \in Z_i$, by the definition of $Z_i$,  $N_{G[X_i-D_i]}^{\leq \ell-1}[z] \cap N_G(D_i)=\emptyset$, a contradiction.

Hence 
\begin{equation} N_{G[X_i]}^{\leq \ell-1}[z] = N_{G[X_i-D_i]}^{\leq \ell-1}[z].  \label{eqn:nbhdequal}
\end{equation}
So $N_{G[X_i]}^{\leq \ell-1}[z] \cap N_G(D_i)=\emptyset$.
Since $z \not \in D_i$, $N_{G[X_i]}^{\leq \ell-1}[z] \cap D_i=\emptyset$.

Suppose $\lvert N_G(N_{G[X_i]}^{\leq \ell-1}[z])-X_i \rvert \geq r$.
Since $\{X_i,Y_i\}$ is a partition of $V(G)$, $\lvert N_G(N_{G-Y_i}^{\leq \ell-1}[z]) \cap Y_i \rvert = \lvert N_G(N_{G[X_i]}^{\leq \ell-1}[z]) \cap Y_i \rvert \geq r$. 
Therefore $G[N_{G - Y_i}^{\leq \ell -1}[z]] = G[N_{G[X_i]}^{\leq \ell-1}[z]]$ is a $(z, Y_i, \ell, r)$-span in $G$; it is disjoint from members in $\mathcal{C}_i$ since $N_{G[X_i]}^{\leq \ell-1}[z] \cap D_i=\emptyset$.
Hence there exists a minimal $(z,Y_i,\ell,r)$-span $H'$ in $G$ such that $V(H') \subseteq N_{G[X_i]}^{\leq \ell-1}[z]$.  
But $V(H') \cap D_i \subseteq N_{G[X_i]}^{\leq \ell-1}[z] \cap D_i=\emptyset$, contradicting the maximality of $\C_i$.

Therefore, $\lvert N_G(N_{G[X_i]}^{\leq \ell-1}[z])-X_i \rvert \leq r-1$.

Let $u$ be a vertex in $X_i$ such that the distance in $G[X_i]$ from $z$ to $u$ is at most $\ell$.
Suppose $u \in D_i$.
Since $z \not \in D_i$, there exists a vertex $u' \in X_i \cap N_G(u)$ such that the distance in $G[X_i]$ from $z$ to $u'$ is at most $\ell-1$.
So $u' \in N_{G[X_i]}^{\leq \ell-1}[z] \cap N_G(u) \subseteq N_{G[X_i]}^{\leq \ell-1}[z] \cap N_G(D_i) = N_{G[X_i-D_i]}^{\leq \ell-1}[z] \cap N_G(D_i)$ by (\ref{eqn:nbhdequal}).
But $z \in Z_i$, so $N_{G[X_i-D_i]}^{\leq \ell-1}[z] \cap N_G(D_i)=\emptyset$, a contradiction.

Hence $u \not \in D_i$.
If $\lvert N_G(u)-X_i \rvert \geq r$, then the graph consisting of the vertex $u$ is a minimal $(u,Y_{i},0,r)$-span (and hence a minimal $(u,Y_{i},\ell,r)$-span), so $u$ is contained in $D_i$ by the maximality of $\C_{i}$, a contradiction.
So $\lvert N_G(u)-X_i \rvert \leq r-1$.
$\Box$

\medskip

If there exists a nonnegative integer $i^*$ such that $\lvert Z_{i^*} \rvert > \beta \lvert V(G)-X_{i^*} \rvert$, then by defining $X=X_{i^*}$ and $Z=Z_{i^*}$, we know that $\lvert Z \rvert > \beta \lvert V(G)-X \rvert$, and every vertex in $X \subseteq X_0$ has degree at most $d$ in $G$; statements 4(b)-4(d) follow from the definition of $Z_{i^*}$ and Claim \ref{claim:1}, so Statement 4 holds.

So we may assume that $\lvert Z_i \rvert \leq \beta \lvert V(G)-X_i \rvert=\beta \lvert Y_i \rvert$ for every nonnegative integer $i$. 

Since $X_{i+1} \subseteq X_i$, we have for any $v \in X_{i+1}$ and any $0 \leq j \leq \ell-1$,
\begin{equation}
  N_{G[X_{i+1}]}^{\leq j}[v] \subseteq N_{G[X_i]}^{\leq j}[v]. \label{eq:nbhdinclusion}
\end{equation}
Note that for every $v \in X_{i+1}$, if there exists an integer $j$ with $0 \leq j \leq \ell-1$ such that $N_{G[X_{i+1}]}^{\leq j}[v] = N_{G[X_i]}^{\leq j}[v]$, then it is easy to see that for every $u \in N_{G[X_{i+1}]}^{\leq j}[v]$, the distance between $u$ and $v$ in $G[X_{i+1}]$ is the same as the distance between $u$ and $v$ in $G[X_i]$ by induction on the distance between $u$ and $v$ in $G[X_{i+1}]$, and hence for every integer $j'$ with $0 \leq j' \leq j$, $N_{G[X_{i+1}]}^{\leq j'}[v] = N_{G[X_i]}^{\leq j'}[v]$.

\begin{claim} \label{claim:2-1}
Let $i$ be a nonnegative integer, and let $v$ be a vertex in $X_{i+1}$.
Denote the $Y_{i+1}$-correlation of $v$ by $(a_0,a_1,...,a_{\ell-1})$, and denote the $Y_i$-correlation of $v$ by $(b_0,b_1,...,b_{\ell-1})$. 
If there exists an integer $k \leq \ell -1$ such that $N_{G[X_{i+1}]}^{\leq k}[v] \subsetneq N_{G[X_{i}]}^{\leq k}[v]$, and $N_{G[X_{i+1}]}^{\leq j}[v]= N_{G[X_{i}]}^{\leq j}[v]$ for every $0 \leq j < k$, then $(a_0, a_1, \dots, a_{k-1})$ is strictly greater than $(b_0, b_1, \dots, b_{k-1})$ in the lexicographic order.
\end{claim}

\noindent{\bf Proof of Claim \ref{claim:2-1}:}
Since $N_{G[X_{i+1}]}^{\leq k}[v] \subsetneq N_{G[X_{i}]}^{\leq k}[v]$, $k \geq 1$.

Since $X_{i+1} \subseteq X_i$, $Y_i \subseteq Y_{i+1}$.
By the condition of this claim, $N_G(N_{G[X_i]}^{\leq j}[v]) = N_G(N_{G[X_{i+1}]}^{\leq j}[v])$ for every integer $j$ with $0 \leq j \leq k-1$. 
So for every $j$ with $0 \leq j \leq k-1$, $Y_i \cap N_G(N_{G[X_i]}^{\leq j}[v])  \subseteq Y_{i+1} \cap N_G(N_{G[X_{i+1}]}^{\leq j}[v])$, and hence $a_j \geq b_j$.

Let $u$ be an arbitrary vertex in $N_{G[X_{i}]}^{\leq k}[v] - N_{G[X_{i+1}]}^{\leq k}[v]$. 
By the condition of this claim,  $N_{G[X_i]}^{\leq k-1}[v] = N_{G[X_{i+1}]}^{\leq k-1}[v]$. 
So 
\begin{equation}
u \in N_{G[X_{i}]}^{\leq k}[v] = N_{G[X_i]}[N_{G[X_i]}^{\leq k-1}[v]] =N_{G[X_i]}[N_{G[X_{i+1}]}^{\leq k-1}[v]]. \label{eq:inter1}
\end{equation}
Hence the distance between $u$ and $v$ in $G[X_{i+1} \cup \{u\}]$ is at most $k$. 
Since $u \not \in N_{G[X_{i+1}]}^{\leq k}[v]$, $u \in X_i-X_{i+1} = Y_{i+1}-Y_i$.
Together with (\ref{eq:inter1}), $u \in Y_{i+1} \cap N_G(N_{G[X_{i+1}]}^{\leq k-1}[v]) - Y_i$.

Since $N_G(N_{G[X_i]}^{\leq k-1}[v]) = N_G(N_{G[X_{i+1}]}^{\leq k-1}[v])$, we have $Y_i \cap N_G(N_{G[X_i]}^{\leq k-1}[v])  \subseteq Y_{i+1} \cap N_G(N_{G[X_{i+1}]}^{\leq k-1}[v])$.
Recall that $a_{k-1} = \lvert Y_{i+1} \cap N_G(N_{G[X_{i+1}]}^{\leq k-1}[v]) \rvert$ and $b_{k-1}= \lvert Y_i \cap N_G(N_{G[X_i]}^{\leq k-1}[v]) \rvert$.
Since $u \in Y_{i+1} \cap N_G(N_{G[X_{i+1}]}^{\leq k-1}[v]) - Y_i$, $a_{k-1}>b_{k-1}$.
Therefore, $(a_0,a_1,...,a_{k-1})>(b_0,b_1,...,b_{k-1})$. 
$\Box$

\begin{claim}\label{claim:2-2}
Let $i$ be a nonnegative integer, and let $v$ be a vertex in $X_{i+1}$.
Denote the $Y_{i+1}$-correlation of $v$ by $(a_0,a_1,...,a_{\ell-1})$, and denote the $Y_i$-correlation of $v$ by $(b_0,b_1,...,b_{\ell-1})$. 
If $\ell \geq 1$ and $N_{G[X_{i+1}]}^{\leq j}[v]= N_{G[X_{i}]}^{\leq j}[v]$ for every integer $j$ with $0 \leq j \leq \ell -1$, then $(a_0, a_1, \dots, a_{\ell-1})$ is strictly greater than $(b_0, b_1, \dots, b_{\ell-1})$ in the lexicographic order.
\end{claim}

\noindent{\bf Proof of Claim \ref{claim:2-2}:}
Since $Y_{i+1} \supseteq Y_i$, and for every integer $j$ with $0 \leq j \leq \ell-1$, $N_{G[X_{i+1}]}^{\leq j}[v] = N_{G[X_i]}^{\leq j}[v]$, we have $a_j \geq b_j$ for every $j$ with $0 \leq j \leq \ell-1$. 

Since $v \in X_{i+1}$, $v \not \in Z_i \cup D_i$ by the definition of $X_{i+1}$.
Assume that there exists $v' \in Y_{i+1}-Y_i = X_i - X_{i+1}$ such that the distance in $G[X_i]$ between $v$ and $v'$ is $\ell'$ for some $0 \leq \ell'\leq \ell$, then $v' \in \left(Y_{i+1} \cap N_G(N_{G[X_{i+1}]}^{\leq \ell'-1}[v])\right) - \left( Y_{i} \cap N_G(N_{G[X_{i}]}^{\leq \ell'-1}[v])\right)$, so $a_{\ell'-1}>b_{\ell'-1}$.
Recall that $a_j \geq b_j$ for every $j$ with $0 \leq j \leq \ell-1$, so $(a_0,a_1,...,a_{\ell-1}) > (b_0,b_1,...,b_{\ell-1})$, and hence the claim follows.

Hence we may assume that there is no $v' \in Y_{i+1}-Y_i = X_i - X_{i+1}$ such that the distance in $G[X_i]$ between $v$ and $v'$ is at most $\ell$.
Equivalently, $N_{G[X_i]}^{\leq \ell}[v] \cap Y_{i+1}-Y_i=\emptyset$.
Since $D_i \cup Z_i= Y_{i+1}  - Y_i$, $N_{G[X_i]}^{\leq \ell}[v] \cap (D_i \cup Z_i)=\emptyset$. 
Therefore any vertex in $X_i$ whose distance to $v$ in $G[X_i]$ is at most $\ell$ is not in $D_i \cup Z_i$. 
In other words, for any $j \leq \ell$, we have $N_{G[X_i]}^{\leq j}[v]= 
N_{G[X_i - (D_i \cup Z_i)]}^{\leq j}[v] $. Thus $N_{G[X_{i+1}]}^{\leq \ell-1}[v]=  N_{G[X_i - (D_i \cup Z_i)]}^{\leq \ell-1}[v] =N_{G[X_{i}-D_i]}^{\leq \ell-1}[v]$. 
Since $v \in X_{i+1}$, $v \not \in Z_i$.
Since $N_{G[X_i]}^{\leq \ell}[v] \cap (D_i \cup Z_i)=\emptyset$, by the maximality of $Z_i$, $N_{G[X_{i}-D_i]}^{\leq \ell-1}[v] \cap N_G(D_i) \neq \emptyset$.
So there exists $x \in N_{G[X_{i}-D_i]}^{\leq \ell-1}[v] \cap N_G(D_i) = N_{G[X_{i+1}]}^{\leq \ell-1}[v] \cap N_G(D_i)$.
Hence there exists $y \in D_i \cap N_G(x)$.
Since $y \in N_G(x)$, $y \in N_G(N_{G[X_{i+1}]}^{\leq \ell-1}[v])$.
Since $y \in D_i$, $y \in Y_{i+1}-Y_i$.
So $(Y_{i+1}-Y_i) \cap N_G(N_{G[X_{i+1}]}^{\leq \ell-1}[v]) \neq \emptyset$.
Therefore, $a_{j^*}>b_{j^*}$ for some $j^*$ with $0 \leq j^* \leq \ell-1$.
Recall that $a_j \geq b_j$ for every $j$ with $0 \leq j \leq \ell-1$.
So $(a_0,a_1,...,a_{\ell-1})>(b_0,b_1,...,b_{\ell-1})$. \hfill
$\Box$

\begin{claim}\label{claim:3} 
Let $i$ be a nonnegative integer, and let $v$ be a vertex in $X_{i+1}$.
Denote the $Y_{i+1}$-correlation of $v$ by $(a_0,a_1,...,a_{\ell-1})$, and denote the $Y_i$-correlation of $v$ by $(b_0,b_1,...,b_{\ell-1})$. 
If there exists $k$ with $0 \leq k \leq \ell-1$ such that $b_k \geq r$ and $b_j<r$ for every $0 \leq j \leq k-1$, then $(a_0,a_1,...,a_{k-1})$ is strictly greater than $(b_0,b_1,...,b_{k-1})$ in the lexicographic order.
\end{claim}

\noindent{\bf Proof of Claim \ref{claim:3}:}
If there exists an integer $j^*$ with $0 \leq j^* \leq k$ such that $N_{G[X_{i+1}]}^{\leq j^*}[v] \subsetneq N_{G[X_{i}]}^{\leq j^*}[v]$, then by Claim \ref{claim:2-1}, $(a_0,a_1,...,a_{k-1})$ is strictly greater than $(b_0,b_1,...,b_{k-1})$.
So by (\ref{eq:nbhdinclusion}), we may assume that $N_{G[X_{i+1}]}^{\leq j}[v] = N_{G[X_{i}]}^{\leq j}[v]$ for every $j$ with $0 \leq j \leq k$.
In particular, $N_{G[X_{i+1}]}^{\leq k}[v] = N_{G[X_{i}]}^{\leq k}[v]$. 

Since $Y_i \subseteq Y_{i+1}$, for any $j$ with $0 \leq j \leq k-1$, $N_G( N_{G[X_i]}^{\leq j}[v]) \cap Y_i = N_G( N_{G[X_{i+1}]}^{\leq j}[v]) \cap Y_i \subseteq  N_G( N_{G[X_{i+1}]}^{\leq j}[v] ) \cap Y_{i+1}$. 
So $a_j \geq b_j$ for every $j$ with $0 \leq j \leq k-1$.

Since $b_k \geq r$, there exists a minimal $(v,Y_i,k,r)$-span $Q$ in $G[X_i]$ with $V(Q) \subseteq N_{G[X_i]}^{\leq k}[v]$.
Since $v \in X_{i+1}$ and $v \in V(Q)$, we have $Q \not \in \C_i$.
By the maximality of $\C_i$, there exists a member $M$ of $\C_i$ intersecting $Q$. 
Since $v \in X_{i+1}$, $v \not \in V(M)$.
So $\emptyset \neq V(M) \cap V(Q) \subseteq V(M) \cap N_{G[X_i]}^{\leq k}[v]-\{v\}$.
Hence there exists an integer $k'$ with $1 \leq k' \leq k$ such that $V(M) \cap N_{G[X_i]}^{\leq k'}[v] - N_{G[X_i]}^{\leq k'-1}[v] \neq \emptyset$.
Recall that $N_{G[X_{i+1}]}^{\leq j}[v] = N_{G[X_{i}]}^{\leq j}[v]$ for every $j$ with $0 \leq j \leq k$, so $\emptyset \neq V(M) \cap N_{G[X_i]}^{\leq k'}[v]-N_{G[X_i]}^{\leq k'-1}[v] = V(M) \cap N_{G[X_{i+1}]}^{\leq k'}[v]-N_{G[X_{i+1}]}^{\leq k'-1}[v] \subseteq V(M) \cap N_G(N_{G[X_{i+1}]}^{\leq k'-1}[v])$.
Together with the fact that $V(M) \subseteq D_i \subseteq Y_{i+1}-Y_i$, we have $\emptyset \neq N_G(N_{G[X_{i+1}]}^{\leq k'-1}[v] ) \cap Y_{i+1} - (N_G(N_{G[X_i]}^{\leq k'-1}[v]) \cap Y_i)$.
Hence $a_{k'-1}>b_{k'-1}$.
Therefore, $(a_0,a_1,...,a_{k-1})$ is strictly greater than $(b_0,b_1,...,b_{k-1})$.
$\Box$

\begin{claim}\label{claim:4}
$X_0 \subseteq Y_{(r+1)^{\ell}}$.
\end{claim}

\noindent{\bf Proof of Claim \ref{claim:4}:}
We first assume $\ell=0$.
Then for every two distinct vertices in $X_0-D_0$, their distance in $G[X_0]$ is at least $1=\ell+1$.
And for every $z \in X_0-D_0$, $N_{G[X_0-D_0]}^{\leq \ell-1}[z]=N_{G[X_0-D_0]}^{\leq -1}[z] = \emptyset$.
So $Z_0=X_0-D_0$.
Hence $Z_0 \cup D_0=X_0$.
So $X_0 \subseteq Y_1 =  Y_{(r+1)^{0}}=Y_{(r+1)^{\ell}}$.

Hence we may assume $\ell \geq 1$.
Let $v \in X_0$.
We shall show that there exists a nonnegative integer $i_v$ such that $v \in Y_{i_v}$ and show $i_v \leq (r+1)^\ell$. 

For each nonnegative integer $i$, if $v$ is in $X_{i}$, then let $a^{(i)} = (a_0^{(i)}, a_1^{(i)}, \dots, a_{\ell -1}^{(i)})$ be the $Y_{i}$-correlation of $v$. 
By Claims \ref{claim:2-1} and \ref{claim:2-2} and (\ref{eq:nbhdinclusion}), for every nonnegative integer $i$, if $v \in X_{i+1}$, then $a^{(i+1)} > a^{(i)}$ in the lexicographic order. 
So if $v \in X_{i+1}$, then one entry in $a^{(i)}$ will increase its value by at least one. 
By Claim \ref{claim:3}, if $v \in X_{i+1}$ and there exists $j$ with $0 \leq j\leq \ell -1$ such that the entry $a_j^{(i)} \geq r$ while $a_{j'}^{(i)} < r$ for all $0 \leq j' < j$, then $a^{(i+1)}_{j'}>a^{(i)}_{j'}$ for some $j' < j$. 

Therefore, there exists a nonnegative integer $i_v$ with $i_v \leq r \cdot (r+1)^{\ell-1}$ such that either $v \in Y_{i_v}$ or $a_0^{(i_v)} \geq r$. 
Note that if $v \not \in Y_{i_v}$, then $a_0^{(i_v)} \geq r$, so $\lvert N_G(v) \cap Y_{i_v} \rvert \geq r$.
So when $v \not \in Y_{i_v}$, the graph consists of the vertex $v$ is a $(v,Y_{i_v},0,r)$-span, so $v$ is contained in some member of $\C_{i_v}$ by the maximality of $\C_{i_v}$, and hence $v \in Y_{i_v+1} \subseteq Y_{(r+1)^{\ell}}$.
Therefore, $X_0 \subseteq Y_{(r+1)^{\ell}}$. 
$\Box$

\medskip

Recall that we assume $\lvert Z_i \rvert \leq \beta \lvert Y_i \rvert$ for every nonnegative integer $i$.
By Lemma \ref{r-adherent size}, $\lvert \C_i \rvert \leq \alpha \lvert Y_{i} \rvert$ for every nonnegative integer $i$.
For every nonnegative integer $i$, since each member $T$ of $\C_i$ is a minimal $(v,Y_i,\ell,r)$-span, it contains at most $\ell r+1$ vertices by Lemma \ref{claim:number of vtx in minimal span}.
So for every nonnegative integer $i$, 
\begin{align*} \lvert Y_{i+1}-Y_i \rvert = & \  |Z_i| + \sum_{T \in \C_i} |V(T)| \\
\leq & \ \lvert Z_i \rvert + \lvert \C_i \rvert \cdot (\ell r+1) \leq (\beta+\alpha \cdot (\ell r+1)) \lvert Y_i \rvert = \gamma \lvert Y_i \rvert.
\end{align*}
Hence $\lvert Y_{i+1} \rvert \leq (1+\gamma)\lvert Y_i \rvert$ for every nonnegative integer $i$.
Therefore, $\lvert Y_i \rvert \leq (1+\gamma)^i \lvert Y_0 \rvert$ for every nonnegative $i$.
By Claim \ref{claim:4}, $\lvert X_0 \rvert \leq \lvert Y_{(r+1)^\ell} \rvert \leq (1+\gamma)^{(r+1)^\ell}\lvert Y_0 \rvert$.
Since $V(G)=X_0 \cup Y_0$, \[ \lvert Y_0 \rvert \geq \frac{1}{1+(1+\gamma)^{(r+1)^\ell}}\lvert V(G) \rvert.\]
Therefore, $\sum_{v \in V(G)} \deg_G(v) \geq \sum_{v \in Y_0} \deg_G(v) > d \lvert Y_0 \rvert \geq \frac{d}{1+(1+\gamma)^{(r+1)^\ell}}\lvert V(G) \rvert = k \lvert V(G) \rvert$, which implies that the average degree of $G$ is greater than $k$, a contradiction.
This proves the lemma.
\end{proof}

We remark that the main lemma in the work of Ossona de Mendez, Oum, and Wood \cite[Lemma 2.2]{oow} is implied by the case $(\ell,\beta)=(0,0)$ of Lemma \ref{stronger better path island shallow} (up to the constant $d$).

\section{Existence of good collections} \label{sec:proof-of-C}

We prove Lemmas \ref{weak general collection} and \ref{minor collection}, which will provide the correct value $q$ for Lemma \ref{lem:coverCandp}. Recall the definitions of good signatures in Definitions \ref{def:goodsig} and \ref{def:goodsig2}. 

\begin{lemma} \label{weak general collection}
For every positive integer $r$ and graph $H$, there exists a constant $c = c(r, H)>0$ such that for every $H$-minor free graph $G$, there exists a $(c,r,r)$-good signature collection for $G$.
\end{lemma}

\begin{proof}
Let $r$ be a positive integer and let $H$ be a graph.
By \cite{Mader}, there exists a real number $k$ such that every graph of average degree at least $k$ contains $H$ as a minor.
Define $c={k \choose r}$.
	
Let $G$ be an $H$-minor free graph.
Since $G$ has no $H$-minor, the average degree of $G$ is less than $k$.
So there exists a vertex $z^*$ of $G$ of degree less than $k$.
Let $Z=\{z^*\}$.
Then this lemma immediately follows from Lemma \ref{sufficient good collection} by taking $a=k$, $t=1$ and $\xi=1$.
\end{proof}

Recall that as shown in Lemma \ref{lem:coverCandp}, we want the value $q$ to be as large as possible.
Lemma \ref{weak general collection} provides such a value $q$ and will be used in the proof of our main results.
But it is not strong enough to prove our other results, and in fact it uses almost no structural information for $H$-minor free graphs.
We shall provide a better value for $q$ by looking into the structure of $H$-minor free graphs.
This is the purpose of Lemma \ref{minor collection}, and the rest of this section is dedicated to a proof of it. 

We first sketch the ideas of the proof of Lemma \ref{minor collection}.
Recall that by Lemma \ref{sufficient good collection}, to obtain a good signature collection, it suffices to find a set $Z$ of vertices of bounded maximum degree such that every subgraph with minimum degree at least $r$ containing a particular vertex in $Z$ must contain many vertices in $Z$.
Lemma \ref{stronger better path island shallow} already provides a preliminary form of such $Z$: it shows that either $G$ contains a certain graph as a minor (Statements 1-3), or there exists a set $Z$ of vertices contained in another set $X$ of vertices with bounded maximum degree such that if we visit vertices by going from any vertex $z \in Z$ along paths with bounded length in $G[X]$, then the total number of neighbors outside $X$ of those vertices we visited is small (Statement 4).
If we can show that the total number of neighbors is 0, then every subgraph with large minimum degree containing $z$ must contain many vertices in $Z$, so we are done.
However, it is too good to be true in general. 
So we shall bootstrap Lemma \ref{stronger better path island shallow} to obtain Lemma \ref{better path island shallow 2} to show that either we can ensure every subgraph with large minimum degree containing $z$ contains many vertices in $Z$ (Statement 4(b) in Lemma \ref{better path island shallow 2}), or we obtain something similar so that we can adapt the proof of Lemma \ref{sufficient good collection} to obtain a good signature collection (Statement 4(c)ii in Lemma \ref{better path island shallow 2}), or there exists an obstruction that looks like the graphs that give the upper bound of the thresholds (Statement 4(c)i in Lemma \ref{better path island shallow 2}).
Then in Lemmas \ref{subgraph one vertex collection} and \ref{minor collection}, we show how to adapt the proof of Lemma \ref{sufficient good collection} to obtain a desired good signature collection.

Before we proof Lemma \ref{better path island shallow 2}, we recall that for graphs $G$ and $H$ and a nonnegative integer $r$, ${\mathcal F}(G,H,r)$ is the set consisting of the graphs that can be obtained from a disjoint union of $G$ and $H$ by adding edges between $V(G)$ and $V(H)$ such that every vertex in $V(H)$ has degree at least $r$.
For a graph $W$ in ${\mathcal F}(G,H,r)$, the {\it type} of $W$ is the number of edges of $W$ incident with $V(H)$, and the {\it heart} of $W$ is $V(G)$.

\begin{lemma} \label{better path island shallow 2}
For any $r,t,t' \in {\mathbb N}$, $w \in {\mathbb Z}$ with $r \geq w \geq 0$, nonnegative integer $s_0$ and positive real numbers $k,k'$, there exists an integer $d$ such that 
for every graph $G$, either
	\begin{enumerate}
		\item the average degree of $G$ is greater than $k$, 
		\item there exists a graph $H$ of average degree greater than $k'$ such that some subgraph of $G$ is isomorphic to a $[4s_0+2w+5]$-subdivision of $H$, 
		\item $G$ contains $K_{r-w+1,t}$ as a $(2s_0+w+2)$-shallow minor, or 
		\item there exists $X \subseteq V(G)$ such that
			\begin{enumerate}
				\item every vertex in $X$ has degree at most $d$ in $G$, 
				\item there exists $v^* \in X$ such that for every subgraph $R$ of $G$ of minimum degree at least $r$ containing $v^*$, there exists a path in $G[X \cap V(R)]$ of length $w$ starting at $v^*$, and 
				\item either $X=V(G)$, or there exists a nonnegative integer $s$ with $s \leq s_0$ such that either 
					\begin{enumerate}
						\item there exists a connected graph $F_0$ such that $G$ contains $F \wedge_{t'} I$ as a subgraph for some $F \in {\mathcal F}(I_{r-w}, F_0,r)$ of type $s$, where $I$ is the heart of $F$, or 
						\item there exists $x^* \in X$ such that for every subgraph $R$ of $G$ of minimum degree at least $r$ containing $x^*$, there exists a connected subgraph $F$ of $R[X \cap V(R)]$ containing $x^*$ such that the number of edges in $R$ incident with $V(F)$ is at least $s_0+1$. 
					\end{enumerate}
			\end{enumerate}
	\end{enumerate}
\end{lemma}

\begin{proof}
Let $r,t,t' \in {\mathbb N}$, $w \in {\mathbb Z}$ with $r \geq w \geq 0$, $s_0 \in {\mathbb Z}$ with $s_0 \geq 0$, and $k,k'$ be positive real numbers.
Let $\beta = (s_0+1)^2 \cdot 2^{{s_0+1 \choose 2}} \cdot (r-w+1)(t'-1 + r -w) \left(\frac{k'}{2}+\binom{k'}{\lfloor k'/2 \rfloor} \cdot  t' \cdot 2^{(r-w)(s_0+1)}\right)t' \cdot 2^{(r-w)(s_0+1)}$. 
Define $d$ to be the integer mentioned in Lemma \ref{stronger better path island shallow} by taking $(r,t,\ell,k,k',\beta)=(r-w+1,t,2s_0+w+2,k,k',\beta)$. 

Let $G$ be a graph.
Suppose that Statements 1-3 of this lemma do not hold.
So by Lemma \ref{stronger better path island shallow}, there exist $X,Z \subseteq V(G)$ with $Z \subseteq X$ and $\lvert Z \rvert > \beta\lvert V(G)-X \rvert$ such that 
	\begin{itemize}
		\item[(i)] every vertex in $X$ has degree at most $d$ in $G$, 
		\item[(ii)] for any distinct pair of vertices in $Z$, the distance in $G[X]$ between them is at least $2s_0+w+3$, 
		\item[(iii)] for every $z \in Z$ and $u \in X$ whose distance from $z$ in $G[X]$ is at most $2s_0+w+2$, $\lvert N_G(u)-X \rvert \leq r-w$, and
		\item[(iv)] $\lvert N_G(N_{G[X]}^{\leq 2s_0+w+1}[z])-X \rvert \leq r-w$ for every $z \in Z$.
	\end{itemize}

We shall prove that Statement 4 of this lemma holds.
Statement 4(a) immediately follows from (i).

We first prove Statement 4(b).
Let $v^*$ be any vertex in $Z$.
Suppose to the contrary that there exists a subgraph $R$ of $G$ of minimum degree at least $r$ containing $v^*$ such that the longest path $P$ in $R[X \cap V(R)]$ starting at $v^*$ has length at most $w-1$.
For every vertex $v \in V(P)$, $\lvert N_R(v) \cap X -V(P)  \rvert \geq \lvert N_R(v) \rvert - \lvert N_R(v)-X \rvert - \lvert N_R(v) \cap V(P) \rvert \geq \lvert N_R(v) \rvert - \lvert N_G(v)-X \rvert - (\lvert V(P) \rvert-1) \geq r-(r-w)-(w-1) =1$ where the last inequality follows from (iii) by taking $(z, u)$ in (iii) to be $(v^*, v)$. 
So $P$ is not a longest path in $R$ starting at $v^*$ since if $v$ is the other end of $P$, then we can extend $P$ by concatenating a vertex in $N_R(v) \cap X-V(P)$. 
This leads to a contradiction.
Since $R[X \cap V(R)] \subseteq G[X \cap V(R)]$, Statement 4(b) is proved.

Now we prove Statement 4(c). 
We may assume that $X \neq V(G)$, for otherwise we are done.
Assume 4(c)ii does not hold.
We shall show 4(c)i holds.

For every $z \in Z$ and every subgraph $R$ of $G$ of minimum degree at least $r$ containing $z$, define $s_{R,z}$ to be the number of edges of $R$ incident with the vertices in the component of $R[V(R) \cap X]$ containing $z$.
For every $z \in Z$, define $s_z' = \min_R s_{R,z}$, where the minimum is taken over all subgraphs $R$ of $G$ of minimum degree at least $r$ containing $z$. 
Note that for every $z \in Z$, $s_z' \geq r$ as the minimum is taken over all subgraphs of minimum degree at least $r$. 
If there exists $z \in Z$ such that $s_z' \geq s_0+1$, then Statement 4(c)ii holds by taking $x^*=z$.

So we may assume that $s_z' \leq s_0$ for every $z \in Z$.
Define $s$ to be an integer with $0 \leq s \leq s_0$ such that $\lvert \{z \in Z: s_z'=s\} \rvert$ is maximum.
Let $Z_s = \{z \in Z: s_z'=s\}$.
In particular, 
\begin{equation} \lvert Z_s \rvert \geq \frac{1}{s_0+1}\lvert Z \rvert> \frac{\beta}{s_0+1}\lvert V(G)-X \rvert. \label{eq:int2}
\end{equation}

If there is a vertex $z \in Z_s$ such that  for every subgraph $R$ of $G$ of minimum degree at least $r$ containing $z$ with $s_{R,z}=s_z'=s$, the connected component $F_{R,z}$ of $R[V(R) \cap X]$ containing $z$ contains at least $s_0+2$ vertices, then for every such $R$, the number of edges in $R$ incident with $V(F_{R,z})$ is at least $s_0+1 \geq s+1=s_{R,z}+1$, a contradiction. 
So for every $z \in Z_s$, there exists a subgraph $R_z$ of $G$ of minimum degree at least $r$ containing $z$ with $s_{R_z,z}=s$ such that the component $F_{z}$ of $R_z[V(R_z) \cap X]$ containing $z$ satisfies that 
\begin{equation} \lvert V(F_{z}) \rvert \leq s_0+1. \label{eq:sizeFz}
\end{equation}  
Since there are at most $(s_0+1) \cdot 2^{{s_0+1 \choose 2}}$ non-isomorphic labelled graphs on at most $s_0+1$ vertices, there exist a connected (labelled) graph $F$ on at most $s_0+1$ vertices and $Z_s' \subseteq Z_s$ with 
\[  \lvert Z_s' \rvert \geq \frac{\lvert Z_s \rvert}{(s_0+1) \cdot 2^{{s_0+1 \choose 2}}} > \frac{\beta}{(s_0+1)^2 \cdot 2^{{s_0+1 \choose 2}}} \lvert V(G)-X \rvert\] such that $F$ is isomorphic to each (labelled) $F_{z}$ for every $z \in Z_s'$, where the we use (\ref{eq:int2}) for the second inequality.

For every $z \in Z_s'$, since $F_z$ is connected and contains at most $s_0+1$ vertices, 
\begin{equation} V(F_z) \subseteq N_{G[X]}^{\leq s_0}[z]. \label{eq:int3}
\end{equation}
By (iv), for every $z \in Z_s'$, $\lvert N_G(N_{G[X]}^{\leq s_0}[z])-X \rvert \leq r-w$.
So there exist an integer $p$ with $0 \leq p \leq r-w$ and a set $Z_s^* \subseteq Z_s'$ with 
\begin{align} \lvert Z_s^* \rvert \geq \frac{\lvert Z_s' \rvert}{r-w+1} > \frac{\beta}{(s_0+1)^2 \cdot 2^{{s_0+1 \choose 2}} \cdot (r-w+1)} \lvert V(G)-X \rvert  \label{eq:sizeZsstar}  \\ \geq \frac{\beta}{(s_0+1)^2 \cdot 2^{{s_0+1 \choose 2}} \cdot (r-w+1)}
 \geq t'-1+r-w \nonumber
\end{align}
such that $\lvert N_G(V(F_z))-X \rvert=p$ for every $z \in Z_s^*$.

A quick remark is that,
by (ii), for distinct vertices $z_1,z_2$ in $Z_s'$, $N_{G[X]}^{\leq s_0}[z_1]$ and $N_{G[X]}^{\leq s_0}[z_2]$ are disjoint. Together with (\ref{eq:int3}), we have that 
\begin{equation} V(F_z) \cap V(F_{z'})=\emptyset. \label{eq:int4}
\end{equation}

We first assume that $p=0$.
Then for every $z \in Z_s^*$, $F_z$ is of minimum degree at least $r$ since $R$ is of minimum degree at least $r$ and $N_G(V(F_z)) \subseteq X$. 
Since $\lvert Z_s^* \rvert \geq t'+r-w$, the graphs $F_z$ for $z \in Z_s^*$ form at least $r-w+t'$ disjoint copies of $F$ in $G$. 
We just showed that $F$ is of minimum degree at least $r$. 
Let $F'$ be a disjoint union of $F$ and $r-w$ isolated vertices. 
Then $F' \in {\mathcal F}(I_{r-w},F,r)$ and is of type $s$. 
Since $G$ contains $r-w+t'$ disjoint copies of $F$, we know $G$ contains $F' \wedge_{t'} I$ where $I$ is the heart of $F'$, as we can take $t'$ disjoint copies of $F$ and one vertex in each of other $r-w$ copies of $F$.
So Statement 4(c)i holds.

So we may assume that $p \geq 1$. Recall that by the definition of $Z_s^*$, $\lvert N_G(V(F_z))-X \rvert=p$ for every $z \in Z_s^*$.

\begin{claim}\label{claim:doublecountZs}
If there is a subset $S \subseteq V(G) -X$ such that $S$ equals $N_G(V(F_z))-X$ for at least $t' \cdot 2^{p(s_0+1)}$ vertices $z \in Z_s^*$, then Statement 4(c)i holds. 
\end{claim}

\noindent{\bf Proof of Claim \ref{claim:doublecountZs}:}
For every $z \in Z^*_s$, since $|N_G(V(F_z))- X|=p$, and each of the copies $F_z$ are isomorphic (as a labelled graph), there are at most $2^{|V(F_z)| p} \leq 2^{(s_0+1)p}$ possibilities for how vertices in $F_z$ are connected in $G$ to the $p$ vertices in the set $N_G(V(F_z))-X$ by (\ref{eq:sizeFz}) and the fact that there are $|V(F_z)| p$ potential egdes between vertices in $F_z$ and $N_G(V(F_z))-X$. 

Notice that each vertex in $F_z$ has degree at least $r$ in $G[N_G[V(F_{z})]]$.
By a pigeonhole argument, if $S$ is a subset of $V(G)-X$ such that $S$ equals $N_G(V(F_z))- X$ for at least $t' \cdot 2^{p(s_0+1)}$ vertices $z \in Z_s^*$, then there are at least $t'$ vertices $z \in Z_s^*$ such that the graphs $G[S \cup V(F_z)] - E[S]$, denoted by $F'_z$, are isomorphic to a graph $F'$ as a labelled graph.  
Let $F_0$ be $F_z$ for one of these $t'$ vertices $z \in Z_s^*$. 
Then  $F' \in {\mathcal F}(I_p,F_0,r)$ and the union of $F'_z$ among these $t'$ vertices in $Z^*_s$ is a subgraph $G'$ of $G$ isomorphic to $F' \wedge_{t'} I$, where $I$ is the stable set corresponding to $V(I_p)$.
Let $G''$ be the union of $G'$ and $r-w-p$ vertices in the remaining $|Z^*_s|-t' \geq r-w$ vertices in $Z^*_s$.
Then $G''$ is isomorphic to $F'' \wedge_{t'} I''$ for some $F'' \in {\mathcal F}(I_{r-w},F_0,r)$, where $I''$ is the union of $I$ and the new $r-w-p$ vertices.
Therefore Statement 4(c)i holds. 
 \hfill
$\Box$

\begin{claim} \label{claim:manytoone}
If Statement 4(c)i does not hold, then 
$$\lvert \{N_G(V(F_z))-X: z \in Z_s^*\} \rvert \leq \left(\frac{k'}{2}+ \binom{k'}{\lfloor k'/2 \rfloor} \cdot  t' \cdot 2^{p(s_0+1)}\right) \lvert V(G)-X \rvert.$$
\end{claim}

\noindent{\bf Proof of Claim \ref{claim:manytoone}:}
By (\ref{eq:int4}), $V(F_z) \cap V(F_{z'})=\emptyset$ for distinct vertices $z_1,z_2$ in $Z_s^*$.
Starting from $G[\bigcup_{z \in Z_s^*}V(F_z) \cup (V(G)-X)]-E(G[V(G)-X])$, we obtain a graph $H'$ by repeatedly deleting all the vertices in $V(F_z)$ for some $z \in Z_s^*$ where some pair of distinct vertices $y,y'$ in $N_G(V(F_z))-X$ are non-adjacent in the current graph, and adding the edge $yy'$. 
We continue this process until for every remaining vertex $z'$ in $Z_s^*$, $N_G(V(F_z))-X$ is a clique. 

Let $H=H'[V(G)-X]$.
Since $p \geq 1$ and $V(F_z) \subseteq N_{G[X]}^{\leq s_0}[z]$ by (\ref{eq:int4}) for every $z \in Z_s^*$ (which implies any two vertices in $F_z$ can be connected in $F_z$ by a path of length at most $2s_0$), we know $G[\bigcup_{z \in Z_s^*}((N_G(V(F_z))-X) \cup V(F_z))]$ contains a $[2s_0+1]$-subdivision of $H$.
It implies that $G$ contains a $[2s_0+1]$-subdivision of any subgraph of $H$.
Since Statement 2 of this lemma does not hold, the average degree of any subgraph of $H$ is at most $k'$.

For each vertex $z \in Z_s^*$, either $V(F_z)$ has been deleted thus corresponding to a unique edge in $H$, or $V(F_z)$ survives in $H'$, in which case $N_G(V(F_z))-X$ becomes a clique of size $\lvert N_G(V(F_z))-X \rvert = p$ in $H'$, and thus also a clique of size $p$ in $H$ since $N_G(V(F_z))-X \subseteq V(G)-X$. 
There are at most $|E(H)|$ vertices in $Z_s^*$ of the first kind. 
Since the maximum average degree of $H$ is at most $k'$, $\lvert E(H) \rvert \leq \frac{k'\lvert V(H) \rvert}{2} = \frac{k'}{2}\lvert V(G)-X \rvert$. 

For the vertices in $Z'$ of the second kind, note that $N_G(V(F_{z})) - X$ is a clique of size $p$ in $H$. 
Let $c$ be the number of vertices in $Z_s^*$ of the second kind. 
By Claim \ref{claim:doublecountZs},  each $S \subseteq V(G) - X$ is the neighborhood of at most $t' \cdot 2^{p(s_0+1)}$ vertices $z \in Z_s^*$ of the second kind. 
Since each $z \in Z_s^*$ gives a clique of size $p$ in $H$ and by 
Lemma \ref{number cliques}, the number of cliques of size $p$ in $H$ is at most ${k' \choose p-1}\lvert V(G)-X \rvert \leq  \binom{k'}{\lfloor k'/2 \rfloor} \lvert V(G)-X \rvert$. 
Combining these two facts, we have $c \leq \binom{k'}{\lfloor k'/2 \rfloor} \cdot  t' \cdot 2^{p(s_0+1)} \lvert V(G)-X \rvert$. 
Therefore,
\[ \lvert \{N_G(V(F_z))-X: z \in Z_s^*\} \rvert \leq |E(H)| + c \leq \left(\frac{k'}{2}+\binom{k'}{\lfloor k'/2 \rfloor} \cdot  t' \cdot 2^{p(s_0+1)}\right) \lvert V(G)-X \rvert.\]
$\Box$

By Claim \ref{claim:manytoone}, the number of distinct sets of the form $N_G(V(F_z))-X$ for some $z \in Z_s^*$ is at most  $\left(\frac{k'}{2}+\binom{k'}{\lfloor k'/2 \rfloor} \cdot  t' \cdot 2^{p(s_0+1)}\right) \lvert V(G)-X \rvert$. 
However, by (\ref{eq:sizeZsstar}), $\lvert Z_s^* \rvert >  \frac{\beta}{(s_0+1)^2 \cdot 2^{{s_0+1 \choose 2}} \cdot (r-w+1)} \cdot \lvert V(G)-X \rvert $. 
Therefore there is a subset $S \subseteq V(G)-X$ with $\lvert S \rvert = p \leq r-w$ such that there are at least $ \left( \frac{\beta}{(s_0+1)^2 \cdot 2^{{s_0+1 \choose 2}} \cdot (r-w+1)}\right) / \left(\frac{k'}{2}+\binom{k'}{\lfloor k'/2 \rfloor} \cdot  t' \cdot 2^{p(s_0+1)}\right) \geq t' \cdot 2^{p(s_0+1)}$ vertices $z$ in $Z_s^*$ satisfying $S = N_G(F_z) -X$. 
Then Statement 4(c)i holds by Claim \ref{claim:doublecountZs}. 
This completes the proof. 
\end{proof}

We show how to adapt the proof of Lemma \ref{sufficient good collection} to obtain a desired good signature collection by using Lemma \ref{better path island shallow 2} in Lemmas \ref{subgraph one vertex collection} and \ref{minor collection}.

\begin{lemma} \label{subgraph one vertex collection}
For any $r,t,t' \in {\mathbb N}$, $w \in {\mathbb Z}$ with $r \geq w \geq 0$, nonnegative integer $s_0$ and positive real numbers $k,k'$, there exist integers $c,d$ such that for every graph $G$, either
	\begin{enumerate}
		\item the average degree of $G$ is greater than $k$, 
		\item there exists a graph $H$ of average degree greater than $k'$ such that some subgraph of $G$ is isomorphic to a $[4s_0+2w+5]$-subdivision of $H$,
		\item $G$ contains $K_{r-w+1,t}$ as a $(2s_0+w+2)$-shallow minor, or  
		\item there exists a vertex $v^*\in V(G)$ and a collection $\C^*$ of $((w+1)r-{w+1 \choose 2})$-element subsets of $E(G)$ with $\lvert \C^* \rvert \leq c$ such that for every subgraph $R$ of $G$ of minimum degree at least $r$ containing $v^*$, $E(R)$ contains some member of $\C^*$, and either 
			\begin{enumerate}
				\item every vertex of $G$ is of degree at most $d$, and there exists a vertex $x^* \in V(G)$ and a collection $\C$ of ${r+1 \choose 2}$-element subsets of $E(G)$ with $\lvert \C \rvert \leq c$ such that for every subgraph $R'$ of $G$ of minimum degree at least $r$ containing $x^*$, $E(R')$ contains some member of $\C$, or
				\item there exists a nonnegative integer $s$ with $s \leq s_0$ such that either
				\begin{enumerate}
					\item there exists a connected graph $F_0$ such that $G$ contains $F \wedge_{t'} I$ as a subgraph for some $F \in {\mathcal F}(I_{r-w}, F_0,r)$ of type $s$, where $I$ is the heart of $F$, or 
					\item there exists a vertex $x^* \in V(G)$ and a collection $\C$ of $(s_0+1)$-element subsets of $E(G)$ with $\lvert \C \rvert \leq c$ such that for every subgraph $R'$ of $G$ of minimum degree at least $r$ containing $x^*$, $E(R')$ contains some member of $\C$.
				\end{enumerate}
			\end{enumerate}
	\end{enumerate}
\end{lemma}

\begin{proof}
Let $r,t,t' \in {\mathbb N}$, $w \in {\mathbb Z}$ with $r \geq w \geq 0$, $s_0$ be a nonnegative integer, and $k,k'$ be positive real numbers.
Let $d$ be the number $d$ mentioned in Lemma \ref{better path island shallow 2} by taking $(r,t,t',w,s_0,k,k')=(r,t,t',w,s_0,k,k')$.
Define $c={d \choose r}^{r+1} \cdot (4(r+1)d)^{(r+1)^2} + {d \cdot (s_0+3)d^{s_0+2} \choose s_0+1} \cdot 2^{(s_0+3)^2d^{2s_0+4}}$.

Let $G$ be a graph.
Assume that Statements 1-3 of this lemma do not hold.
By Lemma \ref{better path island shallow 2}, there exists $X \subseteq V(G)$ such that 
	\begin{itemize}
		\item[(i)] every vertex in $X$ has degree at most $d$ in $G$, 
		\item[(ii)] there exists $v^* \in X$ such that for every subgraph $R$ of $G$ of minimum degree at least $r$ containing $v^*$, there exists a path $Q_R$ in $G[X \cap V(R)]$ of length $w$ starting at $v^*$, and
		\item[(iii)] either $X=V(G)$, or there exists an integer $s$ with $0 \leq s \leq s_0$ such that either 
			\begin{itemize}
				\item[(C1)] there exists a connected graph $F_0$ such that $G$ contains $F \wedge_{t'} I$ as a subgraph for some $F \in {\mathcal F}(I_{r-w}, F_0,r)$ of type $s$, where $I$ is the heart of $F$, or 
				\item[(C2)] there exists $x^* \in X$ such that for every subgraph $R$ of $G$ of minimum degree at least $r$ containing $x^*$, there exists a connected subgraph $F$ of $R[X \cap V(R)]$ containing $x^*$ such that the number of edges in $R$ incident with $V(F)$ is at least $s_0+1$.
			\end{itemize}
	\end{itemize}

We shall show Statement 4 of this lemma holds.
For every $v \in X$, let $\C_v$ be the collection of all $r$-element subsets of $E(G)$ such that each of the $r$ edges is incident with $v$.
Since every vertex in $X$ has degree at most $d$ in $G$, $\lvert \C_v \rvert \leq {d \choose r}$ for every $v \in X$.

For every subgraph $Q$ in $G[X]$, let 
\[ \C_Q=\{\bigcup_{v \in V(Q)}T_v: (T_v \in \C_v: v \in V(Q))\}.\] 
In other words, each member of $\C_Q$ is a union of $|V(Q)|$ sets where each of them consists of $r$ edges incident with a vertex of $Q$ and no two distinct sets are corresponding to the same vertex of $Q$. 
For every subgraph $Q$ in $G[X]$, since $\lvert\C_Q \rvert \leq \prod_{v \in V(Q)} |\C_v|$, we have 
\begin{equation}
\lvert \C_Q \rvert \leq {d \choose r}^{\lvert V(Q) \rvert}. \label{eq:boundsizeCP}
\end{equation}

\begin{claim} \label{claim:generalsizecollection}
Let $u \in X$ and $q$ be a nonnegative integer.
If $\C$ is the set consisting of all the members of $\C_Q$ for all connected subgraphs $Q$ in $G[X]$ containing $u$ satisfying that $\lvert V(Q) \rvert=q$ and every vertex $v \in V(Q)$ has degree at least $r$ in $G$, then every member of $\C$ has size at least $qr-{q \choose 2}$, and $\lvert \C \rvert \leq {d \choose r}^q \cdot (4qd)^{q^2}$.
\end{claim}

\noindent{\bf Proof of Claim \ref{claim:generalsizecollection}:}
Since every vertex of $Q$ has degree at least $r$ in $G$, every member of $\C$ has size at least $qr-{q \choose 2}$.
Since every vertex in $X$ has degree at most $d$ in $G$, for every $q'$ with $0 \leq q' \leq q$, there are at most $d^{q'} \leq d^q$ paths in $G[X]$ of length $q'$ starting at $u$.
So $\lvert N_G^{\leq q}[u] \rvert \leq qd^q+1$.
Since every connected subgraph $Q$ in $G[X]$ containing $u$ with $\lvert V(Q) \rvert=q$ satisfies $V(Q) \subseteq N_G^{\leq q}[u]$, there are at most ${\lvert N_G^{\leq q}[u] \rvert \choose q} \cdot 2^{\binom{|V(Q)|}{2}} \leq (4qd)^{q^2}$ connected subgraphs $Q$ in $G[X]$ containing $u$ with $\lvert V(Q) \rvert=q$.
So together with (\ref{eq:boundsizeCP}), $\lvert \C \rvert \leq {d \choose r}^q \cdot \lvert \{Q: Q$ is a connected subgraph in $G[X]$ containing $u$ with $\lvert V(Q) \rvert=q\} \rvert \leq {d \choose r}^q \cdot (4qd)^{q^2}$. \hfill
$\Box$

\

Define $\C_0$ to be the union of $\C_{Q_R}$ over all subgraphs $R$ of $G$ of minimum degree at least $r$ containing $v^*$, where $v^*$ and $Q_R$ are defined in (ii).
By Claim \ref{claim:generalsizecollection}, every member of $\C_0$ has size at least $(w+1)r-{w+1 \choose 2}$ and $\lvert \C_0 \rvert \leq {d \choose r}^{w+1} \cdot (4(w+1)d)^{(w+1)^2} \leq c$.
By (ii), for every subgraph $R'$ of $G$ of minimum degree at least $r$ containing $v^*$, $E(R')$ contains some member of $\C_0$.

For every $S \in \C_0$, since $\lvert S \rvert  \geq (w+1)r-{w+1 \choose 2}$, there exists a subset $f^*(S)$ of $S$ of size $(w+1)r-{w+1 \choose 2}$.
Let $\C^*=\{f^*(S): S \in \C_0\}$.
So every member of $\C^*$ has size $(w+1)r-{w+1 \choose 2}$, $\lvert \C^* \rvert \leq \lvert \C_0 \rvert \leq c$, and for every subgraph of $G$ of minimum degree at least $r$ containing $v^*$, its edge-set contains some member of $\C^*$.

Therefore, to prove Statement 4 of this lemma, it suffices to prove Statements 4(a) or 4(b) holds.
We first assume that $X=V(G)$.
Then every vertex of $G$ is of degree at most $d$ by (i).
If every vertex of $G$ has degree less than $r$, then there exists no subgraph of $G$ of minimum degree at least $r$, so Statement 4(a) holds by choosing $\C=\emptyset$ and choosing $x^*$ to be any vertex of $G$.
Hence we may assume that there exists a vertex $v$ of $G$ of degree at least $r$.
For every subgraph $R$ of $G$ of minimum degree at least $r$ containing $v$, there exists a star $T_R$ on $r+1$ vertices centered at $v$ contained in $R$.
Note that every vertex in such $T_R$ has degree at least $r$ in $G$ since $R$ has minimum degree at least $r$.
Define $\C_1$ to be the union of $\C_{T_R}$ over all subgraphs $R$ of $G$ of minimum degree at least $r$ containing $v$. 
By Claim \ref{claim:generalsizecollection}, every member of $\C_1$ has size at least $(r+1)r-{r+1 \choose 2}={r+1 \choose 2}$ and $\lvert \C_1 \rvert \leq {d \choose r}^{r+1} \cdot (4(r+1)d)^{(r+1)^2} \leq c$.
For every subgraph $R'$ of $G$ of minimum degree at least $r$ containing $v$, since $V(T_{R'}) \subseteq V(R')$, $E(R')$ contains some member of $\C_1$. 
Hence Statement 4(a) holds and we are done.

So we may assume that $X \neq V(G)$.
Hence by (iii), there exists a nonnegative integer $s$ with $s \leq s_0$ such that either (C1) or (C2) holds.
We may also assume that Statement 4(b)(i) does not hold, for otherwise we are done.
In particular, (C2) holds by (iii).

Let $\C = \{E(Q): Q$ is a subgraph of $G$ obtained from a connected subgraph $Q'$ of $G[X]$ containing $x^*$ with $|E(Q')| \leq s_0+1$ by adding edges of $G$ incident with $V(Q')$ such that $\lvert E(Q) \rvert=s_0+1\}$.
Note that for every connected subgraph $Q'$ of $G[X]$ with $x^* \in V(Q')$ and $\lvert E(Q') \rvert \leq s_0+1$, $V(Q') \subseteq N_{G[X]}^{\leq s_0+2}[x^*]$ by the connectedness.
Since every vertex in $X$ has degree at most $d$ in $G$, $|V(Q')| \leq \lvert N_{G[X]}^{\leq s_0+2}[x^*] \rvert \leq (s_0+3)d^{s_0+2}$. 
Thus the number of such connected graphs $Q'$ is at most $2^{\lvert N_{G[X]}^{\leq s_0+2}[x^*] \rvert} \cdot 2^{\binom{\lvert N_{G[X]}^{\leq s_0+2}[x^*] \rvert}{2}} \leq 2^{(\lvert N_{G[X]}^{\leq s_0+2}[x^*] \rvert)^2} \leq  2^{(s_0+3)^2d^{2s_0+4}}$.
So the number of subgraphs $Q$ of $G$ obtained from a connected subgraphs $Q'$ of $G[X]$ containing $x^*$ with $|E(Q')| \leq s_0+1$ by adding edges of $G$ incident with $V(Q')$ such that $\lvert E(Q) \rvert=s_0+1$ is at most $\binom{ d |V(Q')|}{s_0+1}$ multiplying by the   number of $Q'$, which is at most ${d \cdot (s_0+3)d^{s_0+2} \choose s_0+1} \cdot 2^{(s_0+3)^2d^{2s_0+4}} \leq c$.
Hence $\lvert \C \rvert \leq c$.
In addition, every member of $\C$ has size $s_0+1$.

By (C2), for every subgraph $R$ of $G$ of minimum degree at least $r$ containing $x^*$, there exists a connected subgraph $F_R'$ of $R[X \cap V(R)]$ containing $x^*$ whose number of edges in $R$ incident to vertices in $F_R'$ is at least $s_0+1$, so there exists a connected subgraph $F_R$ of $R$ obtained from a connected subgraph $F_R''$ of $F_R'$ containing $x^*$ with $|E(F_R'')| \leq s_0+1$ by adding edges of $R$ incident with $V(F_R'')$ such that $\lvert E(F_R) \rvert=s_0+1$. 
Note that $E(F_R) \in \C$. 
Hence $E(R)$ contains $E(F_R) \in \C$. 
Therefore Statement 4(b)(ii) holds.
\end{proof}

Now we are ready to prove Lemma \ref{minor collection}. 

\begin{lemma} \label{minor collection} 
For any $r,t,t' \in {\mathbb N}$, integer $w$ with $r \geq w \geq 0$, and nonnegative integer $s_0$, there exists an integer $c$ such that for every graph $G$, either
	\begin{enumerate}
		\item $G$ contains $K_{r-w+1} \vee I_t$ as a minor, or
		\item there exist a $(c,(w+1)r-{w+1 \choose 2}, r)$-good signature collection $\C$ for $G$ and a nonnegative integer $s$ with $s \leq s_0$ such that either
			\begin{enumerate}
				\item there exists a connected graph $F_0$ such that $G$ contains $F \wedge_{t'} I$ as a subgraph for some $F \in {\mathcal F}(I_{r-w}, F_0,r)$ of type $s$, where $I$ is the heart of $F$,  or
				\item there exists a $(c, \min\{s_0+1,{r+1 \choose 2}\}, r)$-good signature collection for $G$. 
			\end{enumerate}
	\end{enumerate}
\end{lemma}

\begin{proof}
Let $r,t,t' \in {\mathbb N}$ and $w$ be an integer with $r \geq w \geq 0$.
Let $s_0$ be a nonnegative integer.
Let $k$ be a real number such that every graph with average degree at least $k$ contains $K_{r-w+1} \vee I_t$ as a minor. 
Note that such a number $k$ exists since we can take $k$ to be any value larger than the supreme of maximum average degree in all $K_{r-w+1+t}$-minor free graphs, and the supreme exists by \cite{Thomason}. 
Define $c$ and $d$ to be the numbers $c$ and $d$ mentioned in Lemma \ref{subgraph one vertex collection} by taking $(r,t,t',w,s_0,k,k')=(r,t+{r-w+1 \choose 2},t',w,s_0,k,k)$. 

Let $G$ be a graph.
We shall prove this lemma by induction on $\lvert V(G) \rvert$.
This lemma holds when $\lvert V(G) \rvert=1$ since there exists no subgraph of $G$ of minimum degree at least $r$ and hence Statement 2 holds.
Now we assume that this lemma holds for all graphs with fewer vertices than $G$.

We may assume that $G$ does not contain $K_{r-w+1} \vee I_t$ as a minor, for otherwise we are done.
Since $G$ does not contain $K_{r-w+1} \vee I_t$ as a minor, every subgraph of $G$ has average degree less than $k$.
Similarly, there does not exist a graph $H$ of average degree greater than $k$ such that some subgraph $H'$ of $G$ is a $[4s_0+2w+5]$-subdivision of $H$, for otherwise $H'$ (and hence $G$) contains a subdivision of a subgraph of $H$ that contains $K_{r-w+1} \vee I_t$ as a minor, a contradiction.
In addition, since $K_{r-w+1,t+{r-w+1 \choose 2}}$ contains $K_{r-w+1} \vee I_t$ as a minor, $G$ does not contain $K_{r-w+1,t+{r-w+1 \choose 2}}$ as a $(2s_0+w+2)$-shallow minor.

Hence, applying Lemma \ref{subgraph one vertex collection} by taking $(r,t,t',w,s_0,k,k')=(r,t+{r-w+1 \choose 2},t',w,s_0,k,k)$, there exists $x^* \in V(G)$ and a collection $\C_{x^*}$ of $q$-element subsets of $E(G)$ with $\lvert \C_{x^*} \rvert \leq c$ such that for every subgraph $R$ of $G$ of minimum degree at least $r$ containing $x^*$, $E(R)$ contains some member of $\C_{x^*}$, where $q$ is defined as follows:
	\begin{itemize}
		\item if every vertex of $G$ is of degree at most $d$, then $q={r+1 \choose 2}$;
		\item otherwise, if there exists a connected graph $F_0$ such that $G$ contains $F \wedge_{t'} I$ as a subgraph for some $F \in {\mathcal F}(I_{r-w}, F_0,r)$ of type $s$ for some integer $s$ with $0 \leq s \leq s_0$, where $I$ is the heart of $F$, then $q=(w+1)r-{w+1 \choose 2}$;
		\item otherwise, $q=\max\{(w+1)r-{w+1 \choose 2},\min\{s_0+1,{r+1 \choose 2}\}\}$.
	\end{itemize}
Since $w$ is an integer with $0 \leq w \leq r$, $(w+1)r-{w+1 \choose 2} \leq {r+1 \choose 2}$.

Let $G'=G-x^*$.
Note that $G'$ does not contain $K_{r-w+1} \vee I_t$ as a minor.
So by the induction hypothesis, there exists a collection $\C'$ of $q'$-element subsets of $E(G')$ with $\lvert \C' \rvert \leq c\lvert V(G') \rvert = c(\lvert V(G) \rvert-1)$ such that for every subgraph $R$ of $G'$ of minimum degree at least $r$, $E(R)$ contains some member of $\C'$, where 
	\begin{itemize}
		\item if there exists a connected graph $F_0$ such that $G'$ contains $F \wedge_{t'} I$ as a subgraph for some $F \in {\mathcal F}(I_{r-w}, F_0,r)$ of type $s$ for some integer $s$ with $0 \leq s \leq s_0$, where $I$ is the heart of $F$, then $q'=((w+1)r-{w+1 \choose 2})$, and 
		\item otherwise, $q'=\max\{((w+1)r-{w+1 \choose 2}), \min\{s_0+1,{r+1 \choose 2}\}\}$.
	\end{itemize}
Note that if there exists a connected graph $F_0$ such that $G'$ contains $F \wedge_{t'} I$ as a subgraph for some $F \in {\mathcal F}(I_{r-w}, F_0,r)$ of type $s$, where $I$ is the heart of $F$, then so does $G$.
So $q' \leq q$.
Hence for every $S \in \C_{x^*}$, there exists a subset $f(S)$ of $S$ of size $q'$ such that $\lvert \{f(S): S \in \C_{x^*}\} \rvert \leq c$, and for every subgraph $R$ of $G$ of minimum degree at least $r$ containing $x^*$, $E(R)$ contains some member of $\{f(S): S \in \C_{x^*}\}$.

Define $\C=\{f(S): S \in \C_{x^*}\} \cup \C'$.
So $\C$ is a collection of $q'$-element subsets of $E(G)$ with size at most $\lvert \C_{x^*} \rvert + \lvert \C' \rvert \leq c\lvert V(G) \rvert$.
Let $R$ be a subgraph of $G$ of minimum degree at least $r$.
If $R$ contains $x^*$, then $E(R)$ contains some member of $\{f(S): S \in \C_{x^*}\} \subseteq \C$.
If $R$ does not contain $x^*$, then $R$ is a subgraph of $G'$ of minimum degree at least $r$, so $E(R)$ contains some member of $\C' \subseteq \C$.
Therefore, Statement 2 holds for $G$.
This proves this lemma.
\end{proof}

\section{Proof of main theorems}  \label{sec:Cimplies-main-thm}

We prove Theorems \ref{thm: main critical exponent 1}, \ref{thm:critical exponent lower bound} and Corollary \ref{cor:many_intro} in this section. 
We first prove the following simple lemma. 

\begin{lemma} \label{Krsminorsubgraph}
Let $r$ be a positive integer.
Let $H$ be a graph that is not a subgraph of $K_{r} \vee I_t$ for any positive integer $t$.
Then $\{K_{r,s}: s \geq r\} \subseteq \M(H)$.
\end{lemma}

\begin{proof}
For every integer $s$ with $s \geq r$, every minor of $K_{r,s}$ is a subgraph of $K_r \vee I_s$.
Hence, if there is an integer $s$ such that $K_{r,s}$ contains $H$ as a minor, then $H$ is a subgraph of $K_r \vee I_s$, a contradiction.
Hence $K_{r,s}$ does not contain $H$ as a minor for every $s \geq r$.
\end{proof}

Now we combine Lemmas \ref{lem:coverCandp}, \ref{weak general collection} and \ref{minor collection} to determine or find a lower bound for the threshold for degeneracy.
It is an essential step toward proving our main results in Section \ref{subsec:ourresult}.

\begin{lemma} \label{critical exponent 1-w}
Let $r \geq 2$ be an integer.
Let $H$ be a graph.
Then $p_{\M(H)}^{\D_r}=\Omega(n^{-1/q_H})$, where $q_H$ is defined as follows. 
	\begin{enumerate}
		\item If $H$ is not a subgraph of $K_{r} \vee I_t$ for any positive integer $t$, then $q_H=r$.
		\item Otherwise let $w$ be the largest integer with $1 \leq w \leq r$ such that $H$ is a subgraph of $K_{r-w+1} \vee I_t$ for some positive integer $t$.
			\begin{enumerate}
				\item If $H$ is not a subgraph of $K_{r-w} \vee tK_{w+1}$ for any positive integer $t$, then $q_H=(w+1)r-{w+1 \choose 2}$.
				\item Otherwise, $q_H=\max\{\min\{s+1,{r+1 \choose 2}\},(w+1)r-{w+1 \choose 2}\}$, where $s$ is the largest integer with $0 \leq s \leq {r+1 \choose 2}$ such that for every integer $s'$ with $0 \leq s' \leq s$, every connected graph $F_0$ and every graph $F \in {\mathcal F}(I_{r-w},F_0,r)$ of type $s'$, $H$ is a minor of $F \wedge_t I$ for some positive integer $t$, where $I$ is the heart of $F$.
			\end{enumerate}
	\end{enumerate}
Furthermore, $p_{\M(H)}^{\D_r}=\Theta(n^{-1/q_H})$ and $p_{\M(H)}^{\chi_r^\ell}=\Theta(n^{-1/q_H})$ in Statements 1 and 2(a). 
\end{lemma}

\begin{proof}
We first assume that $H$ is not a subgraph of $K_{r} \vee I_t$ for any positive integer $t$.
By Lemma \ref{weak general collection}, there exists a real number $c$ (only depending on $r$ and $H$) such that for every $H$-minor free graph $G$, there exists and a collection $\C$ of $r$-element subsets of $E(G)$ with $\lvert \C \rvert \leq c \lvert V(G) \rvert$ such that for every subgraph of $G$ of minimum degree at least $r$, its edge-set contains some member of $\C$.
So the threshold $p_{\M(H)}^{\D_r} = \Omega(n^{-1/ r})$ by Lemma \ref{lem:coverCandp}.
In addition, by Lemma \ref{Krsminorsubgraph}, $\M(H)$ contains $\{K_{r,s}: s \geq r\}$. 
So $p_{\M(H)}^{\D_r} = O(n^{-1/ r})$ and $p_{\M(H)}^{\chi_r^\ell} = O(n^{-1/ r})$ by Corollary \ref{upper bound threshold}. 
Thus $p_{\M(H)}^{\D_r} = \Theta(n^{-1/ r})$ and $p_{\M(H)}^{\chi_r^\ell} = \Theta(n^{-1/ r})$ by Proposition \ref{relation three properties}.
This proves Statement 1.

Now we may assume that $H$ is a subgraph of $K_r \vee I_t$ for some positive integer $t$.
So there exists the largest integer $w$ with $1 \leq w \leq r$ such that $H$ is a subgraph of $K_{r-w+1} \vee I_t$ for some positive integer $t$.
Hence there exists an integer $t_H \geq r$ such that $H$ is a subgraph of $K_{r-w+1} \vee I_{t_H}$.
Since $K_{r-w+1, t_H + {r-w+1 \choose 2}}$ contains $K_{r-w+1} \vee I_{t_H}$ as a minor, every $H$-minor free graph does not contain $K_{r-w+1, t_H + {r-w+1 \choose 2}}$ as a minor and hence does not contain $K_{r-w+1} \vee I_{t_H + {r-w+1 \choose 2}}$ as a minor.

Let $c_1$ be the number $c$ mentioned in Lemma \ref{minor collection} by taking $(r,t,t',w,s_0)=(r,t_H + {r-w+1 \choose 2},1,w, \allowbreak {r+1 \choose 2})$.
Since every $H$-minor free graph does not contain $K_{r-w+1} \vee I_{t_H + {r-w+1 \choose 2}}$ as a minor, Lemma \ref{minor collection} implies that for every $H$-minor free graph $G$, there exists a collection of subsets of $E(G)$ which is a good-$(c_1, \C_{G,1}$ of $((w+1)r-{w+1 \choose 2}), r)$ signature. 
So $p_{\M(H)}^{\D_r} = \Omega(n^{-1/((w+1)r-{w+1 \choose 2})})$ by Lemma \ref{lem:coverCandp}. 
Hence by Proposition \ref{relation three properties}, $p_{\M(H)}^{\chi^\ell_r} = \Omega(n^{-1/((w+1)r-{w+1 \choose 2})})$ by Lemma \ref{lem:coverCandp}. 

Now we assume that $H$ is not a subgraph of $K_{r-w} \vee tK_{w+1}$ for any positive integer $t$.
Note that for every positive integer $s$ with $s \geq r-w$, every minor of $I_{r-w} \vee sK_{w+1}$ is a subgraph of $K_{r-w} \vee sK_{w+1}$.
So for every positive integer $t$ with $t \geq r-w$, $I_{r-w} \vee tK_{w+1}$ does not contain $H$ as a minor.
That is, $\{I_{r-w} \vee sK_{w+1}: s \geq r-w\} \subseteq \M(H)$. 
By Corollary \ref{upper bound threshold}, $p_{\M(H)}^{\D_r}=O(n^{-1/q_H})$ and $p_{\M(H)}^{\chi_r^\ell}=O(n^{-1/q_H})$.
This proves Statement 2(a).

Hence we may assume that $H$ is a subgraph of $K_{r-w} \vee tK_{w+1}$ for some positive integer $t$.
Note that it implies that $H$ is a subgraph of $K_{r-w} \vee tK_{w+1}$ for every sufficiently large positive integer $t$.

We say that a triple $(a,F_0,F)$ is a {\it standard triple} if $a$ is a nonnegative integer, $F_0$ is a connected graph, and $F$ is a member of ${\mathcal F}(I_{r-w},F_0,r)$ of type $a$. 
Let $s$ be the largest integer with $0 \leq s \leq {r+1 \choose 2}$ such that for every integer $s'$ with $0 \leq s' \leq s$ and for every standard triple $(s',F_0,F)$, $H$ is a minor of $F \wedge_t I$ for some positive integer $t$, where $I$ is the heart of $F$. 
The number $s$ is well-defined (i.e., $s \geq 0$) since there is no graph $F$ in ${\mathcal F}(I_{r-w},F_0,r)$ of type $0$.

This definition implies that for every integer $s'$ with $0 \leq s' \leq s$ and standard triple $(s',F_0,F)$, there exists an integer $t_{s',F_0,F}$ such that $H$ is a minor of $F \wedge_{t} I$ for every integer $t$ with $t \geq t_{s',F_0,F}$, where $I$ is the heart of $F$.
In addition, for every integer $s'$ with $0 \leq s' \leq s$ and standard triple $(s',F_0,F)$, since $F_0$ is connected, we know $\lvert V(F_0) \rvert \leq \lvert E(F_0) \rvert+1 \leq s'+1 \leq {r+1 \choose 2}+1$.
So there are only finitely many different standard triple $(s',F_0,F)$ with $0 \leq s' \leq s$.
We define $t^*$ to be the maximum $t_{s',F_0,F}$ among all integers $s'$ with $0 \leq s' \leq s$ and standard triples $(s',F_0,F)$.
So $H$ is a minor of $F \wedge_{t^*} I$, where $I$ is the heart of $F$.

Applying Lemma \ref{minor collection} by taking $(r,t,t',w,s_0) = (r,t_H+{r-w+1 \choose 2},t^*,w,s)$, there exists a number $c_2$ such that for every $K_{r-w+1} \vee I_{t_H+{r-w+1 \choose 2}}$-minor free graph $G$, there exists an integer $s_G$ with $0 \leq s_G \leq s$ such that either 
	\begin{itemize}
		\item[(i)] there exists a connected graph $F_0$ such that $G$ contains $F \wedge_{t^*} I$ as a subgraph for some $F \in {\mathcal F}(I_{r-w},F_0,r)$ of type $s_G$, where $I$ is the heart of $F$, or 
		\item[(ii)] there exists a collection $\C$ of subsets of $E(G)$ which is a good-$(c_2,\min\{s+1, {r+1 \choose 2}\}, r)$ signature. 
	\end{itemize}

Let $G$ be an $H$-minor free graph.
Suppose that (i) holds for $G$.
Then there exists a connected graph $F_0$ such that $G$ contains $F \wedge_{t^*} I$ as a subgraph for some $F \in {\mathcal F}(I_{r-w},F_0,r)$ of type $s_G \leq s$, where $I$ is the heart of $F$.
By the definition of $t^*$, $H$ is a minor of $F \wedge_{t^*} I$, so $G$ contains $H$ as a minor, contradiction.
Hence (ii) holds for $G$.
Therefore, there exists a collection $\C_{G,2}$ of subsets of $E(G)$ which is a good-$(c_2,\min\{s+1, {r+1 \choose 2}\}, r)$ signature. 
Hence $p_{\M(H)}^{\D_r} = \Omega(n^{-1/\min\{s+1, {r+1 \choose 2}\}})$ by Lemma \ref{lem:coverCandp} and Statement 2(b) holds.
This proves the lemma.
\end{proof}

Note that the value $q_H$ and the conditions stated in Lemma \ref{critical exponent 1-w} is less precise than the ones stated in results in Section \ref{subsec:ourresult}.
We shall improve them in Lemmas \ref{critical exponent H min deg-w} and \ref{lemma:r-1Vee2}.
In order to do so, we need to understand the subgraphs of $K_{r-1} \vee tK_2$, as they appear in Statement 2(a) in Lemma \ref{critical exponent 1-w}.
Lemmas \ref{lemma:L_tminor} and \ref{lemma:oneisnotminor} are dedicated to this purpose, and their proofs are included in the appendix (Section \ref{subsec:appendix_pf_minor}).

\begin{lemma} \label{lemma:L_tminor}
Let $r$ be a positive integer with $r \geq 4$.
Let $H$ be a graph of minimum degree at least $r$ such that $H$ is a subgraph of $K_{r-1} \vee tK_2$ for some positive integer $t$.
Let $t^*$ be the minimum such that $H$ is a subgraph of $K_{r-1} \vee t^*K_2$.
Then either $H$ is not a minor of $L_t$ (defined in  Definition \ref{dfn:Lt}) for any positive integer $t$, or $2t^*=3q$ for some positive integer $q$. 
\end{lemma}

\begin{lemma} \label{lemma:oneisnotminor}
Let $r$ be a positive integer with $r \geq 2$.
Let $H$ be a graph of minimum degree at least $r$ such that $H$ is a subgraph of $K_{r-1} \vee tK_2$ for some positive integer $t$.
Then either 
	\begin{enumerate}
		\item $H$ is not a minor of $K_{r-2} \vee tK_3$ for any positive integer $t$, or 
		\item $r \geq 4$ and $H$ is not a minor of $L_t$ for any positive integer $t$, or
		\item $r \in \{2,3\}$ and $H=K_{r+1}$.
	\end{enumerate}
\end{lemma}

Now we are ready to improve Lemma \ref{critical exponent 1-w} by using more precise descriptions.
Lemmas \ref{critical exponent H min deg-w} and \ref{lemma:r-1Vee2} are essentially equivalent to our main results stated in Section \ref{subsec:ourresult}.
We will prove how to derive the main results in Section \ref{subsec:ourresult} from these two lemmas after they are proved.

\begin{lemma} \label{critical exponent H min deg-w}
Let $r$ be a positive integer with $r \geq 2$.
Let $H$ be a graph of minimum degree at least $r$.
Then $p_{\M(H)}^{\D_r}=\Theta(n^{-1/q_H})$, where $q_H$ is defined as follows.
	\begin{enumerate}
		\item If $H$ is not a subgraph of $K_r \vee I_t$ for any positive integer $t$, then $q_H=r$.
		\item If $H$ is a subgraph of $K_r \vee I_t$ for some positive integer $t$, and $H$ is not a subgraph of $K_{r-1} \vee tK_2$ for any positive integer $t$, then $q_H=2r-1$.
		\item If $H$ is a subgraph of $K_r \vee I_t$ and is a subgraph of $K_{r-1} \vee tK_2$ for some positive integer $t$, and $H \neq K_{r+1}$, then $q_H=s+1$, where $s$ is the largest integer with $0 \leq s \leq {r+1 \choose 2}$ such that for every integer $s'$ with $0 \leq s' \leq s$, every connected graph $F_0$ and every graph $F \in {\mathcal F}(I_{r-1},F_0,r)$ of type $s'$, $H$ is a minor of $F \wedge_t I$ for some positive integer $t$, where $I$ is the heart of $F$. 
Furthermore, $2r-1 \leq s+1 \leq {r+1 \choose 2}$. 
		\item If $H=K_{r+1}$ and $r \leq 3$, then $q_H=\infty$; if $H=K_{r+1}$ and $r \geq 4$, then $q_H=3r-3$.
	\end{enumerate}
Moreover, $p_{\M(H)}^{\chi_r^\ell}=\Theta(n^{-1/q_H})$ for Statements 1, 2 and 4.
\end{lemma}

\begin{proof}
Statement 1 immediately follows from Statement 1 of Lemma \ref{critical exponent 1-w}.

So we may assume that $H$ is a subgraph of $K_r \vee I_t$ for some positive integer $t$.
Since $H$ has minimum degree at least $r$, $H$ is not a subgraph of $K_{r-1} \vee I_t$ for any positive integer $t$.
So $1$ equals the largest integer $w$ with $1 \leq w \leq r$ such that $H$ is a subgraph of $K_{r-w+1} \vee I_t$ for some positive integer $t$.
Let $w=1$.

If $H$ is not a subgraph of $K_{r-1} \vee tK_2 = K_{r-w} \vee tK_{w+1}$ for any positive integer $t$, then $p_{\M(H)}^{\D_r}=\Theta(n^{-1/q_H})$ and $p_{\M(H)}^{\chi_r^\ell}=\Theta(n^{-1/q_H})$, where $q_H = 2r-1$ by Statement 2(a) in Lemma \ref{critical exponent 1-w}.
So Statement 2 of this lemma holds.

Hence we may assume that $H$ is a subgraph of $K_{r-1} \vee tK_2$ for some positive integer $t$.

Now we assume that $H \neq K_{r+1}$ and prove Statement 3 of this lemma.
By Lemma \ref{critical exponent 1-w}, $p_{\M(H)}^{\D_r}=\Omega(n^{-1/q_H})$, where $q_H=\max\{\min\{s+1, {r+1 \choose 2}\},2r-1\}$ and $s$ is the largest integer with $0 \leq s \leq {r+1 \choose 2}$ such that for every integer $s'$ with $0 \leq s' \leq s$, every connected graph $F_0$ and every graph $F \in {\mathcal F}(I_{r-1},F_0,r)$ of type $s'$, $H$ is a minor of $F \wedge_t I$ for some positive integer $t$, where $I$ is the heart of $F$.
For every positive integer $t$, define $F_t$ to be the graph that is the disjoint union of $I_{r-1}$ and $t$ copies of $K_{r+1}$.
Clearly, for every positive integer $t$, $F_t=F \wedge_t I$ for some $F \in {\mathcal F}(I_{r-1},K_{r+1},r)$ of type ${r+1 \choose 2}$.
Suppose that $H$ is a minor of $F_t$ for some positive integer $t$.
Since the minimum degree of $H$ is at least $r$, $H$ is a disjoint union of copies of $K_{r+1}$.
On the other hand, since $H$ is a subgraph of $K_r \vee I_t$ for some positive integer $t$, one can delete at most $r$ vertices to make $H$ edgeless.
Therefore $H$ is one copy of $K_{r+1}$.
That is, $H = K_{r+1}$, a contradiction.
So $H$ is not a minor of $F_t$ for some positive integer $t$.
In particular, $s \leq {r+1 \choose 2}-1$.
Hence, by the maximality of $s$, there exists a connected graph $F_0^*$ and a graph $F^* \in {\mathcal F}(I_{r-1},F_0^*,r)$ of type $s+1 \leq {r+1 \choose 2}$ such that $H$ is not a minor of $F^* \wedge_t I$ for any positive integer $t$, where $I$ is the heart of $F^*$.
Therefore, $\{F^* \wedge_t I: t  \in {\mathbb N}\} \subseteq \M(H)$.
By Statement 4 of Corollary \ref{upper bound threshold}, $p_{\M(H)}^{\D_r} = O(n^{-1/(s+1)})$.
In addition, since $F^* \in {\mathcal F}(I_{r-1},F_0^*,r)$, $\lvert V(F_0^*) \rvert \geq 2$.
Note that for any two vertices in $F_0^*$, there are at least $r+(r-1)=2r-1$ edges of $F^*$ incident with them.
So $s+1 \geq 2r-1$.
Hence $\max\{\min\{s+1,{r+1 \choose 2}\},2r-1\}=s+1$ and $p_{\M(H)}^{\D_r}=\Omega(n^{-1/(s+1)})$ and hence $p_{\M(H)}^{\D_r}=\Theta(n^{-1/(s+1)})$.
This proves Statement 3.

Now we assume that $H=K_{r+1}$ and prove Statement 4.

So $H$ is a subgraph of $K_r \vee I_t$ and $K_{r-1} \vee tK_2$ for some positive integer $t$. Recall that $w=1$. 
Note that for every nonnegative integer $s'$, connected graph $F_1$ and graph $F' \in {\mathcal F}(I_{r-w},F_1,r)$ of type $s'$, if $\lvert V(F_1) \rvert \geq 3$, then $s' \geq 3r-3$ since for any $S \subseteq V(F_1)$ with $\lvert S \rvert=3$, there are at least $3r-{3 \choose 2}=3r-3$ edges of $F'$ incident with $S$.
So if $F_1$ is a connected graph and $F'$ is a member of ${\mathcal F}(I_{r-w},F_1,r)$ of type at most $3r-4$, then $\lvert V(F_1) \rvert \leq 2$, so $\lvert V(F_1) \rvert=2$ since $w=1$, and hence $F'=I_{r-1} \vee K_2$.
Hence for every nonnegative integer $s'$ with $0 \leq s' \leq 3r-4$, connected graph $F_1$ and graph $F' \in {\mathcal F}(I_{r-w},F_1,r)$ of type $s'$, $H$ is a minor of $F' \wedge_t I$ for some positive integer $t$, where $I$ is the heart of $F'$.
Therefore, by Statement 2(b) in Lemma \ref{critical exponent 1-w}, $p_{\M(H)}^{\D_r} = \Omega(n^{-1/q})$, where $q \geq \max\{\min\{3r-4+1,{r+1 \choose 2}\},2r-1\} = \max\{3r-3,2r-1\}=3r-3$, since $r \geq 2$.

If $r=2$, then every $H$-minor free graph is a forest and does not contain any subgraph of minimum degree at least two, thus $G$ itself (which is also $G(p)$ where $p$ is the constant function $p=1$) is already $1$-degenerate, so $p_{\M(H)}^{\D_r} = \Theta(1)$.
If $r=3$, then $H=K_4$, and by \cite{Dirac}, every $K_4$-minor free graph contains a vertex of degree at most two, so no subgraph of any $H$-minor free graph has minimum degree at least $r= 3$, and hence $p_{\M(H)}^{\D_r} = \Theta(1)$.
Recall that $p_{\M(H)}^{\chi_r^\ell}=\Theta(1)$ when $p_{\M(H)}^{\D_r}=\Theta(1)$ by Proposition \ref{relation three properties}.

Hence we may assume that $r \geq 4$.
Since $K_{r+1} = K_{r-1} \vee K_2$, $L_t$ is $K_{r+1}$-minor free by Lemma \ref{lemma:L_tminor}.
Hence $p_{\M(H)}^{\D_r}=O(n^{-1/(3r-3)})$ and $p_{\M(H)}^{\chi_r^\ell}=O(n^{-1/(3r-3)})$ by Statement 3 of Corollary \ref{upper bound threshold}.
This completes the proof.
\end{proof}

\begin{lemma} \label{lemma:r-1Vee2}
Let $r \geq 2$ be an integer.
Let $H$ be a graph with $\delta(H) \geq r$.
If $H \neq K_{r+1}$ and $H$ is a subgraph of $K_{r-1} \vee tK_2$ and a subgraph of $K_r \vee I_t$ for some positive integer $t$, then $p_{\M(H)}^{\D_r} = \Theta(n^{-1/(3r-3)})$ and $p_{\M(H)}^{\chi_r^\ell} = \Theta(n^{-1/(3r-3)})$.
\end{lemma}

\begin{proof}
Let $t^*$ be the minimum positive integer such that $H$ is a subgraph of $K_{r-1} \vee t^* K_2$.
Since $\delta(H) \geq r$, $H$ can be obtained from $K_{r-1} \vee t^*K_2$ by deleting a set $S$ of edges contained in $K_{r-1}$.
Let $s$ be the largest integer with $0 \leq s \leq {r+1 \choose 2}$ such that for every integer $s'$ with $0 \leq s' \leq s$, every connected graph $F_0$ and every graph $F \in {\mathcal F}(I_{r-1},F_0,r)$ of type $s'$, $H$ is a minor of $F \wedge_t I$ for some positive integer $t$, where $I$ is the heart of $F$.
We shall prove that $s=3r-4$.

Suppose to the contrary that $s \leq 3r-5$.
Since $r \geq 2$, $3r-5 \leq {r+1 \choose 2}-1$.
So by the maximality of $s$, there exist an integer $s'$ with $0 \leq s' \leq 3r-5+1$, a connected graph $F_0$ and a graph $F \in {\mathcal F}(I_{r-1},F_0,r)$ of type $s'$ such that $H$ is not a minor of $F \wedge_t I$ for any positive integer $t$, where $I$ is the heart of $F$. 
If $\lvert V(F_0) \rvert \geq 3$, then for any $Z \subseteq V(F_0)$ with $|Z|=3$, there exist at least $|Z|r-{|Z| \choose 2} = 3r-3>s'$ edges of $F$ incident with $Z \subseteq V(F_0)$, a contradiction.
So $|V(F_0)| \leq 2$.
Hence $F_0=K_1$ or $K_2$.
Since the heart of $F$ has size $r-1$, $F_0=K_2$.
So $F = I_{r-1} \vee K_2$. Since $H$ is a subgraph of $K_{r-1} \vee t K_2$ which is a minor of $F \wedge_{t'} I$ where $I$ is the heart of $F$ for sufficiently large $t'$, we have 
$H$ is a minor of $F \wedge_{t'} I$ for some sufficiently large positive integer $t'$, where $I$ is the heart of $F$. This is a contradiction.

So $s \geq 3r-4$.
By Lemma \ref{lemma:oneisnotminor}, either $H$ is not a minor of $K_{r-2} \vee tK_3$ for any positive integer $t$, or $r \geq 4$ and $H$ is not a minor of $L_t$ of any positive integer $t$.
For every positive integer $t$, let $L_t'$ be the graph obtained from $I_{r-2} \vee tK_3$ by adding an isolated vertex.
Since $H$ has no isolated vertex, if $H$ is not a minor of $K_{r-2} \vee tK_3$ for any positive integer $t$, then $H$ is not a minor of $L'_t$ for any positive integer $t$.
Hence either $r \geq 4$ and $H$ is not a minor of $L_t$ for any positive integer $t$, or $H$ is not a minor of $L'_t$ for any positive integer $t$.

Note that for every positive integer $t$, $L_t = F \wedge_t I$ for some $F \in {\mathcal F}(I_{r-1},K_3,r)$ of type $3r-3$, where $I$ is the heart of $F$, and $L'_t = F' \wedge_t I'$ for some $F' \in {\mathcal F}(I_{r-1},K_3,r)$ of type $3r-3$, where $I'$ is the heart of $F'$.
So $s \leq 3r-4$.

Therefore, $s=3r-4$.
By Statement 3 of Lemma \ref{critical exponent H min deg-w}, $p_{\M(H)}^{\D_r} = \Theta(n^{-1/(s+1)})=\Theta(n^{-1/(3r-3)})$.
Hence $p_{\M(H)}^{\chi^\ell_r} = \Omega(n^{-1/(3r-3)})$.
Recall that either $H$ is not a minor of $K_{r-2} \vee tK_3$ for any positive integer $t$, or $r \geq 4$ and $H$ is not a minor of $L_t$ of any positive integer $t$.
So $p_{\M(H)}^{\chi_r^\ell}=O(n^{-1/(3r-3)})$ by Statements 2(a) and 3 of Corollary \ref{upper bound threshold}.
Therefore $p_{\M(H)}^{\chi_r^\ell} = \Theta(n^{-1/(3r-3)})$.
\end{proof}

Now we prove a lower bound for $s_r(H)$, which is part of Theorem \ref{thm:critical exponent lower bound}.
Recall that $s_r(H)$ is defined in Definition \ref{def:sr}.

\begin{lemma} \label{sr_K_c_k}
Let $r$ be a positive integer.
Let $H$ be a graph with $\tau(H) \geq 2$.
Then $s_r(H) \geq {r+1 \choose 2}-\frac{(\tau(H)-1)(\tau(H)-2)}{2}-1 = (r-\tau(H)+2)r - {r-\tau(H)+2 \choose 2}-1$.
\end{lemma}

\begin{proof}
Let $c=\tau(H)$.
Let $w=(r-c+2)r - {r-c+2 \choose 2}-1$.
Note that $w={r+1 \choose 2}-\frac{(c-1)(c-2)}{2}-1$.

Suppose to the contrary that $s_r(H) <w$. 
So there exist an integer $s'$ with $0 \leq s' \leq w$, a connected graph $F_0$ and a graph $F \in {\mathcal F}(I_{c-1},F_0,r)$ of type $s'$ such that $H$ is not a minor of $F \wedge_t I$ for any positive integer $t$, where $I$ is the heart of $F$.
Since every vertex in $F_0$ has degree in $F$ at least $r$, and every edge in $F$ is incident with some vertex in $F_0$, we have $\lvert V(F_0) \rvert \cdot r \leq \sum_{v \in V(F_0)}\deg_F(v) = \lvert E(F) \rvert + \lvert E(F_0) \rvert \leq 2s' < 2{r+1 \choose 2}$, so $\lvert V(F_0) \rvert < r+1$.
Since every vertex in $F_0$ has degree in $F$ at least $r$, $\lvert V(F) \rvert \geq r+1$, so $\lvert V(F_0) \rvert = \lvert V(F) \rvert-(c-1) \geq r+1-(c-1)=r+2-c$. 
Hence $r+2-c \leq \lvert V(F_0) \rvert \leq r$.

For every $i \in [c-1]$, let $A_i = \{v \in V(F_0): \deg_{F_0}(v)=r-i\}$.
Since every vertex in $F_0$ has degree in $F$ at least $r$, we know that for every $i$, every vertex in $A_i$ is incident with at least $i$ edges in $E(F)-E(F_0)$.
Since $\lvert V(F)-V(F_0) \rvert = c-1$, $\bigcup_{i=1}^{c-1}A_i=V(F_0)$. 
And note that if $A_i \neq \emptyset$, then $r-i \leq \lvert V(F_0) \rvert-1$, so $i \geq r+1-\lvert V(F_0) \rvert$.
So $V(F_0)=\bigcup_{i=1}^{c-1}A_i=\bigcup_{i=r+1-\lvert V(F_0) \rvert}^{c-1}A_i$.

Hence 
\begin{align*} 
	s' = \lvert E(F) \rvert \geq \lvert E(F_0) \rvert + \sum_{i=1}^{c-1}i\lvert A_i \rvert & = \frac{1}{2}\sum_{i=1}^{c-1}(r-i)\lvert A_i \rvert + \sum_{i=1}^{c-1}i\lvert A_i \rvert  \\
	& = \frac{r+1}{2}\sum_{i=1}^{c-1}\lvert A_i \rvert + \frac{1}{2}\sum_{i=1}^{c-1}(i-1)\lvert A_i \rvert \\
	& = \frac{r+1}{2}\lvert V(F_0) \rvert + \frac{1}{2}\sum_{i=1}^{c-1}(i-1)\lvert A_i \rvert \\
	& = \frac{r+1}{2}\lvert V(F_0) \rvert + \frac{1}{2}\sum_{i=r+1-\lvert V(F_0) \rvert}^{c-1}(i-1)\lvert A_i \rvert \\
	& \geq \frac{r+1}{2}\lvert V(F_0) \rvert + \frac{r-\lvert V(F_0) \rvert}{2}\sum_{i=r+1-\lvert V(F_0) \rvert}^{c-1}\lvert A_i \rvert \\
	& = \frac{2r+1-\lvert V(F_0) \rvert}{2}\lvert V(F_0) \rvert.
\end{align*}
Since $r+2-c \leq \lvert V(F_0) \rvert \leq r$, 
\begin{align*} 
s' \geq & \frac{2r+1-\lvert V(F_0) \rvert}{2}\lvert V(F_0) \rvert \geq \frac{(2r+1-(r+2-c))(r+2-c)}{2}  \\
= & \frac{(r-1+c)(r+2-c)}{2} = \frac{r^2+r-(c-1)(c-2)}{2} = {r+1 \choose 2} - \frac{(c-1)(c-2)}{2}>w,
\end{align*} a contradiction.
\end{proof}

Now we are ready to prove Theorems \ref{thm: main critical exponent 1} and \ref{thm:critical exponent lower bound} and Corollary \ref{cor:many_intro}. 
We first show a simple connection between vertex-cover and subgraphs of $K_s \vee I_t$ for some integers $s,t$.
 
\begin{lemma}\label{claim:tauHequiv}
Let $r,w,t$ be nonnegative integers such that $r \geq 1$ and $r \geq w \geq 0$.
Then the following two statements are equivalent:
\begin{enumerate}
\item $H$ is a subgraph of $K_{r-w+1} \vee I_t$ for some positive integer $t$ but not a subgraph of $K_{r-w} \vee I_t$ for any positive integer $t$; 
\item $\tau(H) = r-w+1$.
\end{enumerate}
\end{lemma}

\begin{proof}
Let $s$ be a nonnegative integer.
Note that if a graph $H$ is a subgraph of $K_s \vee I_k$ for some integer $k$, then $\tau(H) \leq s$. 
On the other hand, if $\tau(H)\leq s$, then $H$ is a subgraph of $K_s \vee I_{k}$ for any sufficiently large integer $k$ by embedding the vertices in a minimum vertex-cover into $K_s$ and the rest of the $|V(H)| - \tau(H)$ vertices to $I_k$. 
Therefore $H$ is a subgraph of $K_{r-w+1} \vee I_t$ for some positive integer $t$ is equivalent with $\tau(H) \leq r-w+1$. 
And $H$ is not a subgraph of $K_{r-w} \vee I_t$ for any positive integer $t$ is equivalent with $\tau(H) > r-w$. 
\end{proof}

\noindent{\bf Proof of Theorem \ref{thm:critical exponent lower bound}:}
Statement 1 follows from Lemma \ref{sr_K_c_k}, so it suffices to show Statement 2.
Since $2 \leq \tau(H) \leq r$, there exists $w$ with $r-1 \geq w \geq 1$ such that $\tau(H)=r-w+1$.
By Lemma \ref{claim:tauHequiv}, $w$ is the largest integer with $r-1 \geq w \geq 1$ such that $H$ is a subgraph of $K_{r-w+1} \vee I_t$ for some positive integer $t$.
Note that $w=r-\tau(H)+1$.
Since $H \in {\mathcal H}_r$, $H$ is a subgraph of $K_{r-w} \vee t^*K_{w+1}$ for some positive integer $t^*$. 
By Statement 2(b) of Lemma \ref{critical exponent 1-w}, $p_{\M(H)}^{\mathcal P} = \Omega(n^{-1/q_H})$, where $q_H=\max\{\min\{s+1,{r+1 \choose 2}\},(w+1)r-{w+1 \choose 2}\}$, where $s$ is the largest integer with $0 \leq s \leq {r+1 \choose 2}$ such that for every integer $s'$ with $0 \leq s' \leq s$, every connected graph $F_0$ and every graph $F \in {\mathcal F}(I_{r-w},F_0,r)$ of type $s'$, $H$ is a minor of $F \wedge_t I$ for some positive integer $t$, where $I$ is the heart of $F$.
Note that $w=r-\tau(H)+1$, so $s=s_r(H)$.
Hence $\max\{\min\{s+1,{r+1 \choose 2}\},(w+1)r-{w+1 \choose 2}\} = \min\{s_r(H)+1,{r+1 \choose 2}\}$ by Statement 1 of this theorem.
\hfill
$\Box$

\bigskip

\noindent{\bf Proof of Theorem \ref{thm: main critical exponent 1}:} 
If $\tau(H) \geq r+1$, then $H$ is not a subgraph of $K_r \vee I_t$ for any positive integer $t$ by Lemma \ref{claim:tauHequiv}, so $p_{\M(H)}^{\D_r}=\Theta(n^{-1/r})$ and $p_{\M(H)}^{\chi_r^\ell}=\Theta(n^{-1/r})$ by Statement 1 of Lemma \ref{critical exponent 1-w}.
So Statement 1 of Theorem \ref{thm: main critical exponent 1} holds.

Now we assume that $1 \leq \tau(H) \leq r$ and $H$ is not a subgraph of $K_{\tau(H)-1} \vee tK_{r+2-\tau(H)}$ for any positive integer $t$.
Since $1 \leq \tau(H) \leq r$, there exists $w$ with $r \geq w \geq 1$ such that $\tau(H)=r-w+1$.
So $H$ is a subgraph of $K_{r-w+1} \vee I_t$ for some positive integer $t$ but is not a subgraph of $K_{r-w} \vee I_t$ for any positive integer $t$ by Lemma \ref{claim:tauHequiv}.
Since $H$ is not a subgraph of $K_{\tau(H)-1} \vee tK_{r+2-\tau(H)}=K_{r-w} \vee tK_{w+1}$ for any positive integer $t$, $p_{\M(H)}^{\D_r} = \Theta(n^{-1/q_H})$ and $p_{\M(H)}^{\chi_r^\ell} = \Theta(n^{-1/q_H})$, where $q_H=(w+1)r-{w+1 \choose 2}$, by Statement 2(a) of Lemma \ref{critical exponent 1-w}.
Hence Statement 2 of Theorem \ref{thm: main critical exponent 1} holds.

Now we assume $\tau(H) \leq r$ and $\delta(H) \geq r$. 
Then $H$ is a subgraph of $K_r \vee I_t$ for some positive integer $t$.
If $H$ is not a subgraph of $K_{r-1} \vee tK_2$ for any positive integer $t$, then $p_{\M(H)}^{\D_r}=\Theta(n^{-1/(2r-1)})$ and $p_{\M(H)}^{\chi_r^\ell}=\Theta(n^{-1/(2r-1)})$ by Statement 2 of Lemma \ref{critical exponent H min deg-w}.
Hence Statement 3 of Theorem \ref{thm: main critical exponent 1} holds.

Now we assume $\tau(H) \leq r$, $\delta(H) \geq r$, and $H$ is a subgraph of $K_{r-1} \vee tK_2$ for some positive integer $t$.
So $H$ is a subgraph of $K_r \vee I_t$ and a subgraph of $K_{r-1} \vee tK_2$ for some positive integer $t$.
If $H \neq K_{r+1}$, then $p_{\M(H)}^{\D_r}=\Theta(n^{-1/(3r-3)})$ and $p_{\M(H)}^{\chi_r^\ell}=\Theta(n^{-1/(3r-3)})$ by Lemma \ref{lemma:r-1Vee2}.
If $H = K_{r+1}$ and $r \geq 4$, then $p_{\M(H)}^{\D_r}=\Theta(n^{-1/(3r-3)})$ and $p_{\M(H)}^{\chi_r^\ell}=\Theta(n^{-1/(3r-3)})$ by Statement 4 of Lemma \ref{critical exponent H min deg-w}.
Hence Statement 4 of Theorem \ref{thm: main critical exponent 1} holds.

Furthermore, if $\tau(H)=0$, then $H$ is edgeless, so every graph on more than $\lvert V(H) \rvert$ vertices contains $H$ as a minor, and hence $p_{\M(H)}^{\D_r}=\Theta(1)$.
If $H$ consists of $K_{1,s}$ and isolated vertices for some $s$ with $1 \leq s \leq r$, then every $H$-minor free graph on more than $\lvert V(H) \rvert$ vertices has maximum degree at most $s-1 \leq r-1$ and hence is $(r-1)$-degenerate, so $\P_{\M(H)}^{\D_r}=\Theta(1)$.
If $H=K_{r+1}$ for $r \leq 3$, then $\P_{\M(H)}^{\D_r}=\Theta(1)$ by Statement 4 of Lemma \ref{critical exponent H min deg-w}.
Recall that $p_{\M(H)}^{\chi_r^\ell}=\Theta(1)$ whenever $p_{\M(H)}^{\D_r}=\Theta(1)$ by Proposition \ref{relation three properties}.
This proves Theorem \ref{thm: main critical exponent 1}. \hfill
$\Box$

\bigskip

Now we prove Corollary \ref{cor:many_intro}.
Note that Corollary \ref{cor_planar_first} is contained in the planar case of Corollary \ref{cor:many_intro}.

\bigskip

\noindent{\bf Proof of Corollary \ref{cor:many_intro}:}
Let $\Sigma$ be a surface other than ${\mathbb S}^2$.
Let $g$ be the Euler genus of $\Sigma$.
Let $\G$ be the class of graphs embeddable in $\Sigma$.
Euler's formula implies that every graph in $\G$ is $\frac{5+\sqrt{1+24g}}{2}$-degenerate, so it is $\frac{7+\sqrt{1+24g}}{2}$-choosable.
(Note that it was also proved by Heawood \cite{h}.)
So the thresholds for $\G$ and for each of $\D_r$, $\chi_\ell^r$ and $\chi_r$ are $\Theta(1)$ when $r \geq \frac{7+\sqrt{1+24g}}{2}$.

And Euler's formula implies that $K_{3,2g+3} \not \in \G$.
So by Proposition \ref{prop:minor-closed-and-MH}, the $p_\G^{\D_r}$ is at least $p^{\D_r}_{\M(K_{3,2g+3})} = \Omega(n^{-1/q_H})$, where $q_H \geq \min\{s_r(K_{3,2g+3})+1, {r+1 \choose 2}\} \geq {r+1 \choose 2}-1$ by Theorem \ref{thm:critical exponent lower bound} (for $r \geq 3$) and Theorem \ref{thm: main critical exponent 1} (for $r=2$). 
By Proposition \ref{relation three properties}, $\Omega(n^{-1/q_H})$ is also a lower bound for $p_\G^{\chi^\ell_r}$ and $p_\G^{\chi_r}$. 

A result of B\"{o}hme, Mohar, and Stiebitz \cite{bms} implies that every graph embeddable in the Klein bottle is 6-choosable.
So the thresholds for $\chi^\ell_r$ and $\chi_r$ are $\Theta(1)$ when $r=6$ and $\Sigma$ is the Klein bottle.
In addition, $I_{r-w} \vee tK_{w+1}$ is planar for every positive integer $t$, when $r=3$ and $w=1$.
Hence by Statement 2 of Corollary \ref{upper bound threshold}, $p_\G^{\D_3}$ and $p_\G^{\chi_3^\ell}$ are both $O(n^{-1/5})$. 
Therefore, the thresholds for $\G$ and for each of $\D_3$ and $\chi_3^\ell$ are both $\Theta(n^{-1/5})$.
And the threshold for being $3$-colorable is $O(n^{-1/6})$ by Statement 4 of Corollary \ref{upper bound threshold} since $K_4 \in {\mathcal F}(I_0,K_4,3)$ with empty heart and $K_4 \wedge_t \emptyset$ is planar for any integer $t$. 
Moreover, $K_{2,t}$ is planar for any integer $t$, so Statement 1 of Corollary \ref{upper bound threshold} implies that the thresholds for $\D_2$ and $\chi_2^\ell$ are both $O(n^{-1/2})$. 
This completes the proof for the case for graphs embeddable in $\Sigma$.

Now we consider the class $\G$ of graphs with Colin de Verdi\`{e}re parameter $\mu \leq k$, for some fixed integer $k \geq 2$.
It is known \cite{hls} that for every positive integer $s$ and $t \geq \max\{s,3\}$, $\mu(K_{s,t})=s+1$.
So $K_{k,k+2} \not \in \G$.
Note that $\tau(K_{k,k+2})=k$.
By Propositions \ref{prop:minor-closed-and-MH} and \ref{relation three properties}, $p^{\D_r}_{\M(K_{k,k+2})}$ is a lower bound for the threshold for $\D_r,\chi_\ell^r,\chi^r$ and for $\G$.
By Theorem \ref{thm: main critical exponent 1}, if $r \leq k-1$, then the threshold for $\D_r$ and $\G$ is at least $p^{\D_r}_{\M(K_{k,k+2})} = \Omega(n^{-1/r})$; if $r=k$, then since $K_{k,k+2}$ has minimum degree at least $k=r$, and $K_{k,k+2}$ is not a subgraph of $K_{r-1} \vee tK_2$ for any positive integer $t$, we know the $p_\G^{\D_r} \geq p^{\D_r}_{\M(K_{k,k+2})} = \Omega(n^{-1/(2r-1)})$.
And if $r \leq k-1$, then $\mu(K_{r,t})=r+1 \leq k$ for every sufficiently large integer $t$, so $p_\G^{\D_r}=O(n^{-1/r})$ and $p_\G^{\chi^\ell_r}=O(n^{-1/r})$ by Statement 1 of Corollary \ref{upper bound threshold}.
Note that for every integer $t$, $K_{k-1} \vee tK_2$ can be obtained from a union of $t$ disjoint copies of $K_{k+1}$ by identifying a clique on $k-1$ vertices, so $\mu(K_{k-1} \vee tK_2) \leq k$ by \cite{hls}. 
Hence $\mu(I_{k-1} \vee tK_2) \leq \mu(K_{k-1} \vee tK_2) \leq k$.
So if $r=k$, Statement 2 of Corollary \ref{upper bound threshold} implies that $p_\G^{\D_r}=O(n^{-1/(2r-1)})$ and $p_\G^{\chi^\ell_r}=O(n^{-1/(2r-1)})$.
Moreover, it is known \cite{hls} that $\mu(tK_{r+1})=r$ for any positive integer $t$, so $p_\G^{\chi_r}=O(n^{-1/{r+1 \choose 2}})$ by Lemma \ref{prob general}. 
This completes the proof for the case for the graphs with Colin de Verdi\`{e}re parameter at most $k$.

Now we assume $\G$ is the class of planar graphs.
By the Four Color Theorem, every planar graph is 4-colorable, so $p_\G^{\chi_r}=\Theta(1)$ for $r \geq 4$.
Since every planar graph is 5-choosable \cite{t}, $p_\G^{\chi^\ell_r}=\Theta(1)$ for $r \geq 5$.
And Euler's formula implies that every planar graph is 5-degenerate, so $p_\G^{\D_r}=\Theta(1)$ for $r \geq 6$.
When $r \geq 4$, $\tau(K_{3,r})=3 \leq r$ and $K_{3,r}$ is not a subgraph of $K_2 \vee tK_{r-1} = K_{\tau(K_{3,r})-1} \vee tK_{r+2-\tau(K_{3,r})}$ for any positive integer $t$, so $p_\G^{\chi_r^\ell} \geq p_\G^{\D_r} \geq p_{\M(K_{3,r})}^{\D_r} =\Omega(n^{-1/((r-1)r-{r-1 \choose 2}})$ by Statement 2 of Theorem \ref{thm: main critical exponent 1}.
Note that $(r-1)r-{r-1 \choose 2} = (r-1)(r+2)/2$.
And the icosahedron is a 5-regular planar graph with 12 vertices and 30 edges.
Note that any union of disjoint copies of the icosahedron is still planar. 
So Statement 4 of Corollary \ref{upper bound threshold} implies that $p_\G^{\D_5}=O(n^{-1/30})$.
Similarly, the octahedron is a 4-regular planar graph with 12 edges.
So $p_\G^{\D_4}=O(n^{-1/12})$.
Moreover, there exists a planar non-4-choosable graph with $75$ vertices and $219$ edges \cite{g}, so $p_\G^{\chi^\ell_4}=O(n^{-1/219})$ by Lemma \ref{prob general}. 
And the case $r \leq 3$ follows from the case for the graphs with Colin de Verdi\`{e}re parameter at most $3$.
This completes the proof for the planar case.

Now we assume that $\G$ is the class of linkless embeddable graphs.
Every linkless embeddable graph is $K_6$-minor free \cite{rst}, so it is 5-colorable \cite{rst_color}.
So $p_\G^{\chi_5}=\Theta(1)$.
And $K_{4,4} \not \in \G$ \cite{s}. 
Note that $\tau(K_{4,4})=4 \leq 5$ and $K_{4,4}$ is not a subgraph of $K_3 \vee tK_3 = K_{\tau(K_{4,4})-1} \vee tK_{5+2-\tau(K_{4,4})}$ for any positive integer $t$.
So $p_\G^{\chi_5^\ell} \geq p_\G^{\D_5} \geq p_{\M(K_{4,4})}^{\D_5} =\Omega(n^{-1/12})$ by Statement 2 of Theorem \ref{thm: main critical exponent 1}.
And the icosahedron is a 5-regular planar graph with 12 vertices and 30 edges.
So the graph obtained by adding a new vertex adjacent to the disjoint union of any number of copies of icosahedron is linkless embeddable.
Note that such graphs are $F \wedge_t V(I_1)$ for integers $t$, where $F$ is in ${\mathcal F}(I_1,F_0,6)$ with type $30+12$ and $F_0$ is the icosahedron.
So Statement 4 of Corollary \ref{upper bound threshold} implies that $p_\G^{\D_5}=O(n^{-1/42})$.
Similarly, since there exists a planar non-4-choosable graph with $75$ vertices and $219$ edges, we can obtain a non-5-choosable $K_6$-minor free graph by first taking five disjoint copies of that graph and then adding a new vertex adjacent to all other vertices.
So there exists a non-5-choosable apex graph with $5 \cdot (219+75)=1470$ edges, so $p_\G^{\chi^\ell_5}=O(n^{-1/1470})$.
And the case $r \leq 4$ follows from the case for the graphs with Colin de Verdi\`{e}re parameter at most $4$.
This completes the proof of the linkless embeddable case.

Finally, let $\G$ be the class of outerplanar graphs.
It is well-known that every outerplanar graph is 2-degenerate and hence is 3-choosable and 3-colorable.
So $p_\G^{\D_r}=p_\G^{\chi_r}=p_\G^{\chi_r^\ell}=\Theta(1)$ for every $r \geq 3$.
The case $r=2$ follows from the $r=k=2$ case of the class of graphs with Colin de Verdi\`{e}re parameter $\mu \leq k$.
This completes the proof of the corollary.
\hfill
$\Box$

\section{Concluding remarks} \label{sec:concluding remarks}
In this paper, we initiate a systematic study of threshold probabilities for monotone properties in the random model $G(p)$, where $G$ belongs to a given proper minor-closed family. 
Minor-closed families are natural classes of sparse graphs, and results about threshold probability on sparse graph classes in the literature seem rare.

Theorems \ref{thm: main critical exponent 1} and \ref{thm:critical exponent lower bound} address the properties $\D_r$ (being $(r-1)$-degenerate) and $\chi_r^\ell$ (being $r$-choosable).
Those results immediately provide lower bound for the thresholds for the properties $\R_r$ (non-existence of $r$-regular subgraphs) and $\chi_r$ (being $r$-colorable) by Proposition \ref{relation three properties}.
We do not put any effort on the results for the properties $\R_r$ and $\chi_r$ in this paper; but for some minor-closed families $\M(H)$, the machinery developed in this paper easily provides matching upper bounds, as described in the following two theorems whose proofs are simple and left in the appendix.

\begin{theorem} \label{thm:critical exponent regular intro}
Let $r \geq 2$ be an integer and $H$ a graph. 
Then $p_{\M(H)}^{\R_r}$ is $\Theta(n^{-1/q_H})$, where $q_H$ is defined as follows. 
	\begin{enumerate}
		\item If $\tau(H) \geq r+1$, then $q_H=r$. 
		\item If $1 \leq \tau(H) \leq r$, $r$ is divisible by $r+2 - \tau(H)$ and $H$ is not a subgraph of $K_{\tau(H)-1} \vee tK_{r+2-\tau(H)}$ for any positive integer $t$, then $q_H = (r+2-\tau(H))r - {r+2-\tau(H) \choose 2}$.
		\item If $1 \leq \tau(H) \leq r$, $r$ is even, $H$ has minimum degree at least $r$ and $H$ is not a subgraph of $K_{r-1} \vee tK_2$ for any positive integer $t$, then $q_H=2r-1$.
	\end{enumerate}
Furthermore, if either $H = K_{r+1}$ and $r \leq 3$, or $H=K_{1,s}$ for some $s \leq r$ , then $p_{\M(H)}^{\R_r}=\Theta(1)$.
\end{theorem}

\begin{theorem} \label{thm:critical exponent colorable intro}
Let $r \geq 2$ be an integer and let $H$ be a graph.
Then the following hold.
	\begin{enumerate}
		\item If $1 \leq \tau(H) \leq 2$ and $H$ is not a subgraph of $K_1 \vee t K_{r}$ for any positive integer $t$, then $p_{\M(H)}^{\chi_r}=\Theta(n^{-2/(r(r+1))})$. 
		\item If either $H = K_{r+1}$ and $r\leq 3$, or $H$ has at most one component on more than two vertices and every component of $H$ is an isolated vertex or a star of maximum degree at most $r$, then $p_{\M(H)}^{\chi_r}=\Theta(1)$.
	\end{enumerate}
\end{theorem}

As we mentioned above, we do not try to strengthen results for $p_{\M(H)}^{\R_r}$ and $p_{\M(H)}^{\chi_r}$ in this paper.
We leave it for future research.

\begin{question}
For any integer $r \geq 2$ and graph $H$, what are $p_{\M(H)}^{\R_r}$ and $p_{\M(H)}^{\chi_r}$?
And more generally, what are $p_\G^{\R_r}$ and $p_{\G}^{\chi_r}$ for any given proper minor-closed family?
\end{question}

The threshold studied in this paper is also called the crude threshold. 
A function $p^*: \mathbb{N} \to [0,1]$ is an {\it (upper) sharp threshold} for a graph class $\G$ and a monotone property $\P$ if the following hold.
 \begin{enumerate}
	\item for every sequence $(G_i)_{i \in {\mathbb N}}$ of graphs with $G_i \in \G$ and $\lvert V(G_i) \rvert \to \infty$ and any $\epsilon >0$ , the random subgraphs $G_i((1-\epsilon)p(n_i))$ are in $\P$ a.a.s.\ where $n_i = |V(G_i)|$;
	\item there is some sequence $(G_i)_{i \in {\mathbb N}}$ of graphs with $G_i \in \G$ and $\lvert V(G_i) \rvert \to \infty$ such that for any $\epsilon > 0$, the random subgraphs $G_i((1+\epsilon)p(n_i))$ are not in $\P$ a.a.s.\ where $n_i = |V(G_i)|$. 
\end{enumerate}

Friedgut \cite{sharpthreshold} provides a necessary and sufficient condition to check whether there is a sharp threshold for a general class of random  models. However it is not an easy task to apply to our model.
The next natural question is:
\begin{question}
What are the sharp thresholds for properties $\D_r,\chi_r^\ell,\chi_r,\R_r$ for minor-closed families?
\end{question}

It is also interesting to study other global properties, where some natural algorithms are NP-hard even on some proper minor-closed families, such as the set of planar graphs.

\paragraph{Acknowledgement} The authors would like to thank Pasin Manurangsi for bringing up the question for $3$-colorability in planar graphs to the second author, thank Jacob Fox, Sivakanth Gopi, and Ilya Razenshteyn for helpful discussions, and thank the anonymous referee for careful reading and comments.

\appendix
\section{Appendix}

\subsection{Degeneracy function and extremal function}
\begin{proposition} \label{equiv_thr_den}
For every integer $r$ with $r \geq 2$ and every connected graph $H$, $p_{\M(H)}^{\D_r}=\Theta(1)$ if and only if $r-1 \geq d_H^*$. 
\end{proposition}
\begin{proof}
Since $d_H(n)$ is non-decreasing in $n$, we know $r-1 \geq d_H^*$ if and only if $r-1 \geq d_H(n)$ for every $n \in {\mathbb N}$.
Hence it suffices to prove that $\P_{\M(H)}^{\D_r}=\Theta(1)$ if and only if $r-1 \geq d_H(n)$ for every $n \in {\mathbb N}$.

If $r-1 \geq d_H(n)$ for every $n \in {\mathbb N}$, then every graph $G \in \M(H)$ on sufficiently many vertices is already $(r-1)$-degenerate and thus the threshold probability is $\Theta(1)$. 

Now we show that $\thredeg = \Theta(1)$ implies that $r-1 \geq d_H(n)$ for every $n \in {\mathbb N}$.  
For every graph $G \in \M(H)$, let $p(G)$ be the supremum of all $p$ such that  the random subgraph $G(p)$ is $(r-1)$-degenerate with probability at least $0.9$. 
Note that such $p(G)$ exists since degeneracy is a monotone property. 
For every $n \in {\mathbb N}$, let $p(n)$ be the minimum of $p(G)$ among all graphs $G \in \M(H)$ on $n$ vertices.
Note that there are only finite number of graphs on $n$ vertices.
Since adding isolated vertices to any $G \in \M(H)$ results in a $G'\in \M(H)$ on more vertices, and $p(G) = p(G')$, the function $p$ is non-increasing.
Hence $\lim_{n \to \infty} p(n)$ exists.

Let $p^* = \lim_{n \to \infty} p(n)$.  
We claim that $p^*=1$ or $p^*=0$.
Suppose to the contrary that $0< p^* < 1$.
Let $p'$ be any real number with $0< p' < p^*$. 
Let $G \in \M(H)$ be a graph such that $p(G) < 1$, and let $a$ be the probability that $G(p')$ is $(r-1)$-degenerate. 
Thus $0.9 \leq a$.
Since $p(G)<1$, $G$ is not $(r-1)$-degenerate, $a \leq 1 - {p'}^{e(G)}$. 
In particular, $0< a < 1$. 
For every $k \in {\mathbb N}$, let $G_k$ be a union of $k$ disjoint copies of $G$. 
Thus when $k \geq \lceil \log_a(1/2) \rceil$, the probability that at least one copy of $G_k(p')$ is not $(r-1)$-degenerate is $1-a^k \geq 1- a^{\lceil \log_a(1/2) \rceil} \geq 0.5>0.1$.
So $p(G_k) \leq p'$ for every $k \geq \lceil \log_a(1/2) \rceil$.
That is, $p(\lvert V(G_k) \rvert) \leq p'$ for every $k \geq \lceil \log_a(1/2) \rceil$.
Hence $(p(n_k): k \geq \lceil \log_a(1/2) \rceil)$ is a subsequence of $(p(n): n \in {\mathbb N})$, where $n_k = \lvert V(G_k) \rvert$, such that $p(n_k) \leq p'$ for every $k \geq \lceil \log_a(1/2) \rceil$.
Therefore, $p^* \leq p'$, a contradiction.

Suppose $p_{\M(H)}^{\D_r}=\Theta(1)$ and $p^*=0$.
Let $q(n)=\min\{p(n)+\frac{1}{n},1\}$ for every $n \in {\mathbb N}$.
Since $q(n)>p(n)$ for every $n \in {\mathbb N}$, there exist $G_1,G_2,...$ such that $\lvert V(G_n) \rvert = n$ and $\Pr(G_n(q) \in \D_r)<0.9$ for every $n \in {\mathbb N}$.
Hence $\lim_{n \rightarrow \infty} \frac{q(n)}{1} \leq p^*=0$, but $\lim_{n \rightarrow \infty}\Pr(G_n(q) \in \D_r) \neq 1$, contradicting $p_{\M(H)}^{\D_r}=\Theta(1)$.

Therefore, if $\thredeg = \Theta(1)$, then $p^* = 1$. 
Thus $p(n) = 1$ for all $n \in {\mathbb N}$ since $p(n)$ is non-increasing in $n$.
Suppose that there exists $G \in \M(H)$ such that $G$ is not $(r-1)$-degenerate.
Let $w=(0.2)^{1/n^2}$.
Then $\Pr(G(w) \in \D_r) = 1-\Pr(G(w) \not \in \D_r) \leq 1-\Pr(G(w)=G)=1-w^{\lvert E(G) \rvert} \leq 1-w^{n^2}<0.9$.
So $p(G) \leq w$ and hence $p(\lvert V(G) \rvert)<1$, a contradiction.

Therefore every graph in $\M(H)$ is $(r-1)$-degenerate, which is equivalent to $r-1 \geq d_H(n)$ for every $n \in {\mathbb N}$. 
\end{proof}

\begin{proposition} \label{2approx}
Let $H$ be a graph.
Then $f_H^* \leq d^*_H \leq 2f_H^*$.
\end{proposition}

\begin{proof}
By the definition of $d_H$, every $H$-minor free graph $G$ is $d_H(\lvert V(G) \rvert)$-degenerate, so $G$ contains a vertex of degree at most $d_H(\lvert V(G) \rvert)$.
Hence every $H$-minor free graph on $n$ vertices contains at most $\sum_{i=1}^nd_H(i)$ edges by induction.
Since $d_H$ is non-decreasing, every $H$-minor free graph on $n$ vertices contains at most $\sum_{i=1}^nd_H(i) \leq d_H(n)n$ edges.
That is, $f_H(n) \leq d_H(n)n$ for every $n \in {\mathbb N}$.
Hence $f_H^*=\sup_{n \in {\mathbb N}} \frac{f_H(n)}{n} \leq \sup_{n \in {\mathbb N}}d_H(n) = d^*$.

By the definition of $f_H^*$, $\lvert E(G) \rvert/ \lvert V(H) \rvert \leq f_H^*$ for every $H$-minor free graph $G$.
So every $H$-minor free graph $G$ contains a vertex of degree at most $2\lvert E(G) \rvert/\lvert V(G) \rvert \leq 2f_H^*$.
Hence every $H$-minor free graph is $2f_H^*$-degenerate.
That is, $d_H(n) \leq 2f_H^*$ for every $n \in {\mathbb N}$.
Therefore, $d^*=\sup_{n \in {\mathbb N}}d_H(n) \leq 2f_H^*$.
\end{proof}

\subsection{Proof of Lemma \ref{nonchoosable}} \label{subsec:appendix_pf_color}

\begin{proof}[{\bf Proof of Lemma \ref{nonchoosable}:}]
Denote the vertices in $V(I_{r-w})$ by $v_1,v_2,...,v_{r-w}$.
For each $i$ with $1 \leq i \leq r-w$, define a list of $r$ colors $L_{v_i}=\{ri+j: 0 \leq j \leq r-1\}$. Thus $L_{v_i} \cap L_{v_j} = \emptyset$ for $1 \leq i < j \leq r-w$.
And for each vertex $v$ in $V(I_{r-w} \vee r^{r-w} K_{w+1})-\{v_1,v_2,...,v_{r-w}\}$, we define $L_v$ to be a set of size $r$ that is a union of $\{-1,-2,...,-w\}$ and a set $S_v$ with $|S_v \cap L_{v_i}| = 1$ for every $1 \leq i \leq r-w$, such that for every distinct vertices $x,y \in V(I_{r-w} \vee r^{r-w} K_{w+1})-\{v_1,v_2,...,v_{r-w}\}$, $L_x=L_y$ if and only if $x,y$ are in the same component of $(I_{r-w} \vee r^{r-w} K_{w+1})-\{v_1,v_2,...,v_{r-w}\}$. 
This is possible since there are $r^{r-w}$ components and there are $r^{r-w}$ ways to pick precisely one element from each size-$r$ list $L_{v_i}$ for $1 \leq i \leq r-w$.

Suppose to the contrary that $I_{r-w} \vee r^{r-w}K_{w+1}$ is $r$-choosable.
Then there exists a function $f$ such that $f(v) \in L_v$ for every $v \in I_{r-w} \vee r^{r-w}K_{w+1}$, and $f(x) \neq f(y)$ for every adjacent vertices $x,y$.
By construction, there exists a component $C$ of $(I_{r-w} \vee r^{r-w}K_{w+1})-\{v_i: 1 \leq i \leq r-w\}$ such that $L_v-\{f(v_i): 1 \leq i \leq r-w\} = \{-1,-2,...,-w\}$ for every $v \in V(C)$.
Since $\lvert V(C) \rvert = w+1$ and $L_v-\{f(v_i): 1 \leq i \leq r-w\} = \{-1,-2,...,-w\}$ for every $v \in V(C)$, there exist two distinct vertices $x,y$ of $C$ such that $f(x)=f(y)$.
Since $C$ is isomorphic to $K_{w+1}$, $x$ is adjacent to $y$, a contradiction.
Therefore, $I_{r-w} \vee r^{r-w}K_{w+1}$ is not $r$-choosable.
\end{proof}

\subsection {Proof of Lemmas \ref{lemma:L_tminor} and \ref{lemma:oneisnotminor}} \label{subsec:appendix_pf_minor}

We need the following auxiliary lemma.

\begin{lemma} \label{lemma:Veeminor}
Let $r$ be a positive integer with $r \geq 2$.
Let $H$ be a graph of minimum degree at least $r$ such that $H$ is a subgraph of $K_{r-1} \vee tK_2$ for some positive integer $t$.
Let $t^*$ be the minimum such that $H$ is a subgraph of $K_{r-1} \vee t^*K_2$.
Then either $H$ is not a minor of $K_{r-2} \vee tK_3$ for any positive integer $t$, or $2t^*=3q-1$ for some positive integer $q$. 
\end{lemma}

\begin{proof}
We may assume that there exists an $H$-minor $\alpha$ in $K_{r-2} \vee tK_3$ for some positive integer $t$, for otherwise we are done.
Let $Y$ be the vertex-set $V(K_{r-2})$ in $K_{r-2} \vee tK_3$.

Since $\delta(H) \geq r$ and $H$ is a subgraph of $K_{r-1} \vee t^*K_2$, $H=(K_{r-1} \vee t^*K_2)-S$, where $S$ is a set of edges of $K_{r-1} \vee t^*K_2$ in $E(K_{r-1})$.
Hence $I_{r-1} \vee t^*K_2 \subseteq H \subseteq K_{r-1} \vee t^*K_2$.
We call each vertex of $H$ in $V(K_{r-1})$ an {\it inner vertex}, and call each vertex of $H$ in $V(t^*K_2)$ an {\it outer vertex}.

\begin{claim} \label{claim:Veeminor}
Let $A_1$ be a branch set of $\alpha$ disjoint from $Y$.
Let $X$ be the vertex-set of the component of $(K_{r-2} \vee tK_3)-Y$ intersecting $A_1$.
Then the following hold.
	\begin{enumerate}
		\item $A_1$ consists of one vertex.
		\item $X$ is a union of three branch sets of $\alpha$.
		\item Every vertex in $Y$ belongs to a branch set, and different vertices of $Y$ belong to different branch sets of $\alpha$.
		\item either $A_1$ is a branch set corresponding to an inner vertex, or $t^*=1$.
	\end{enumerate}
\end{claim}

\noindent{\bf Proof of Claim \ref{claim:Veeminor}}
Since $\delta(H) \geq r$, $A_1$ is adjacent in $K_{r-2} \vee tK_3$ to at least $r$ other branch sets of $\alpha$.
Since $A_1$ is disjoint from $Y$, $|A_1|=1$.
Hence every vertex in $Y \cup (X-A_1)$ belongs to a branch set, and different vertices in $Y \cup (X-A_1)$ belong to different branch sets.
So Statements 1-3 hold.

Assume that $A_1$ is a branch set corresponding to an outer vertex.
Since every outer vertex is adjacent to all inner vertices, each branch set corresponding to an inner vertex either intersects $Y$ or is contained in $X$.
Since there are $r-1$ inner vertices and $|Y|=r-2$, there exists an inner vertex whose branch set is contained in $X$, so every branch set corresponding to an outer vertex intersects $Y \cup X$.
Hence there are at most $|X \cup Y|-2t^* = r+1-2t^*$ branch sets corresponding to inner vertices adjacent to $A_1$.
So $r+1-2t^* \geq r-1$.
That is, $t^*=1$.
So Statement 4 holds.
$\Box$

Since $|Y|=r-2$ and $|V(H)|=r-1+2t^*>r-2$, there exists a vertex $v$ of $H$ such that the branch vertex corresponding to $v$ in $\alpha$ is disjoint from $Y$.
Hence there exist a positive integer $q$ and components $C_1,C_2,...,C_q$ of $(K_{r-2} \vee tK_3)-Y$ such that those $C_i$ are the components of $(K_{r-2} \vee tK_3)-Y$ containing some branch sets disjoint from $Y$.
We may assume that $t^* \neq 1$, for otherwise $2t^*=3-1$ and we are done.
So by Claim \ref{claim:Veeminor}, for each $i \in [q]$, $V(C_i)$ is the union of three branch sets of $\alpha$ corresponding to inner vertices.
So the number of inner vertices whose branch sets are disjoint from $Y$ is $3q$.

Since each outer vertex is adjacent to all inner vertices, each branch set corresponding to an outer vertex intersects $Y$ and hence contains exactly one vertex in $Y$ (by Claim \ref{claim:Veeminor}).
Hence by Claim \ref{claim:Veeminor}, there are exactly $|Y|-2t^* = r-2-2t^*$ branch sets corresponding to inner vertices intersecting $Y$.
Therefore, the number of inner vertices is $3q+r-2-2t^*$.
In addition, the number of inner vertices is $|V(I_{r-1})|=r-1$.
Hence $2t^*=3q-1$.
This proves the lemma.
\end{proof}

\bigskip

\begin{proof}[{\bf Proof of Lemma \ref{lemma:L_tminor}:}]
Let us recall the definition of $L_t$.
Let $Y$ be the stable set of size $r-1$ in $I_{r-1} \vee K_3$ corresponding to $V(I_{r-1})$, and let $X=V(I_{r-1} \vee K_3)-Y$.
Let $L$ be a connected graph obtained from $I_{r-1} \vee K_3$ by deleting the edges of a matching of size three between $X$ and $Y$.
Denote $Y=\{y_1,y_2,...,y_{r-1}\}$.
For every positive integer $t$, $L_t$ is the graph obtained from a union of disjoint $t$ copies of $L$ by for each $i$ with $1 \leq i \leq r-1$, identifying the $y_i$ in each copy of $L$ into a new vertex $y_i^*$.

We may assume that there exists an $H$-minor $\alpha$ in $L_t$, for otherwise we are done. 
Since $\delta(H) \geq r$ and $H$ is a subgraph of $K_{r-1} \vee t^*K_2$, $I_{r-1} \vee t^*K_2 \subseteq H \subseteq K_{r-1} \vee t^*K_2$.
We call each vertex of $H$ in $V(K_{r-1})$ an {\it inner vertex}, and call each vertex of $H$ in $V(t^*K_2)$ an {\it outer vertex}.

\begin{claim} \label{claim:L_tminor}
Let $A_1$ be a branch set of $\alpha$ disjoint from $Y$.
Let $Z$ be the vertex-set of the component of $L_t-Y$ intersecting $A_1$.
Then the following hold.
	\begin{itemize}
		\item $A_1$ consists of one vertex.
		\item $Z$ is a union of three branch sets of $\alpha$.
		\item Every vertex in $Y$ belongs to a branch set, and different vertices of $Y$ belong to different branch sets of $\alpha$.
		\item $A_1$ is a branch set corresponding to an inner vertex.
	\end{itemize}
\end{claim}

\noindent{\bf Proof of Claim \ref{claim:L_tminor}:}
Since $\delta(H) \geq r$, $A_1$ is adjacent in $L_t$ to at least $r$ other branch sets of $\alpha$.
So $1 \leq |A_1| \leq 2$.

Suppose $|A_1|=2$.
Then $|Y \cup (Z-A_1)|=r$.
So each vertex in $Y \cup (Z-A_1)$ is contained in a branch set of $\alpha$, and different vertices in $Y \cup (Z-A_1)$ are contained in different branch sets.
So some branch set of $\alpha$ consists of the single vertex $u$ in $Z-A_1$.
Since $u$ is nonadjacent in $L_t$ to some vertex in $Y$, the branch set consisting of $u$ is adjacent to at most $(|Y|-1)+1 =r-1$ branch sets of $\alpha$, contradicting $\delta(H) \geq r$.

So $|A_1|=1$ and Statement 1 holds.
Let $x_1$ be the vertex in $A_1$.
By symmetry, we may assume that $y_1$ is the vertex in $Y$ nonadjacent to $x_1$ in $L_t$.
Since $A_1 \cap Y = \emptyset$ and $\delta(H) \geq r$, each vertex in $(Y-\{y_1\}) \cup (Z-A_1)$ is contained in a branch set of $\alpha$, and different vertices in $(Y-\{y_1\}) \cup (Z-A_1)$ are contained in different branch sets of $\alpha$.
This implies that there exist two different branch sets $A_2,A_3$ of $\alpha$ other than $A_1$ such that $A_2 \cap Z \neq \emptyset \neq A_3 \cap Z$, and one of $A_2,A_3$ is disjoint from $Y$.
By symmetry, we may assume that $A_2$ is disjoint from $Y$.
So $|A_2|=1$.
Let $x_2$ be the vertex in $A_2$.
By symmetry, we may assume that $y_2$ is the vertex in $Y$ nonadjacent to $x_2$ in $L_t$.
Since $\delta(H) \geq r$, each vertex in $(Y-\{y_2\}) \cup (Z-A_2)$ is contained in a branch set of $\alpha$, and different vertices in $(Y-\{y_2\}) \cup (Z-A_2)$ are contained in different branch sets of $\alpha$.
This implies that $y_1 \not \in A_3$.
So $A_3$ consists of one vertex, say $x_3$, in $Z$.
Hence $Z$ is a union of three branch sets $A_1,A_2,A_3$ of $\alpha$, where each of $A_i$ consists of one vertex.
So Statement 2 holds.

By symmetry, let $y_3$ be the vertex in $Y$ nonadjacent to $x_3$ in $L_t$.
Since $\delta(H) \geq r$, each vertex in $(Y-\{y_3\}) \cup (Z-A_3)$ is contained in a branch set of $\alpha$, and different vertices in $(Y-\{y_3\}) \cup (Z-A_3)$ are contained in different branch sets of $\alpha$.
So $y_1$ and $y_2$ are contained in different branch sets.
Hence each vertex of $Y$ is contained in a branch set of $\alpha$ other than $A_1,A_2,A_3$, and different vertices of $Y$ are contained in different branch sets of $\alpha$.
This proves Statement 3.

Suppose that $A_1$ is the branch set of $\alpha$ corresponding to an outer vertex $v_1$ of $H$.
Let $v_1'$ be the outer vertex of $H$ adjacent to $v_1$ in $H$.
Since the neighbors of $v_1$ are $v_1'$ and the $r-1$ inner vertices, $y_1$ is contained in the branch set of $\alpha$ corresponding to an outer vertex other than $v_1'$.
Suppose some of $A_2,A_3$, say $A_2$, is the branch set of $\alpha$ corresponding to an outer vertex $v_2$ of $H$.
Then $y_2$ is contained in the branch set of $\alpha$ corresponding to an outer vertex.
So there are at most $(|Y|-2) + (|Z|-2) \leq r-2$ branch sets corresponding to an inner vertex intersecting $(Y-\{y_1\}) \cup Z$.
Since there are $r-1$ inner vertices, $A_1$ is nonadjacent to some branch vertex corresponding to an inner vertex, a contradiction.
So each of $A_2,A_3$ is the branch set corresponding to an inner vertex.
Hence every branch set corresponding to an outer vertex other than $v_1$ intersects $Y$.
So there are at most $|Y|-(2t^*-1) \leq r-2t^*$ branch sets corresponding to an inner vertex intersecting $Y$.
Since $A_1$ is adjacent to $r-1$ branch sets corresponding to inner vertices, $r-2t^*+2 = r-2t^*+(|Z|-1) \geq r-1$, we know $t^*=1$.
So $v_1$ and $v_1'$ are the only outer vertices.
But $y_1$ is contained in the branch set of $\alpha$ corresponding to an outer vertex other than $v_1'$, a contradiction.
Statement 4 is thus proved.
$\Box$

Since $|Y|=r-1$ and $|V(H)|=r-1+2t^*>r-1$, there exists a vertex $v$ of $H$ such that the branch set corresponding to $v$ in $\alpha$ is disjoint from $Y$.
Hence there exist a positive integer $q$ and components $C_1,C_2,...,C_q$ of $L_t-Y$ such that those $C_i$ are the components of $L_t-Y$ containing some branch sets disjoint from $Y$.
By Claim \ref{claim:L_tminor}, for each $i \in [q]$, $V(C_i)$ is the union of three branch sets of $\alpha$ corresponding to inner vertices.
So the number of inner vertices whose branch sets are disjoint from $Y$ is $3q$.

Since each outer vertex is adjacent to all inner vertices, each branch set corresponding to an outer vertex intersects $Y$ and hence contains exactly one vertex in $Y$ (by Claim \ref{claim:L_tminor}).
Hence by Claim \ref{claim:L_tminor}, there are exactly $|Y|-2t^* = r-1-2t^*$ branch sets corresponding to an inner vertex intersecting $Y$.

Therefore, the number of inner vertices is $3q+r-1-2t^*$.
In addition, the number of inner vertices is $|V(I_{r-1})|=r-1$.
Hence $2t^*=3q$.
This proves the lemma.
\end{proof}

\bigskip

\begin{proof}[{\bf Proof of Lemma \ref{lemma:oneisnotminor}:}] 
When $r \geq 4$, Statements 1 or 2 hold by Lemmas \ref{lemma:Veeminor} and \ref{lemma:L_tminor}.
So we may assume that $r \in \{2,3\}$.
We may assume that $H$ is a minor of $K_{r-2} \vee tK_3$ for some positive integer $t$, for otherwise we are done.
Note that for any positive integer $t$, every minor of $K_{r-2} \vee tK_3$ is a subgraph of $K_{r-2} \vee tK_3$.
So $H$ is a subgraph of $K_{r-2} \vee tK_3$ for some positive integer $t$.

When $r=2$, $H$ is a subgraph of $K_{r-1} \vee tK_2 = K_1 \vee tK_2$ and a subgraph of $tK_3$ for some positive integer $t$, so $H=K_3=K_{r+1}$ since $\delta(H) \geq 2$.
So we may assume $r=3$.
Hence $H$ is a subgraph of $K_2 \vee tK_2$ and a subgraph of $K_1 \vee tK_3$ for some positive integer $t$.
Since $\delta(H) \geq 3$ and $H$ is a subgraph of $K_2 \vee tK_2$ for some positive integer $t$, there exists a positive integer $t^*$ such that $H=K_2 \vee t^*K_2$ or $H=I_2 \vee t^*K_2$.
In particular, $H$ is 2-connected.
Since $H$ is a subgraph of $K_1 \vee tK_3$ for some positive integer $t$ and $\delta(H) \geq 3$, either $H=K_4$ or $H$ has a cut-vertex.
So $H=K_4$.
This proves the lemma.
\end{proof}

\subsection{Results on being $r$-colorable, having no $r$-regular subgraphs}\label{subsubsec:otherproperty}

\begin{proof}[{\bf Proof of Theorem \ref{thm:critical exponent regular intro}:}]
Statement 1 holds by Statement 1 of Corollary \ref{upper bound threshold}, Lemma \ref{Krsminorsubgraph}, Proposition \ref{relation three properties} and Statement 1 in Theorem \ref{thm: main critical exponent 1}.
Now we can assume $1 \leq \tau(H) \leq r$. 
So there exists an integer $w$ with $1 \leq w \leq r$ such that $\tau(H)=r-w+1$.

We first prove Statement 2.
So $r$ is divisible by $w+1$ and $H$ is not a subgraph of $K_{r-w} \vee tK_{w+1}$ for any positive integer $t$.
Since every minor of $I_{r-w} \vee tK_{w+1}$ is a subgraph of $K_{r-w} \vee tK_{w+1}$, $\{I_{r-w} \vee sK_{w+1}: s \geq s_0\} \subseteq \M(H)$ for some sufficiently large $s_0$.
Hence Statement 2 of this theorem follows from Statement 2 of Corollary \ref{upper bound threshold}, Statement 2 of Theorem \ref{thm: main critical exponent 1} and Proposition \ref{relation three properties}.

Now we prove Statement 3.
Note that for any positive integer $t$, every minor of $I_{r-1} \vee t K_2$ is a subgraph of $K_{r-1}\vee tK_2$.
Hence $\{I_{r-1} \vee sK_2: s \in {\mathbb N}\} \subseteq \M(H)$.
And $K_{r+1}=K_{r-1} \vee K_2$, so $H \neq K_{r+1}$.
Hence Statement 3 of this theorem follows from Statement 2(c) of Corollary \ref{upper bound threshold} by taking $w=1$, Statement 3 of Theorem \ref{thm: main critical exponent 1} and Proposition \ref{relation three properties}.

If either $H=K_{r+1}$ and $r \leq 3$, or $H=K_{1,s}$ for some $s \leq r$, then every graph in $\M(H)$ is $(r-1)$-degenerate and hence $p_{\M(H)}^{\R_r}=p_{\M(H)}^{\D_r}=\Theta(1)$. 
\end{proof}

\bigskip

\begin{proof}[{\bf Proof of Theorem \ref{thm:critical exponent colorable intro}:}]
We first prove Statement 1.
If $\tau(H)=1$, then $H$ is a disjoint union of a star and isolated vertices, so $H$ is a subgraph of $K_1 \vee tK_r$ for some positive integer $t$, a contradiction.
So $\tau(H)=2$.
Hence Statement 2 in Theorem \ref{thm: main critical exponent 1} and Proposition \ref{relation three properties} implies that $p_{\M(H)}^{\chi_r} = \Omega(n^{-2/(r(r+1))})$.
Since every minor of $K_1 \vee tK_r$ is a subgraph of $K_1 \vee tK_r$, we know $\{I_1 \vee sK_r: s \in {\mathbb N}\} \subseteq \M(H)$.
By Statement 2(b) of Corollary \ref{upper bound threshold} by taking $w=r-1$, we know $p_{\M(H)}^{\chi_r}=O(n^{-2/(r(r+1))})$.
This proves Statement 1.

Statement 2 follows from the last sentence of Theorem \ref{thm: main critical exponent 1} and Proposition \ref{relation three properties}. 
\end{proof}

\end{document}